\documentclass[11pt]{amsart}
\usepackage{cite}
\usepackage{geometry}
\usepackage{amsmath,amssymb,graphicx,blkarray,amsthm} 
\usepackage{mathtools,accents}
\usepackage{caption,comment}
\usepackage{subcaption}

\usepackage{bbm}
\usepackage{color}
\usepackage[table]{xcolor}
\usepackage[colorlinks,linkcolor=blue,citecolor=red,urlcolor=black,bookmarks=false,hypertexnames=true]{hyperref} 
\usepackage{float}

\usepackage{enumitem}
\usepackage{caption}
\usepackage{subcaption}

\usepackage{tikz}
\usetikzlibrary{calc,angles,positioning,quotes,decorations,automata}
\usepackage{tkz-euclide}

\usepackage{verbatim}
\usepackage{tikz-cd}

\usepackage{multirow} 

\newlength\celldim \newlength\fontheight \newlength\extraheight
\newcounter{sqcolumns}
\newcolumntype{S}{ @{}
>{\centering \rule[-0.5\extraheight]{0pt}{\fontheight + \extraheight}}
p{\celldim} @{} }
\newcolumntype{Z}{ @{} >{\centering} p{\celldim} @{} }
\newenvironment{squarecells}[1]
{\setlength\celldim{0.5cm}%
\settoheight\fontheight{A}%
\setlength\extraheight{\celldim - \fontheight}%
\setcounter{sqcolumns}{#1 - 1}%
\begin{tabular}{|S|*{\value{sqcolumns}}{Z|}}\hline}
{\end{tabular}}
\newcommand\nl{\tabularnewline\hline}
\newenvironment{squarecells2}[1]
{\setlength\celldim{0.75cm}%
\settoheight\fontheight{A}%
\setlength\extraheight{\celldim - \fontheight}%
\setcounter{sqcolumns}{#1 - 1}%
\begin{tabular}{|S|*{\value{sqcolumns}}{Z|}}\hline}
{\end{tabular}}

\newenvironment{squarecells3}[1]
{\setlength\celldim{0.6cm}%
\settoheight\fontheight{A}%
\setlength\extraheight{\celldim - \fontheight}%
\setcounter{sqcolumns}{#1 - 1}%
\begin{tabular}{|S|*{\value{sqcolumns}}{Z|}}\hline}
{\end{tabular}}

\theoremstyle{plain}
\newtheorem{theorem}{Theorem}
\newtheorem{proposition}[theorem]{Proposition}
\newtheorem{lemma}[theorem]{Lemma}
\newtheorem{corollary}[theorem]{Corollary}
\newtheorem{fact}[theorem]{Fact}

\theoremstyle{definition}
\newtheorem{definition}[theorem]{Definition}
\newtheorem{remark}[theorem]{Remark}

\newtheorem{example}[theorem]{Example}

\newcommand{\Z}{{\mathbb Z}}

\newcommand{\N}{{\mathbb N}}
\newcommand{\C}{{\mathbb C}}
\newcommand{\X}{{\mathbb X}}
\newcommand{\mc}{\mathcal}
\newcommand{\A}{{\mc A}}
\newcommand{\ii}{{\mathrm{i}}}
\newcommand{\dd}{{\mathrm{d}}}
\newcommand{\ee}{{\mathrm{e}}}

\newcommand{\supp}{{\mathrm{supp}}}
\newcommand{\shift}[1]{\mathbb{X}_{#1}}

\newcommand{\exend}{\hfill $\Diamond$}

\DeclareMathOperator{\id}{id}
\title{Radius-zero Extended Symmetries and Irregular Fibres of $\mathbb{Z}^d$-Substitution Subshifts}
\author{\'{A}lvaro Bustos-Gajardo, Daniel Luz and Neil Ma\~nibo }

\begin{document}

\begin{abstract}
    In this work, we consider $\Z^d$-shifts generated by digit substitutions. 
    For such a shift $\X$, we study the elements  of the normaliser of $\Z^d$ in the group of self homeomorphisms (called extended symmetries) whose local maps guaranteed by the generalised Curtis--Hedlund--Lyndon theorem have radius zero. Using the formalism of minimal sets developed by Lema\'{n}czyk, M\"ullner and Yassawi, we provide an algorithm to compute elements of $\mathcal{N}(\X)$ that preserve the hierarchical structure. We also investigate the interaction of extended symmetries with (i) the height lattice and (ii) the irregular fibres over the maximal equicontinuous factor. Towards (ii), we introduce the notion of derived substitutions to provide a complete description of the irregular fibres, extending a result by Coven, Quas and Yassawi in the one-dimensional case. 
\end{abstract}

\maketitle

\section{Introduction}

The automorphism (or symmetry) group is a well studied object in topological dynamics. For a topological dynamical system (TDS) $(\mathbb{X},G)$, it is the \emph{centraliser} of the group $G$ within the space $\text{Homeo}(\mathbb{X})$ of self-homeomorphisms of the space $\mathbb{X}$; see \cite{KS}. In symbolic dynamics, where one has a subshift $\mathbb{X}$ admitting a group action (typically by $\mathbb{Z}$ or $\mathbb{Z}^d$), there are lots of open questions surrounding them, despite a rather concrete representation of the group elements provided by the Curtis--Hedlund--Lyndon (CHL) theorem. There has been a steady progress on shifts of low-complexity (e.g. substitution shifts, $\mathcal{S}$-adic shifts, some Toeplitz shifts) \cite{CK,DDMP,PS,EM}, as well as shifts of number-theoretic origin ($\mathcal{B}$-free subshifts, algebraic shifts) \cite{Men, FY, DKK, BBHLN}, where results regarding the size of the automorphism group (relative to $G$) and closed forms are abundant. 
There are also realisation and embeddability results in the case when the shift has positive entropy and has a large centraliser, which is typical of shifts of finite type, sofic shifts, and random substitution subshifts) \cite{BLR,FRS,Sal}.

Another interesting group is the \emph{normaliser} of $G$ in $\text{Homeo}(\mathbb{X})$, which we denote by $\mathcal{N}(\mathbb{X})$; compare \cite{BRY,CorP}.  This group (of which the centraliser is a normal subgroup) is also invariant under topological conjugacy, up to isomorphism. This has various names for shift spaces, such as reversing symmetries (stemming from time-reversing symmetries in physics) for $d=1$ \cite{LRC,LR,BR} and extended symmetries/isomorphisms for higher-dimensional shifts \cite{BRY,CabP}. In this setting, there is a generalised version of the CHL theorem, see Section~\ref{sec:ext-sym}.

There are several advantages of looking at the normaliser. 
For  $\mathbb{Z}^d$-shifts, it is clear that there is more geometric freedom that is not seen by the centraliser. The normaliser sees the maps that are compatible with shifts with inherent geometric symmetry.  As examples, one has the $\mathbb{Z}^2$-subshift derived from the chair tiling and the Ledrappier shift defined by a modular condition on triangles. The normalisers for these shifts contain copies of $D_4$ and $D_3$, respectively, which are completely missed by the centraliser; see \cite{BRY} for details. In the more abstract setting, this also becomes a natural group to look at when the group $G$ acting on $X$ is no longer abelian, or even worse, when $G$ admits a trivial center \cite{CorP}. In the latter case, $G$ is not even in its own centraliser. For number-theoretic shifts such as $k$-free shifts over an algebraic number field $K$, under certain assumptions, the normaliser admits a closed form involving canonical groups, namely the automorphism group $\text{Aut}_{\mathbb{Q}}(K)$ of $K$ and the unit group $\mathcal{O}_{K}^{\times}$ of the relevant ring of integers; see \cite{BBHLN,BBN,BBN-2,GK} for developments along this direction.

In this paper, we continue the investigation of the normaliser of  higher-dimensional substitution shifts over a finite alphabet $\mathcal{A}$ initiated in \cite{B,BLM}. 
We restrict to shifts generated by \emph{digit substitutions}, which are higher-dimensional generalisations of constant-length substitutions in dimension one. 
In this work, we primarily deal with substitution subshifts which are neither bijective  nor Toeplitz. These are precisely those whose column number $c^{ }_{\theta}$ satisfies $1< c^{ }_{\theta}<|\mathcal{A}|$, where $\mathcal{A}$ is the alphabet; see Section~\ref{sec:subs} for definitions.

For constant-length substitutions in one dimension, a bound for the radius for elements of the automorphism group $\mathcal{C}(\X)$ is derived in \cite{CQY}. In \cite{MY}, it is shown that, up to topological conjugacy, the elements of $\mathcal{C}(\X)$ (modulo a shift) can be chosen to have either radius $0$ or $1$. \emph{Radius-zero symmetries} are letter exchange maps, that is,  permutations of elements of $\A$ that extend to a map on the language of $\X$. When $\X$ is a bijective substitution subshift, it was shown in \cite{B} that every element of $\mathcal{C}(\X)$ (up to a composition with a shift) is radius-zero; see also \cite{LMe}. This also holds for elements of $\mathcal{N}(\X)$ for the same class of shifts \cite{BLM}. These automorphisms also preserve the hierarchical structure, because they map $n$-supertiles to $n$-supertiles, for $n$ large enough.

In Section~\ref{sec:supertile-shuff}, we focus on the elements of the normaliser whose block map in the CHL representation has radius zero.
In this setting, there have already been interesting recent developments. In \cite{CabP}, the authors presented the first example of zero-entropy shifts in $\mathbb{Z}^d$ (which are, in fact, substitutive and Toeplitz) with $|\mathcal{N}(\X)/\mathcal{C}(\X)|=\infty$. In particular, they construct a planar example that admits an extended symmetry for every $A\in \text{GL}(2,\mathbb{Z})$, and generalisations thereof with an explicit form for $\mathcal{N}(\mathbb{X})$. All of these extended symmetries have radius zero, but some of them (necessarily) break the hierarchical structure. 

In the class of bijective substitutions, there is an algorithm to determine all non-trivial elements of $\mathcal{N}(\X)$, which are induced by some permutation $\tau\colon \mathcal{A}\to\mathcal{A}$ \cite{BLM}. 
In a related work, M\"ullner and Yassawi exploited the structure of \emph{minimal sets} to provide an equivalent condition for the existence of radius-zero automorphisms for constant-length substitutions. We combine these strategies to detect the presence of radius-zero extended symmetries that preserve the hierarchical structure. We call such elements of $\mathcal{N}(\X)$ \emph{supertile-shuffling}, inspired by the \emph{shuffle group} in \cite{FRS}. Our approach is also based on minimal sets and the group structure they induce, and the result can be seen as a generalisation of those in \cite{BLM,MY}. We prove the following result in Section~\ref{sec:supertile-shuff}.

\begin{theorem}\label{thm:main}
Let $\theta$ be an aperiodic primitive block $\Z^d$-substitution on $\mathcal{A}$, whose column number is $1<c_{\theta}<|\mathcal{A}|$. Let $\mathcal{N}(\X_{\theta})$ be the extended symmetry group of the corresponding shift space. Then there is an algorithm to compute the elements of $\mathcal{N}(\X_{\theta})$ that are supertile-shuffling. 
\end{theorem}

In Section~\ref{sec:height}, we look at the compatibility of extended symmetries with the \emph{height lattice} $\Gamma\leqslant \mathbb{Z}^d$ of $\mathbb{X}_{\theta}$. This lattice allows one to partition the alphabet $\mathcal{A}$ such that the shift map respects this partition. This also generates dynamical eigenvalues apart from those coming from the odometer factor \cite{Frank2,Bartlett}. We show that any extended symmetry must preserve $\Gamma$, hence also introducing constraints to allowable elements of $\mathcal{N}(\X)$. 

\begin{theorem}\label{thm:height}
Let $\theta$ be an aperiodic primitive digit substitution on $\mathbb{Z}^d$ (i.e., constant shape with expansive map $Q$) with height lattice $\Gamma$. If $\Phi\in \mathcal{N}(\mathbb{X})$ with linear component $A\in \textnormal{GL}(d,\mathbb{Z})$, then $A\Gamma=\Gamma$.
\end{theorem}

In Section~\ref{sec:kappa}, we present an example of a substitution which admits a non-trivial element of $\mathcal{N}(\mathbb{X})$ that is radius-zero but not supertile-shuffling. We use this as a springboard for the next section on irregular fibres.

Another feature of subshifts coming from digit substitutions is that they have a non-trivial \emph{maximal equicontinuous factor} (MEF), which is a rotation on a profinite abelian group $\mathbb{A}$. In the case of trivial height, 
this group $\mathbb{A}$ is a generalised odometer $\Z_Q$, and the factor map $\pi_{\rm MEF}$ can be derived directly from the supertile structure. With this tool, one can carry out further analysis of the elements of $\mathcal{C}(\mathbb{X})$ and  $\mathcal{N}(\mathbb{X})$. More concretely, there exists a group homomorphism $\kappa\colon \mathcal{C}(\mathbb{X})\to \mathbb{A}$ which is compatible with the factor map that allows one to view an automorphism as a translation on $\mathbb{A}$; see \cite{BRY,CabP}. 

Under mild assumptions, the kernel of this map is exactly the subgroup of radius-zero elements of $\mathcal{C}(\X)$; see \cite{MY}. For $\mathcal{N}(\X)$, one gets a $\text{GL}(d,\mathbb{Z})$-cocycle of this map, where an extended symmetry can be viewed as an affine-type transformation on $\mathbb{A}$. 
This map, together with the set of irregular fibres over $\mathbb{A}$, provides certain restrictions on admissible extended symmetries.
In Section~\ref{sec:irregular-fibres}, we generalise the result of Coven, Quas, and Yassawi in \cite{CQY} to provide a complete characterisation of the irregular fibres, for block substitutions. 

\begin{theorem}\label{thm:irreg-fibres}
    Let $\theta$ be an aperiodic primitive block substitution on $\mathbb{Z}^d$ with trivial height lattice (i.e., $\Gamma=\mathbb{Z}^d$). Then, there exists a sofic shift $\mathcal{Z}_{\theta}$, such that $\boldsymbol{z}\in \mathbb{A}=\mathbb{Z}^{d}_{Q}$ corresponds to an irregular fibre if, and only if, $\boldsymbol{z}+\boldsymbol{n}\in \mathcal{Z}_{\theta}$ for some $\boldsymbol{n}\in \mathbb{Z}^d$. 
\end{theorem}

The sofic shift $\mathcal{Z}_{\theta}$ can be constructed explicitly from the substitution $\theta$. We mention here that the explicit characterisation of irregular fibres in Theorem~\ref{thm:irreg-fibres}
might be of independent interest beyond questions surrounding extended symmetries. 
In particular, this finds possible applications in the study of tameness in topological dynamical systems \cite{FKY,KY} and multivariate notions of (mean) equicontinuity and sensitivity \cite{GL,GJY,BHJ}.

\section{Preliminaries}\label{sec:prelim}

\subsection{Shifts and substitutions}\label{sec:subs}

Let $\mathcal{A}$ be a finite \textit{alphabet}. The \textit{full shift} on $\mathcal{A}$ is the set $\mathcal{A}^{\mathbb{Z}^d}$ together with the shift action, that is a family of maps $\sigma^{ }_{\boldsymbol{n}}$, which act on an element $(x_{\boldsymbol{i}})_{\boldsymbol{i} \in \mathbb{Z}^d} \in \mathcal{A}^{\mathbb{Z}^d}$ via $\sigma^{ }_{\boldsymbol{n}}(x)_{\boldsymbol{i}}= x_{\boldsymbol{i}+\boldsymbol{n}}$.
In general, $\mathbb{Z}^d$ can be replaced with a more general group. 
A \textit{pattern} (\textit{patch}) $P$ is an element of $\mathcal{A}^B$, where is $B$ is a (finite) subset of  $\mathbb{Z}^d$. We refer to $B$ as the \textit{support} of $P$ and write $P_{\boldsymbol{b}}$ for the symbol at position $\boldsymbol{b}$.  We denote the set of finite patterns to be $\mathcal{A}^{*}$.
Most supports we are interested in will be hyperrectangular. We will denote the hypercube of sidelength $2r$ in $d$ dimensions centered around the origin by $[-r,r]^d$. 
The \textit{shift metric} $\dd$ is based on such hypercubes. Given $x,y\in \mathcal{A}^{\mathbb{Z}^d}$, $\dd$ is defined via 
\[
\dd(x,y):= \sup\big( \big\{ 2^{-r} \colon x|_{[-r,r]^d} \neq y|_{[-r,r]^d} \big\} \cup \{0\}\big).
\]
A \textit{cylinder set} $[P]$ of a patch $P \in \mathcal{A}^{B}$ is 
$
[P]= \{x \in \mathcal{A}^{\mathbb{Z}^d} \::\: x_{\boldsymbol{i}}=P_{\boldsymbol{i}} \; \text{ for all } {\boldsymbol{i}} \in B \}.
$ For patches of type $[-r,r]^d$, we see that the cylinder set consists of all elements that are at distance at most $2^{-r}$ from any fixed element of the cylinder (and each other) with respect to the shift metric.  

A \textit{subshift} $\mathbb{X}$ is a subset of $\mathcal{A}^{\mathbb{Z}^d}$ that is closed both with respect to the shift action and the shift metric. Equipped with the action of $\mathbb{Z}^d$ via the shift, $(\mathbb{X},\mathbb{Z}^d)$ is a topological dynamical system (TDS).  
Given two patterns $P,R \in \mathcal{A}^{*}$ or $\mathcal{A}^{\mathbb{Z}^d}$  we call $Z$ a subpattern of $P$ denoted by $R \sqsubset P$, if $P|_{B}=Z$, where $B$ is the support of $Z$. Here, we identify finite patterns up to translation. 
The \textit{language} $ \mathcal{L}_\mathbb{X}$ of the shift $\mathbb{X}$ is the collection of all its subpatterns, that is $\mathcal{L}_\mathbb{X}= \bigcup_{y \in \mathbb{X}} \bigcup_{B \Subset \mathbb{Z}^d} y|_{B}\subseteq \mathcal{A}^{*}$. We call these patterns \textit{legal} for $\mathbb{X}$. The language $\mathcal{L}_{\X}$ thas the following two properties:
\begin{itemize}
    \item it is \emph{extensible}, that is, for any $P\in\mathcal{L}_{\X}$ and any $U\supset\supp(P)$ there exists $Q\in\mathcal{A}^U\cap\mathcal{L}_{\X}$ for which $P\sqsubset Q$, and
    \item it is \emph{factorial}: if $P\sqsubset Q$ and $Q\in\mathcal{L}_{\X}$, then $P\in\mathcal{L}_{\X}$.
\end{itemize}
Any language with these two properties is necessarily the language of a uniquely determined shift space; the converse is also true.

The subshifts we investigate in this work are generated by \textit{digit substitutions}; compare \cite{FM}. These are also known as \emph{lattice substitutions} \cite{FS,LMS} and \emph{constant-shape substitutions} \cite{Cab,Cab2,CabP}, and are a generalisation of constant-length substitutions in one dimension. An expansive  endomorphism $Q\colon \mathbb{Z}^d\to \mathbb{Z}^d$  (that is $|\lambda|>1 \; \text{for all} \; \lambda  \in \textnormal{spec}(Q))$ gives rise to a non-trivial finite quotient $\mathbb{Z}^d/Q \mathbb{Z}^d$. Any collection of representatives from every coset gives rise to a (complete) \emph{digit set} $\mathcal{D}$; it naturally follows that $\mathbb{Z}^d= Q \mathbb{Z}^d + \mathcal{D}$. Such a digit system $(Q,\mathcal{D})$ gives a good baseline structure for a substitution since the shifted digit sets do not overlap and cover the whole space. We define $\mathcal
D^{(0)}=\left\{0\right\}$ and $\mathcal{D}^{(n)}:=Q\mathcal{D}^{(n-1)}+\mathcal{D}$. Every $\boldsymbol{n}\in \mathcal{D}^{(k)}$ admits the unique decomposition $\boldsymbol{n}=\sum_{i=0}^{k-1}Q^i(\boldsymbol{n}_i)$, where $\boldsymbol{n}_i\in \mathcal{D}$. The sequence $[\boldsymbol{n}_{k-1},\ldots,\boldsymbol{n}_0]$ is called the \emph{$Q$-adic decomposition} of $\boldsymbol{n}$. 

A \emph{digit substitution} is a map $  \theta=  \theta^{1}: \mathcal{A} \rightarrow \mathcal{A}^{\mathcal{D}}$. The image $\theta(a)$ is called a \emph{level-1 supertile} and using the previously mentioned coset structure we can iterate the substitution to define  \emph{$n$-supertiles} via: 
\[
    \theta^{n}(a)= \bigcup_{\boldsymbol{j} \in \supp (\theta^{n-1})} \theta(\theta^{n-1}(a)_{\boldsymbol{j}})+Q(\boldsymbol{j}). 
\]
where the above notation signifies that the $n$-supertile consists of level $1$-supertiles (obtained by the inflation of every symbol in the $(n-1)$-supertile) shifted to the correct position via $Q(\boldsymbol{j})$.
One can easily check that $\theta^{n}(a)$ is supported on $\mathcal{D}^{(n)}$.

To every digit substitution, we associate a language constructed as follows:
\[\mathcal{L}_\theta=\{P\in\mathcal{A^*} \mid (\exists a\in\mathcal{A},\,k\geqslant 1)\colon P\sqsubset\theta^k(a) \}.\]
This always results in a factorial language; thus, if the language is extensible, it defines a unique shift space, which we shall denote by $\X_\theta$. However, this language may fail to be extensible, since there might be too many gaps in the supertiles. Nevertheless, there are effectively checkable conditions on the digit set $\mathcal{D}$ that guarantee this; see \cite{Cab,Cab2} and \cite{Vince}. 
In this work, we will only deal with digit substitutions, and we will assume that they generate subshifts. For brevity, we will omit `digit' and just use `substitution' throughout.
When $\mathcal{D}$ is hyperrectangular, we call the corresponding substitution a \emph{block substitution}; compare \cite{Frank2,BG}.

\begin{figure}[!h]

\begin{subfigure}[b]{0.3\textwidth}
  \centering\raisebox{1.5cm}{\includegraphics[scale=0.6]{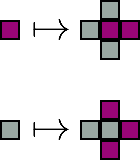}}  
  \caption{The substitution $\theta$.}
 \end{subfigure} \qquad 
\begin{subfigure}[b]{0.5\textwidth}
  \centering
  \includegraphics[scale=0.4]{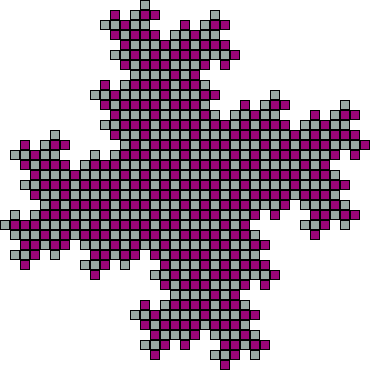}   
  \caption{$\theta^4(\textcolor{violet!80!black}{\blacksquare})$}
\end{subfigure}
\caption{The \emph{helix substitution}, an example of a two-dimensional digit substitution $\theta$ on two letters. Here, $\mathcal{D}=\left\{0,\boldsymbol{e}_1, -\boldsymbol{e}_1,\boldsymbol{e}_2,-\boldsymbol{e}_2\right\}$ and $Q=\begin{psmallmatrix}
    1 & {-}2\\
    2 & 1
\end{psmallmatrix}$.}
\end{figure}

  A substitution is called \textit{primitive} if there is a natural number $k$ such that, for every $a \in \mathcal{A}$, $\theta^k(a)$ contains every letter in the alphabet. The shift $\mathbb{X}_{\theta}$ defined by a primitive substitution is \textit{minimal}, that is, it equals the closure  of the $\mathbb{Z}^d$-orbit of any element $\omega\in\mathbb{X}_{\theta}$ (also called the \emph{hull} of $\omega$). Here, the closure is taken with respect to the local topology induced by the shift metric. 
If $\theta$ is primitive, for a large enough $k$, the level-$k$ supertile of an arbitrary starting letter contains every subpattern found in the inflation of the other letters. Thus, one may fix the letter $a$ in the definition of $\mathcal{L}_\theta$ (instead of running through all letters of $\mathcal{A})$, and obtain the same language. It immediately follows that, for a primitive $\theta$, $\mathcal{L}_\theta=\mathcal{L}_{\theta^k}$ and thus $\mathbb{X}_{\theta}=\mathbb{X}_{\theta^k}$.

 Given a substitution $\theta$, we can define a map $\theta_{\boldsymbol{j}}\colon \mathcal{A} \rightarrow \mathcal{A}$ via $\theta_{\boldsymbol{j}}(a):=\theta(a)_{\boldsymbol{j}}$, for every ${\boldsymbol{j}} \in \text{supp}(\theta)=\mathcal{D}$. The map $\theta_{\boldsymbol{j}}$ is called the \textit{$\boldsymbol{j}$-th column} of the substitution. 
If every column of $\theta$ is a bijection on $\mathcal{A}$,  $\theta$ is called \textit{bijective}.

An infinite configuration $\omega\in \mathcal{A}^{\mathbb{Z}^d}$ over a finite alphabet is (fully) \textit{aperiodic} when the hull generated by the shift action does not contain a periodic sequence, with respect to any sublattice of $\mathbb{Z}^d$. 
A primitive substitution $\theta$ is aperiodic if the unique hull that it generates only contains aperiodic elements. 
A configuration $x \in \mathbb{X}_\theta$ is \emph{recognizable}, if, for every $n\in\mathbb{N}$, there is a unique $y^{(n)} \in \mathbb{X}_{\theta}$ and $\boldsymbol{j}_n \in \text{supp}(\theta^n)=\mathcal{D}^{(n)}$ such that $x=\sigma^{ }_{\boldsymbol{j}_n}\circ \theta^n(y^{(n)})$. In our setting, assuming that $\theta$ generates a well-defined subshift, aperiodicity implies recognizability; see \cite{Sol}. 

\subsection{Extended symmetries}\label{sec:ext-sym}
Let $\mathbb{X}\subset \mathcal{A}^{\mathbb{Z}^d}$ be a subshift over a finite alphabet $\mathcal{A}$ and consider the topological dynamical system $(\mathbb{X},\mathbb{Z}^d)$. One can think of the action of $\mathbb{Z}^d$ to be one generated by $d$ commuting transformations $\left\{\mathcal{T}_i\colon 1\leq i\leq d\right\}$, which act faithfully on $\mathbb{X}$.

Let $\text{Homeo}(\mathbb{X})$ be the space of all self-homeomorphisms of $\mathbb{X}$. 
The group consisting of all $f\in\text{Homeo}(\mathbb{X}) $ that commute with $\sigma^{ }_{\boldsymbol{n}}$, for all $\boldsymbol{n}\in\mathbb{Z}^d$, (that is, $f\in \text{cent}_{\text{Homeo}(\mathbb{X})}(\mathbb{Z}^d)$) is called the \emph{symmetry group} $\mathcal{C}(\X)$ of $\mathbb{X}$ (also called the \emph{automorphism group} in the literature). 
Similarly, the group $\mathcal{N}(\mathbb{X}):=\text{norm}_{\text{Homeo}(\mathbb{X})}(\mathbb{Z}^d)$ is called the \textit{extended symmetry group} of $\mathbb{X}$. This group is also known as the \emph{isomorphism group} and the group of \emph{$\mathrm{GL}(d,\mathbb{Z})$-self  conjugacies.}
From the faithfulness of the action, one can obtain the following structural result about $\mathcal{N}(\mathbb{X})$. 

\begin{fact}[{\hspace{-0.02em}\cite[Prop.~2]{BRY}}]
Let $(\mathbb{X},\mathbb{Z}^d)$ be as above. 
\begin{enumerate}
\item There is a group homomorphism $\psi\colon \mathcal{N}(\mathbb{X})\to \textnormal{GL}(d,\mathbb{Z})$, whose kernel is $\mathcal{C}(\mathbb{X})$. 
\item If $\mathcal{N}(\mathbb{X})$ contains a subgroup $\mathcal{H}$ that satisfies $\mathcal{H}\simeq \psi(\mathcal{H})$ and $\mathcal{H}=\psi(\mathcal{N}(\mathbb{X}))$, one has the short exact sequence
$1\longrightarrow \mathcal{C}(\mathbb{X})\longrightarrow \mathcal{N}(\mathbb{X}) \longrightarrow  \mathcal{H}^{\prime}=\psi(\mathcal{H})\longrightarrow 1$ 
and one has that $\mathcal{N}(\mathbb{X})=\mathcal{C}(\mathbb{X})\rtimes \mathcal{H}$.\qed 
\end{enumerate}
\end{fact}

When the generators $\left\{\mathcal{T}_{i}\right\}$ of the action correspond to translations along the canonical basis vectors $\boldsymbol{e}_i$ of $\Z^d$, one has 
\[
\mathcal{N}(\mathbb{X})=\left\{\Phi\in \text{Homeo}(\mathbb{X})\colon \Phi\circ \sigma^{ }_{\boldsymbol{n}}=\sigma^{ }_{\boldsymbol{An}}\circ \Phi, \text{ for all }\boldsymbol{n}\in \mathbb{Z}^d, \text{ for some } A\in \text{GL}(d,\mathbb{Z}) \right\}.
\]
In general, one has the following version of the Curtis--Hedlund--Lyndon theorem for extended symmetries.

\begin{proposition}[\hspace{-0.02em}{\cite[Prop.~3]{BRY}}] Let $(\mathbb{X},\mathbb{Z}^d)$  and $\mathcal{N}(\mathbb{X})$ be the same as above. Any extended symmetry $\Phi\in \mathcal{N}(\mathbb{X})$ is of the form $\Phi=\tau\circ f_{A}$, with $\psi(\Phi)=A$, and where $\tau\colon f_{A}(\mathbb{X})\to \mathbb{X}$ is induced by a local map of finite radius and where $f_A$ defined by $\left(f_A(x)\right)_{\boldsymbol{n}}:=x_{A^{-1}\boldsymbol{n}}$ is the linear component. \qed
\end{proposition}

\subsection{Coincidence graph and minimal sets}\label{sec:minsets}

In this section, we recall some notation regarding the \emph{coincidence graph} of a substitution $\theta$; see
\cite{Rob, FS,CQY}. We also recall some notation and important properties of minimal sets for constant-length substitutions developed in \cite{LMu,MY}, which trivially extend to our higher-dimensional setting. 
The columns of a non-bijective  substitution $\theta$ no longer generate a permutation subgroup, but only a semigroup which we will denote by $S_\theta$.  For the definitions below, we require that
the substitution $\theta$ has trivial height. 
We defer to the formal definition of height 
in Section~\ref{sec:height}. Throughout the text, $\Sigma_{F}$ denotes the symmetric group over the finite set $F$. 

Let $\theta$ be a substitution on the alphabet $\mathcal{A}$ with trivial height and let $\boldsymbol{j}\in \text{supp}(\theta)$. The \emph{cardinality} of a column $\theta_{\boldsymbol{j}}$ is given by $|\theta_{\boldsymbol{j}}(\mathcal{A})|$.  This extends naturally to columns of powers of $\theta$. The \emph{column number} of $\theta$ is $
c_\theta:= \min_{k\in \mathbb{N}}\min_{\boldsymbol{j}\in \text{supp}(\theta^k)}|\theta^k_{\boldsymbol{j}}(\mathcal{A})|$. Any set of the form $\theta^{k}_{\boldsymbol{j}}(\mathcal{A})$ with $|\theta^{k}_{\boldsymbol{j}}(\mathcal{A})|=c_{\theta}$ is called a \emph{minimal set}. We denote by $\mathcal{X}$ the set of all minimal sets of $\theta$.
 The subsets of $\mathcal{A}$ that appear as columns are encoded in the coincidence graph $\mathcal{G}(\theta)$,  which is defined as follows.

\begin{definition}
    The \emph{coincidence graph}  of the substitution $\theta$ is the graph $\mathcal{G}(\theta)$ defined, recursively, as follows.
    \begin{enumerate}
      \item[\textbf{(1)}] $\mathcal{A} \in V (\mathcal{G} (\theta))$, and every vertex of
      $\mathcal{G} (\theta)$ is some subset of $\mathcal{A}$.
      \item[\textbf{(2)}] For every $\boldsymbol{j}\in \text{supp}(\theta)$ and every $U \in V (\mathcal{G}
      (\theta))$, there is an edge with label $\boldsymbol{j}$ going from $U$ to
      $\theta_{\boldsymbol{j}} (U)$, its image under the column
      $\theta_{\boldsymbol{j}}$.  Note that we allow loops and multiedges, that is, $U =
      \theta_{\boldsymbol{j}} (U)$ or $\theta_{\boldsymbol{j}} (U) =
      \theta_{\boldsymbol{j}'} (U)$ for $\boldsymbol{j} \neq \boldsymbol{j}'$ are possible. 
      \item[\textbf{(3)}] $\mathcal{G} (\theta)$ is the smallest graph satisfying \textbf{(1)} and \textbf{(2)}. 
    \end{enumerate}
\end{definition}

This graph was also called the \emph{forward substitution graph} in \cite{Rob}. Minimal sets are mapped to minimal sets under any column $\theta_{\boldsymbol{j}}$, and hence $\mathcal{X}$ forms a closed subgraph of $\mathcal{G}(\theta)$.

We now recall several results in on the compatibility relations of minimal sets with symmetries; compare \cite[Sec.~4]{MY}.

\begin{fact}\label{fact:minset-sym}
Let $\tau$ be the local map of a radius-$0$ symmetry.  One has the following properties. 
\begin{enumerate}[leftmargin=*]
    \item[\textnormal{(i)}] For every $M\in \mathcal{X}$, $\tau(M)=M$. Moreover,  for a sufficiently large $r$ the following holds. 
	\begin{equation}\label{eq:tau-theta-commute}
	    \tau\circ(\theta^r)_{\boldsymbol{n}} = (\theta^r)_{\boldsymbol{n}}\circ\tau.
	\end{equation} 
    \item[\textnormal{(ii)}] For an arbitrary $M_0\in\mathcal{X}$, there exists and idempotent column $\iota = (\theta^k)_{\boldsymbol{j}}$ with $\text{im}(\iota)=M_0$. In particular, $\iota(m)=m$, for all $m\in M_0$.
    \item[\textnormal{(iii)}] Any bijection $\nu_0\colon M_0\to[c_\theta]=\{1,2,\dotsc,c_\theta\}$ extends to the whole alphabet $\mathcal{A}$ via $\nu=\nu_0\circ \iota$. Moreover, there exists $\tau^{\prime}\in \Sigma_{[c_{\theta}]}$ such that 
    the following diagram commutes. 
    \begin{center}
		\begin{tikzcd}
			{[c_\theta]}\arrow[r,"\tau'"] & {[c_\theta]} \\
			M\arrow[r,"\tau"]\arrow[u,"\nu"] & M\arrow[u,swap,"\nu"]
		\end{tikzcd}
	\end{center}
    Moreover, $\tau'$ is independent of the minimal set $M$.\qed 
\end{enumerate}
\end{fact}

The commutativity relation in (iii) allows one to associate a family of permutations to each column $\theta_{\boldsymbol{n}}$, which is indexed by the minimal sets. 

\begin{fact}\label{fact:sym-beta}
    Let $\theta$, $\tau$ and $\nu$ be as in Fact~\ref{fact:minset-sym}. For every $M\in \mathcal{X}$, there is a  unique bijection $\beta_{M,\boldsymbol{n}}\colon[c_\theta]\to[c_\theta]$ for which the following diagram commutes:
	\begin{center}
		\begin{tikzcd}
			{[c_\theta]} & {[c_\theta]} \arrow[l,swap,dashed,"\beta_{M,\boldsymbol{n}}"] \\
			M\arrow[r,"\theta_{\boldsymbol{n}}"]\arrow[u,"\nu"] & \theta_{\boldsymbol{n}}[M]\arrow[u,swap,"\nu"]
		\end{tikzcd}
	\end{center}
    More generally, the existence of these bijections holds for any $f\in S_\theta$. Moreover, one has 
    \begin{equation*}
\beta_{M,f}^{-1}\circ\tau'=\tau'\circ \beta_{M,f}^{-1},\quad \text{ for all }f\in S_{\theta}. \qed
    \end{equation*} 
\end{fact}
These maps are introduced in \cite{MY} with the notation $\sigma^{ }_{M,j}$ via their inverses for the one-dimensional case. We prefer to use $\beta$ to avoid confusion with the usual notation for the shift map.
Note that the $\beta^{-1}_{M,f}$ is defined by 
\begin{equation}\label{eq:dfn_sigmas}
		\beta_{M,f}^{-1}\coloneqq \nu\circ f\circ (\nu\rvert_M)^{-1}.
	\end{equation}

We end this section with the following stronger version of (ii) in Fact~\ref{fact:minset-sym}, which we will be crucial in constructing analogues of the $\beta_{M,f}$ in Fact~\ref{fact:sym-beta} compatible with extended symmetries.
  
 \begin{lemma}\label{lem:power-idempotent}
   Let $\theta$ be a primitive substitution. Then, there exists a power $k\ge 1$ of $\theta$ such that every minimal set $M$ is realized as some column of $\theta^k$. Moreover, each of these columns acts as the identity on its corresponding minimal set. 
\end{lemma}

\begin{proof}
    As discussed in \cite{LMu} (see also \cite{MY}), to each substitution $\theta\colon\mathcal{A}\to\mathcal{A}^\mathcal{D}$ we may construct another substitution $\tilde{\theta}\colon\mathcal{X}\to\mathcal{X}^\mathcal{D}$ called the \emph{synchronising part} of $\theta$, where $\mathcal{X}$ is the collection of all minimal sets of the substitution $\theta$, given by:
        \[ \tilde{\theta}(M) = \bigcup_{\boldsymbol{j} \in \supp(\theta)} \theta_{\boldsymbol{j}}(M),\]
    where the right hand side is seen as a pattern in $\mathcal{X}^\mathcal{D}$.  Recall that any column maps a minimal set $M\in\mathcal{X}$ to another $M'\in\mathcal{X}$, making $\tilde{\theta}$ well-defined. It is not hard to verify, as well, that $\widetilde{\theta^k}=\tilde{\theta}^k$, thus we may replace $\theta$ by any situable power if so required.
    
    If $\theta$ is primitive, then $\tilde{\theta}$ is as well \cite{LMu}. Our result will follow from the same reasoning as in the proof of this statement given therein. If  $M=(\theta^k)_{\boldsymbol{j}}(\mathcal{A})$ is a minimal set, then for any other minimal set $M'$ we would have $(\theta^k)_{\boldsymbol{j}}(M')\subseteq (\theta^k)_{\boldsymbol{j}}(\mathcal{A})=M$, where this inclusion is forced to be an equality since  $\lvert(\theta^k)_{\boldsymbol{j}}(M')\rvert = c_\theta$ because of the minimality of $c_\theta$. Thus, $\tilde{\theta}$ must have a coincidence in its $k$-th power and, consequently, every power of $\tilde{\theta}^k$ has also a column where only $M$ appears. Iterating this argument with $\tilde{\theta}^k$ and the remaining minimal sets we eventually get a power $\tilde{\theta}^{k'}$ which has, for every minimal set $M'\in\mathcal{X}$, a column $(\tilde{\theta}^{k'})_{{\boldsymbol{j}}_{M'}}$ where only $M'$ appears.
    
    As the union of all minimal sets is $\mathcal{A}$, this implies that in the corresponding column $(\theta^{k'})_{{\boldsymbol{j}}_{M'}}$ only letters from $M'$ appear; in particular, by minimality, $(\theta^{k'})_{{\boldsymbol{j}}_{M'}}(M')=M'$ for all minimal sets $M'\in\mathcal{X}$. Since each $M'$ is finite and of cardinality $c_\theta$, it is then clear that $[(\theta^{k'})_{{\boldsymbol{j}}_{M'}}]^{c_\theta!}$ is the identity on $M'$. Thus, $\theta^{k'\cdot c_\theta!}$ is the desired power.
\end{proof}

\subsection{The MEF, the $\kappa$-map, and irregular fibres}\label{sec:mef-kappa}

Any TDS $(\mathbb{X},\mathbb{Z}^d)$ admits a \emph{maximal equicontinuous factor} (MEF), which is a rotation on a compact Abelian group $\mathbb{A}$, together with the corresponding factor map $\pi_{\rm MEF}\colon\X\to\mathbb{A}$. While for many kinds of systems this group $\mathbb{A}$ is trivial, in our setting it is guaranteed to be an infinite group, being a product of a generalised odometer $\Z_Q$ and a finite group. In particular, one may always define a \emph{tiling factor} $\pi_{\rm tile}\colon\X_\theta\to\Z_Q$, which is determined by the supertile structure of a point in $x$, in the sense that the first $k$ digits of $\pi_{\rm tile}(x)$ determine the position of the $k$-th order supertile in $x$; in the \emph{trivial height} case (discussed in Section~\ref{sec:height}), $\pi_{\rm tile}$ coincides with $\pi_{\rm MEF}$, while in the general case one may define $\pi_{\rm MEF}$ from $\pi_{\rm tile}$ with some additional information.

When the group $\mathbb{A}$ is nontrivial, any element $\alpha \in \mathcal{C}(\mathbb{X})$ induces a translation on $\mathbb{A}$ given by some $\kappa(\alpha)\in \mathbb{A}$. This induces a group homomorphism $\kappa\colon\mathcal{C}(\mathbb{X})\to\mathbb{A}$, which is  defined by the equation $
\pi(\alpha(x))=\kappa(\alpha)+\pi(x)$. This map extends to a $\mathrm{GL}(d,\mathbb{Z})$-cocycle for $\mathcal{N}(\mathbb{X})$ and is given in the following result.

\begin{theorem}[\hspace{-0.02em}{\cite[Thm.~5]{BRY}}]
Let $(\mathbb{X},\mathbb{Z}^d)$ be as above and assume additionally that it is transitive. Let $\mathbb{A}$ be its MEF and $\pi^{ }_{\textnormal{MEF}}\colon \mathbb{X}\to \mathbb{A}$ be the corresponding factor map, where the induced $\mathbb{Z}^d$-action has dense range in $\mathbb{A}$. 

If $\kappa(\mathbb{Z}^d)$ is a free Abelian group (that is, torsion-free, then there is a cocycle extension of $\kappa$ on $\mathcal{N}(\mathbb{X})$ defined by 
\[
\kappa(\Psi\circ\Phi)=\kappa(\Psi)+\zeta(\Psi)(\kappa(\Phi)),
\]
where $\zeta\colon \mathcal{N}(\mathbb{X})\to \textnormal{Aut}(\mathbb{A})$. The induced mapping on $\mathbb{A}$ is given by $z\mapsto \kappa(\Phi)+\zeta(\Phi)(z)$. \qed
\end{theorem}

In the setting of the previous theorem, we call
$\kappa(\Phi):=\pi(\Phi(x))-\zeta(R)(\Phi(x))\in \mathbb{A}$  the \emph{$\kappa$-value} of $\Phi$. 
For symmetries (i.e., whenever $A=\text{id}$), we have the following results regarding the $\kappa$ map for one-dimensional subshifts generated by constant-length substitutions. 
Let $\theta$ be a primitive length-$L$ substitution on $\mathcal{A}$ with trivial height and column number $c_{\theta}$. 
It is well known that $\mathbb{X}_{\theta}$ is almost $c$-to-$1$ over the MEF $\mathbb{A}$. 
The set of \emph{irregular fibres} over the MEF comprises of elements $\boldsymbol{z}\in\mathbb{A}$ such that $|\pi^{-1}_{\text{MEF}}(z)|>c_{\theta}$.
Let us denote by $\mathcal{Z}_{\theta}^{(m)}:=\left\{\boldsymbol{z}\in \mathbb{A}\colon |\pi^{-1}_{\text{MEF}}(\boldsymbol{z})|=m\right\}$. In one dimension, there is a graph-theoretic representation of the set $\mathcal{Z}_{\theta}=\bigcup_{m>c} \mathcal{Z}^{(m)}_{\theta} $ of irregular fibres \cite[Lem.~3.12]{CQY}, which is a sofic shift that can be completely derived from $\mathcal{G}{(\theta)}$; see also \cite{FS,Rob}) for higher-dimensional examples.

The following set of results gives some details on what is known for symmetries in the one-dimensional case; contrast Example~\ref{ex:kappa-nonzero} to see how the situation differs for the normaliser.

\begin{fact}\label{fact:rad-0-symm}
Let $\theta$ be a primitive length-$L$ substitution with height $1$ and column number $c$, and let $\mathbb{X}_{\theta}$ be the subshift it defines. 
\begin{enumerate}
\item For any automorphism $\alpha\in \mathcal{C}(\mathbb{X}_{\theta})$, $\kappa(\alpha)\in \mathbb{Q}$.  {\textnormal{\cite[Prop.~3.24]{CQY}}}.
\item For $m> c$,  and any automorphism $\alpha $ one has $\kappa(\alpha)+\mathcal{Z}_{\theta}^{(m)}\subseteq \mathcal{Z}_{\theta}^{(m)}$. That is, $\kappa(\alpha)$ sends irregular fibres to irregular fibres. 
\item If $\alpha$ is an automorphism of radius $0$, then $\kappa(\alpha)=0$. The converse true if $\theta$ is strongly injective {\textnormal{\cite[Rem.~28]{MY}}}.
\item An $\alpha$ is an automorphism of radius $0$ if and only if it is supertile-shuffling, i.e., it induces a permutation of level-$n$ supertiles $\left\{\theta^n(a)\colon a\in \mathcal{A}\right\}$, for $n$ large enough {\textnormal{\cite[Prop.~27]{MY}}}. 
\end{enumerate}
\qed
\end{fact}

Let $\theta$ be a primitive substitution in $\mathbb{Z}^d$, with trivial height lattice and column number $c_{\theta}$, and let $\mathbb{X}_{\theta}$ be the subshift it defines. The following proposition relates the $\kappa$-cocyle and irregular fibres. 

\begin{proposition}[\hspace{-0.02em}{\cite[Cor.~3]{BRY}}]\label{prop:ext-sym-kappa}
For $m> c_{\theta}$,  and any $\Phi\in\mathcal{N}(\mathbb{X}_{\theta})$, one has the inclusion $\kappa(\Phi)+\zeta(\Phi)\left(\mathcal{Z}^{(m)}_{\theta}\right)\subseteq \mathcal{Z}^{(m)}_{\theta}$. That is, $\kappa(\Phi)$ sends irregular fibres to irregular fibres along directions dictated by $\zeta(\Phi)$ (in fact, this holds for more general subshifts). \qed
\end{proposition}

A more thorough discussion on the MEF (including a closed form for $\mathbb{A}$) in our setting is carried out in Section~\ref{sec:height}.

\section{Supertile-shuffling extended symmetries}\label{sec:supertile-shuff}

In this section, we present an algebraic characterisation of the existence of non-trivial radius-$0$ extended symmetries that are supertile-shuffling. 
We use an approach based on \cite{MY} for symmetries, which exploits the structure of minimal sets for constant-length substitutions, and which generalises the results in \cite{BLM} for bijective substitutions. 

\subsection{Supertile-shuffling}
Suppose the map $A\in \text{GL}(d,\mathbb{Z})$ preserves the block $B$ in $\mathbb{Z}^d$, that is $A[B]=\{A\boldsymbol{n}\::\:\boldsymbol{n}\in B\}$ is a translate of $B$.
Let $\boldsymbol{k}:=\boldsymbol{k}(B)$ the vector that translates $B$ and situates its center at the origin. 
The map $A\odot\boldsymbol{n}:\boldsymbol{n}\mapsto A\boldsymbol{n} +A\boldsymbol{k}-\boldsymbol{k}$ is a bijection from $B$ to itself (and, indeed, a group action of the set of matrices associated to block-preserving maps on $B$). From now on, we will use the notation $A\odot\boldsymbol{n}$ for this operation, and assume without further comment that rotations, reflections, etc., map the block $B$ to itself.

The following result generalises \cite[Prop.~3.21]{CQY} on radius-$0$ symmetries for constant length substitutions. 

\begin{proposition}\label{prop:supertile-shuffling}
Let $\theta$ be an aperiodic and primitive block substitution on a finite alphabet $\mathcal{A}$ in $\mathbb{Z}^d$. 
Suppose further that $A\in \textnormal{GL}(d,\mathbb{Z})$ satisfies $A^{-1}\odot\textnormal{supp}(\theta^n)=\textnormal{supp}(\theta^n)$, for all $n$. 
If a letter exchange map $\tau\in \Sigma_{\mathcal{A}}$ together with $A$ satisfies that, for some $N\in\N$ and for all $a\in\mathcal{A},\boldsymbol{i}\in\supp(\theta^n)$ and $n\ge N$,
\begin{equation}\label{eq:supertile-shuff}
\theta^n_{\boldsymbol{i}}(\tau(a))=\tau(\theta^n_{A^{-1}\odot \boldsymbol{i}}(a)),
\end{equation}
then $(\tau,A)$ defines a radius-$0$ extended symmetry.
\end{proposition}

\begin{proof}
Let $\Phi:=\tau\circ f_{A}$, where $f_{A}(x)_{\boldsymbol{n}}=x_{A^{-1}\boldsymbol{n}}$. We show that $\Phi$ is an extended symmetry, i.e., $\Phi\circ\sigma^{ }_{\boldsymbol{j}}=\sigma^{ }_{A\boldsymbol{j}}\circ \Phi$, for all  $\boldsymbol{j}\in\mathbb{Z}^d$ and $\Phi\in\text{Homeo} (\X_{\theta})$ . 

The first property follows easily from the definition of $f_{A}$ and the fact that $\tau$ is just a letter-exchange map.
More concretely, one has 
\[
(f_{A}\circ \sigma^{ }_{\boldsymbol{j}}(x))_{\boldsymbol{n}}=(\sigma^{ }_{\boldsymbol{j}}(x))_{A^{-1}\boldsymbol{n}}=x_{A^{-1}\boldsymbol{n}-\boldsymbol{j}}=(f_A(x))_{\boldsymbol{n}-A\boldsymbol{j}}=(\sigma^{ }_{A\boldsymbol{j}}\circ f_A(x))_{\boldsymbol{n}}
\]
for all $\boldsymbol{j},\boldsymbol{n}\in \mathbb{Z}^d$ and $A\in \text{GL}(d,\mathbb{Z})$.  Since any letter exchange map commutes with the shift, this immediately implies the claim.

We now show that $\Phi\in \text{Homeo}(\X_{\theta})$, that is, $\Phi(x)\in \X_{\theta}$ and $\Phi$ is continuous and invertible. We first show that, for $x\in \X_{\theta}$ and for  $n\geqslant N$
\begin{equation}\label{eq:Phi-commute-sub}
\Phi\circ \theta^n=\sigma^{ }_{\boldsymbol{\ell}}\circ \theta^n\circ \Phi,
\end{equation}
for some $\boldsymbol{\ell}:=\boldsymbol{\ell}(x)\in\mathbb{Z}^d$, $n\geqslant N$.

Let $x\in \X_{\theta}$, with $x_0=a$ say. Up to a shift $\sigma^{ }_{\boldsymbol{c}}$, one sees a (centred) $n$-supertile around the origin for $\Phi\circ \sigma^{ }_{\boldsymbol{c}}\circ\theta^n(x)$. 
On the other hand, one can first shift $x$ before applying $\Phi$ so that the $n$-supertile around the origin containing $x_0$ is centred. One then successively applies $\Phi$ and $\theta^n$, resulting in an element that has a level $2n$-supertile
around the origin, which contains the $n$-supertile $\theta^n(\tau(a))$ by construction. 
Up to some shifts $\sigma^{ }_{\boldsymbol{k}},\sigma^{ }_{\boldsymbol{s}}$ (depending only on the supertile structure of $x$), we then have 
\[
(\sigma^{ }_{\boldsymbol{s}}\circ \Phi\circ \theta^n(x))_{\boldsymbol{i}}=(\sigma^{ }_{\boldsymbol{k}}\circ\theta^n\circ \Phi(x))_{\boldsymbol{i}}
\]
for $\boldsymbol{i}\in \text{supp}(\theta^n)$ whenever Eq.~\eqref{eq:supertile-shuff} holds. 
Note that this holds for all translates $\sigma^{ }_{\boldsymbol{m}}(x)$, whence minimality and compactness of $\mathbb{X}_{\theta}$ imply that $\Phi\circ \theta^n=\sigma^{ }_{\boldsymbol{\ell}}\circ\theta^n\circ \Phi$.

\begin{figure}
    \centering
    \includegraphics[width=\linewidth]{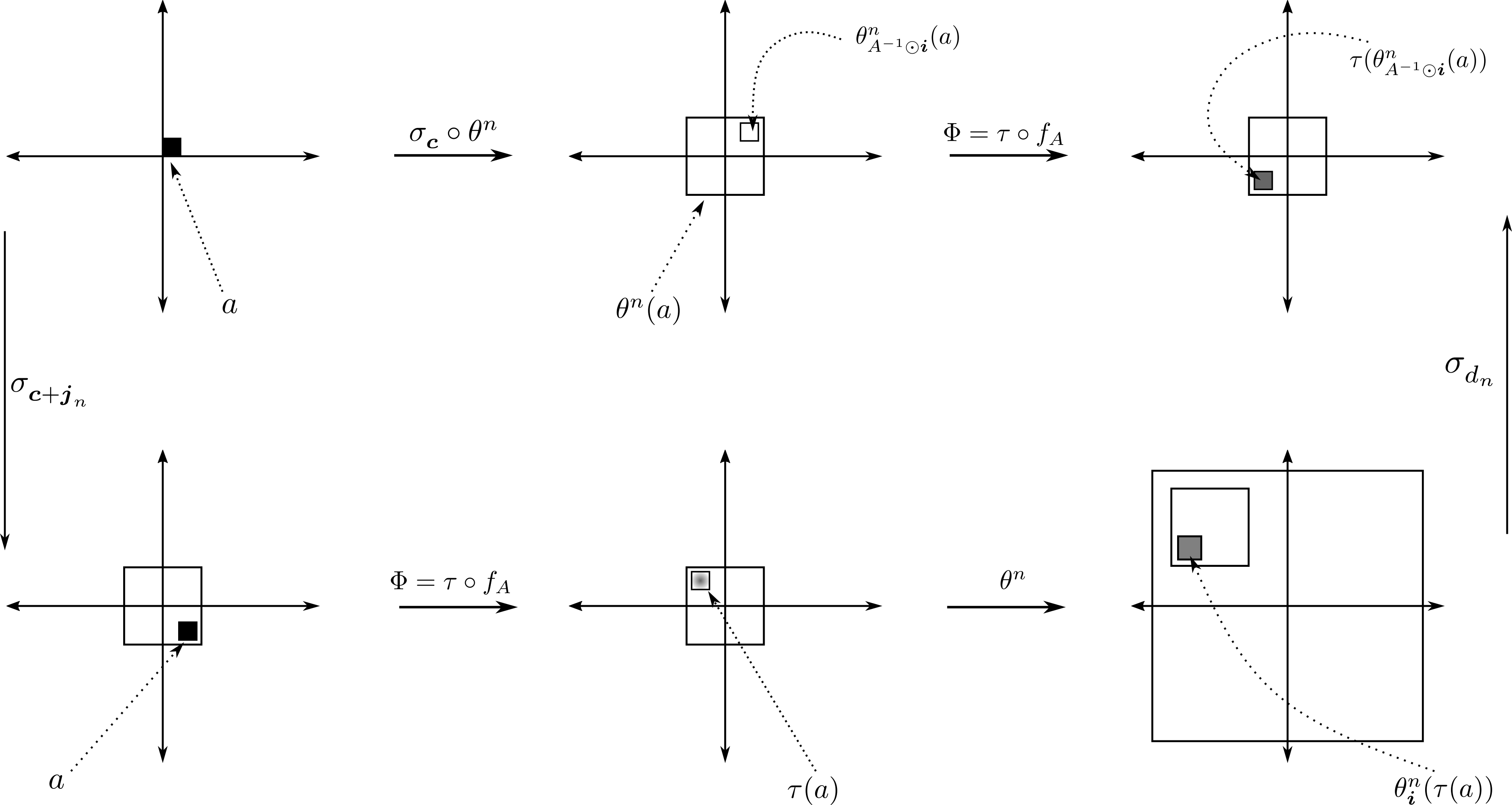}
    \caption{Illustration of the proof that Eq.~\eqref{eq:Phi-commute-sub} holds under the assumption on the supertiles}
    \label{fig:enter-label}
\end{figure}

Without loss of generality, we let $\theta:=\theta^N$. 
Now let $x\in \X_{\theta}$. For $m\in\N$, let $y^{(m)}\in \X_{\theta}$ and $\boldsymbol{j}_m\in\mathbb{Z}^d$ such that $x=\sigma^{ }_{\boldsymbol{j}_m}\circ \theta^{m}(y^{(m)})$. We have the identity $\theta^m\circ \sigma^{ }_{\boldsymbol{j}}=\sigma^{ }_{Q^m\boldsymbol{j}}\circ \theta^m$. Iterating Eq.~\eqref{eq:Phi-commute-sub} yields $\Phi\circ \theta^{m}=\sigma^{ }_{\boldsymbol{\ell}+Q\boldsymbol{\ell}+\cdots+Q^{m-1}\boldsymbol{\ell}}\circ\theta^m\circ\Phi$. For a fixed $m\in  N$, one then has 
\begin{align*}
\Phi(x)&=\Phi\circ \sigma^{ }_{\boldsymbol{j}_m}\circ \theta^{m}(y^{(m)})=\sigma^{ }_{A\boldsymbol{j}_{m}}\circ\Phi \circ \theta^{m}(y^{(m)})\\
&=\sigma^{ }_{A\boldsymbol{j}_{m}}\circ\sigma^{ }_{\boldsymbol{\ell}+Q\boldsymbol{\ell}+\cdots+Q^{m-1}\boldsymbol{\ell}}\circ\theta^n\circ\Phi(y^{(m)})
\end{align*}
Now let $w$ be a $2\times2\times \cdots\times 2$ block in $y^{(n)}$.
Since $w$ is legal and $\theta$ is primitive, it appears in some supertile $\theta^n(a)$, for large enough $n$ and for some $a\in \mathcal{A}$. It follows from Eq.~\eqref{eq:supertile-shuff} that 
$\tau \circ A(w)$ is also legal block, where $A(w)_{\boldsymbol{n}}=w_{A^{-1}\odot\,\boldsymbol{n}}$. This means $\theta^{n}(\tau \circ A(w))$ is also legal. Since $n$ can be chosen to be arbitrarily large, this means any block in $\Phi(x)$ of arbitrary size belongs to $\mathcal{L}_{\theta}$, and hence $\Phi(x)\in \X_{\theta}$. It is clear that $\Phi$ is continuous and  invertible with $\Phi^{-1}:=\tau^{-1}\circ f_{A^{-1}}$, which makes it a self-homeomorphism of $\X_{\theta}$, which completes the proof.
\end{proof}

In the setting of the previous result, we call the element $\Phi\in \mathcal{N}(\X)$ defined by the pair $(\tau,A)$ a \emph{supertile-shuffling} extended symmetry, and we call $\tau$ its corresponding local map.

\subsection{Minimal sets and extended symmetries}\label{sec:minsetextsym}

In this section, we generalise the construction of $\tau'$ and $\beta_{M,n}$ in Section~\ref{sec:minsets} in the setting of supertile-shuffling extended symmetries.

Let $\Phi$ be a supertile-shuffling extended symmetry with local map $\tau\colon\mathcal{A}\to\mathcal{A}$ and matrix $A$. It follows by definition that, for large enough power $r$,
\begin{equation}\label{eq:general_relabel_commutativity}
		\tau\circ(\theta^r)_{\boldsymbol{n}} = (\theta^r)_{A^{-1}\odot \boldsymbol{n}}\circ\tau,
	\end{equation}
	where the equality is stated in terms of columns. Without loss of generality, we shall assume that $r=1$ in what follows. The column
    $\theta_{A^{-1}\odot\boldsymbol{n}}$ necessarily has as image the minimal set $\tau[M]$ whenever $(\theta^r)_{\boldsymbol{n}}(\mathcal{A})=M$ . In particular, by replacing $\theta$ with a power if necessary, we see that $\tau$ maps minimal sets to minimal sets, and it induces a permutation $\tau_*$ of the collection $\mathcal{X}$ of all minimal sets, where if $M=\theta_{\boldsymbol{n}}[\mathcal{A}]$, then $\tau_*(M)=\tau[M]$ is the minimal set $\theta_{A^{-1}\odot\boldsymbol{n}}[\mathcal{A}]$. This is a consistency condition between $\tau$ and the columns of $\theta$.

Unlike the case for symmetries, we notice that  we \emph{do} have a dependency on the reference minimal set $M$ chosen, since $\tau[M]$ is not necessarily equal to $M$. This implies $\beta_{M,f}\ne\beta_{\tau(M),f}$ in general. Hence, the relation between the elements $f\in S_\theta$ and the corresponding $\beta$s is no longer a semigroup morphism, but rather a cocycle-type relation, as the commutative diagram below shows:
	\begin{center}
		\begin{tikzcd}
			{[c_\theta]} & {[c_\theta]} \arrow[l,swap,"\beta_{M,f}"] & {[c_\theta]}\arrow[l,swap,"\beta_{f[M],g}"] \arrow[ll,bend right,dashed,swap,"\beta_{M,g\circ f}"] \\
			M\arrow[r,"f"]\arrow[u,"\nu"]\arrow[rr,bend right,swap,"g\circ f"] & f[M]\arrow[r,"g"]\arrow[u,swap,"\nu"] & g[f[M]]\arrow[u,swap,"\nu"]
		\end{tikzcd}
	\end{center}
	which translates to the equality:
	\begin{equation}\label{eq:sigma_cocycle}
		\beta_{M,g\circ f}^{-1} = \beta_{f[M],g}^{-1}\circ\beta_{M,f}^{-1}.
	\end{equation}
      Thus $\tau'$, as defined previously, explicitly depends on the reference minimal set $M_0$ taken (as in, the permutation is not consistent along different minimal sets).
	
	To alleviate this, we shall redefine $\tau'$ in terms of \emph{two} encoding maps (related to $A$), chosen to ensure consistency. From \eqref{eq:general_relabel_commutativity}, we can see that if the column $\theta_{\boldsymbol{n}}$ is idempotent, then $\theta_{A^{-1}\odot\boldsymbol{n}}$ must be idempotent as well:
		\[\theta_{A^{-1}\odot\boldsymbol{n}}^2 = (\tau^{-1}\circ\theta_{\boldsymbol{n}}\circ\tau)^2 = \tau^{-1}\circ\theta_{\boldsymbol{n}}^2\circ\tau= \tau^{-1}\circ\theta_{\boldsymbol{n}}\circ\tau =\theta_{A^{-1}\odot\boldsymbol{n}}. \]
	We notice both columns have the same cardinality, and thus if they have cardinality $c_\theta$ they can be used to define encodings. Furthermore, by the discussion above, if the image of the first column is $M_0$, the second necessarily has image $\tau[M_0]$.
	
	Choose $\iota = \theta_{\boldsymbol{n}}$ as an appropriate column (as before) and $\bar{\iota} = \theta_{A^{-1}\odot \boldsymbol{n}}$; we then define $\nu\coloneqq\nu_0\circ\iota$ as before, and $\bar{\nu}\coloneqq (\nu\rvert_{\tau[M_0]})\circ\bar{\iota}$. The choice to use $\nu\rvert_{\tau[M_0]}$ as a bijection $\tau[M_0]\to[c_\theta]$ is entirely arbitrary; it ensures that $\nu\rvert_{\tau[M_0]} = \bar{\nu}\rvert_{\tau[M_0]}$, but this is not really necessary. We are now in position to prove the following result, which is a generalised version of Lemma~36 in \cite{MY}.
	\begin{lemma}\label{lem:tau-prime}
		Let $\tau$ be the local map of a supertile-shuffling extended symmetry with associated matrix $A$, and suppose $\nu$ and $\bar{\nu}$ are defined as in the above paragraph. Given $a,b\in\mathcal{A}$, if $\nu(a)=\nu(b)$, then $\bar{\nu}(\tau(a))=\bar{\nu}(\tau(b))$. Hence, there exists a unique permutation $\tau'\colon[c_\theta]\to[c_\theta]$ such that the following holds:
		\begin{equation}\label{eq:generalised_tau_prime}
			\tau'\circ\nu = \bar{\nu}\circ\tau.
		\end{equation}
	\end{lemma}
	\begin{proof}
		This is a direct computation in the same vein as the aforementioned Lemma~36. Note that, as $\nu_0$ is a bijection, $\nu(a)=\nu(b)$ if, and only if, $\iota(a)=\iota(b)$; as $\iota=\theta_{\boldsymbol{n}}$ and $\bar{\iota}=\theta_{A^{-1}\odot\boldsymbol{n}}$, by  Eq.~\eqref{eq:general_relabel_commutativity} we see that $\tau\circ\iota=\bar{\iota}\circ\tau$. Thus:
		\begin{align*}
			\bar{\nu}(\tau(a)) &= (\nu\rvert_{\tau[M_0]})\circ\bar{\iota}\circ\tau(a)=(\nu\rvert_{\tau[M_0]})\circ\tau\circ\iota(a)\\
			&=(\nu\rvert_{\tau[M_0]})\circ\tau\circ\iota(b)=(\nu\rvert_{\tau[M_0]})\circ\bar{\iota}\circ\tau(b)=\bar{\nu}(\tau(b)).
		\end{align*}
		The following claim holds by taking $\tau'=\bar{\nu}\circ\tau\circ(\nu\rvert_M)^{-1}$, where $M$ can be any minimal set.
	\end{proof}

		We stress that the definition of $\tau'$ above is independent of the chosen $M\in\mathcal{X}$, as it depends only on the encodings $\nu$ and $\bar{\nu}$; any choice of $M$ yields the same permutation.
	Now, we can relate these permutations to the columns in the same way as in the $A=\text{id}$ case. However, as we are dealing with two encodings at once, we also have a \emph{pair} of $\beta$-maps, given by the equations:
		\[\beta_{M,f}^{-1}\coloneqq \nu\circ f\circ (\nu\rvert_M)^{-1},\quad\text{and}\quad\bar{\beta}_{M,f}^{-1}\coloneqq \bar{\nu}\circ f\circ (\bar{\nu}\rvert_M)^{-1}.\]
	It turns out that the appropriate generalisation of $\tau'$ relates one set of $\beta$-maps to the other, as seen in the commutative cube below.
	
	\begin{center}
		\begin{tikzcd}
		& {[c_\theta]}\arrow[rrr,"\tau'"] & & & {[c_\theta]} \\
		f[M]\arrow[ur,"\nu"] & & & \tau[f[M]]\arrow[ur,"\bar{\nu}"] \\
		& {[c_\theta]}\arrow[rrr,"\tau'"]\arrow[uu, near start,"\beta^{-1}_{M,f}"] & & & {[c_\theta]}\arrow[uu,"\bar{\beta}^{-1}_{\tau[M],\bar{f}}",swap] \\
		M\arrow[rrr,"\tau",crossing over]\arrow[uu,"f"]\arrow[ur,"\nu"] & & & \tau[M]\arrow[uu,swap,near start,"\bar{f}",crossing over]\arrow[ur,swap,"\bar{\nu}"]
		\arrow[from=2-1,to=2-4,"\tau",crossing over]
		\end{tikzcd}
	\end{center}
	Here, $f$ and $\bar{f}$ correspond to columns at position $\boldsymbol{n}$ and $A^{-1}\odot\boldsymbol{n}$, respectively, where one has $\tau[f[M]]=\bar{f}[\tau[M]]$ as a consequence of Eq.~\eqref{eq:general_relabel_commutativity}. Note that the left and right side faces of the cube correspond to the definition of $\beta_{M,f}$ and $\bar{\beta}_{M,f}$, respectively, the front face of the cube is Eq.~\eqref{eq:general_relabel_commutativity}, and the top and bottom faces of the cube correspond to the definition of $\tau'$, employing the independence of this definition with regards to the chosen minimal set. From all these equalities, we derive the proof of the identity corresponding to the back face of the cube, as seen in the following result. 
    
	\begin{proposition}\label{prop:beta-commute}
		Let $\tau$ be the local function of a supertile-shuffling extended symmetry with associated matrix $A$, and suppose $\nu$ and $\bar{\nu}$ are a pair of associated encodings. Let $f=\theta_{\boldsymbol{n}}$ and $\bar{f}=\theta_{A^{-1}\odot\boldsymbol{n}}$. Then, the following equality holds:
		\[\tau'\circ\beta_{M,f}^{-1} = \bar{\beta}_{\tau[M],\bar{f}}^{-1}\circ\tau',\quad\text{or, in terms of columns,}\quad\tau'\circ\beta_{M,\boldsymbol{n}}^{-1} = \bar{\beta}_{\tau[M],A^{-1}\odot\boldsymbol{n}}^{-1}\circ\tau'. \]
	\end{proposition}
	\begin{proof}
		By direct computation, one has
		\begin{align*}
			\bar{\beta}_{\tau[M],\bar{f}}^{-1}\circ\tau' &= (\bar{\nu}\circ \bar{f}\circ \underbrace{(\bar{\nu}\rvert_{\tau[M]})^{-1})\circ(\bar{\nu}}_{\id_{\tau[M]}}\circ \underbrace{\tau\circ (\nu\rvert_M)^{-1}}_{\text{image is }\tau[M]}) = \bar{\nu}\circ \bar{f}\circ\tau\circ(\nu\rvert_M)^{-1} \\ &= \bar{\nu}\circ \underbrace{\tau\circ f}_{\text{by \eqref{eq:general_relabel_commutativity}}}\circ(\nu\rvert_M)^{-1} = \bar{\nu}\circ \tau\circ \underbrace{((\nu\rvert_{f[M]})^{-1}\circ\nu)}_{\id_{f[M]}}\circ  \underbrace{f\circ(\nu\rvert_M)^{-1}}_{\text{image is } f[M]} \,\,\, =\tau'\circ\beta_{M,f}^{-1}. \qedhere
		\end{align*}
	\end{proof}

\subsection{Main result}
 We are now poised to prove Theorem~\ref{thm:main}. The following result can be seen as a generalisation of \cite[Thm.~37]{MY} (to the setting of extended symmetries) and \cite[Thm.~3.13]{BLM} (in the case of non-bijective substitutions). In particular, conditions \textbf{(1)} and \textbf{(2)} are vacuously satisfied when $\theta$ is bijective. 
Here, we assume that $\theta$ is a primitive, aperiodic substitution, with $c_{\theta}>1$. 

 \begin{theorem}\label{thm:supertile-shuffling} Let $\tau\in \Sigma_{\mathcal{A}}$ and $A\in \mathrm{GL}(d,\mathbb{Z})$. 
The pair $(\tau,A)$ 
gives rise to a supertile-shuffling extended symmetry of $(\X_{\theta},\mathbb{Z}^d)$ if and only if it satisfies the following conditions.
\begin{enumerate}
    \item[\textnormal{\textbf{(1)}}] $\theta_{\boldsymbol{j}}(\tau[M])=\tau[\theta_{A^{-1}\odot\boldsymbol{j}}(M)]$, for all $\boldsymbol{j}\in \textnormal{supp}(\theta)$
    \item[\textnormal{\textbf{(2)}}] $\tau'$ is well defined via $\tau^{\prime}\circ \nu=\bar{\nu}\circ \tau$.
    \item[\textnormal{\textbf{(3)}}] $\tau'\circ \beta^{-1}_{M,\boldsymbol{j}}\circ (\tau^{\prime})^{-1}=\bar{\beta}^{-1}_{\tau[M],{A^{-1}\odot \boldsymbol{j}}}$, for every minimal set $M$ and for all $\boldsymbol{j}\in \textnormal{supp}(\theta)$
    \end{enumerate}  
\end{theorem}

\begin{proof}

Assume $(\tau,A)$ gives rise to a supertile-shuffling extended symmetry. By definition, 
$\theta_{\boldsymbol{j}}(\tau(a))=\tau(\theta_{A^{-1}\odot \boldsymbol{j}}(a))$, which holds for all  $a\in\mathcal{A}$ and $\boldsymbol{j}\in \text{supp}(\theta)$. Note that property \textbf{(1)} follows from this immediately. 
Properties \textbf{(2)} and \textbf{(3)} follow from Lemma~\ref{lem:tau-prime} and Proposition~\ref{prop:beta-commute}, respectively.

For the other direction, we first show that these three conditions imply $\theta_{\boldsymbol{j}}(\tau(a))=\tau(\theta_{A^{-1}\odot\boldsymbol{j}}(a))$. Suppose that $\tau(a)\in M$ and let $\widetilde{M}=\theta_{\boldsymbol{j}}[M]$. 
\begin{align*}
 \theta_{\boldsymbol{j}}(\tau(a))&=\nu\big|^{-1}_{\widetilde{M}}\circ \nu (\theta_{\boldsymbol{j}}(\tau(a))) 
 \tag{$\theta_{\boldsymbol{j}}(\tau(a))\in \widetilde{M}$}\\
&=\nu\big|^{-1}_{\widetilde{M}}\circ \beta^{-1}_{M,\boldsymbol{j}}\circ\nu (\tau(a)) \tag{\text{Def. of} $\beta^{-1}_{M,\boldsymbol{j}}$}\\
&=\nu\big|^{-1}_{\widetilde{M}}\circ \beta^{-1}_{M,\boldsymbol{j}}\circ\tau' \circ \bar{\nu}(a) \tag{Property \textbf{(2)}}\\
&=\nu\big|^{-1}_{\widetilde{M}}\circ \tau^{\prime} \circ \bar{\beta}^{-1}_{\tau^{-1}[M],A^{-1}\odot\boldsymbol{j}} \circ \bar{\nu}(a)\tag{Property \textbf{(3)}}\\
&=\nu\big|^{-1}_{\widetilde{M}}\circ \tau^{\prime} \circ \bar{\nu}\circ \theta_{A^{-1}\odot\boldsymbol{j}} \circ \bar{\nu}\big|^{-1}_{\tau^{-1}[M]} \circ \bar{\nu}(a)
\tag{\text{Def. of} $\beta^{-1}_{\tau^{-1}[M],A^{-1}\odot\boldsymbol{j}}$}\\
&=\nu\big|^{-1}_{\widetilde{M}} \circ \nu\circ \tau \circ \theta_{A^{-1}\odot\boldsymbol{j}}(a)
\tag{$a\in \tau^{-1}[M]$}\\
&=\tau(\theta_{A^{-1}\odot\boldsymbol{j}}(a))
\tag{$\tau(\theta_{A^{-1}\odot\boldsymbol{j}}(a))\in \widetilde{M}$ by \textbf{(1)}}
\end{align*}

We now show that if the conditions in Theorem~\ref{thm:supertile-shuffling} are satisfied, then they are satisfied for all powers $\theta^k$ of $\theta$. Note that by Lemma~\ref{lem:tau-prime}, $\tau^{\prime}$ does not depend on the power $k$, so \textbf{(2)} extends trivially. 
It remains to prove the claim for \textbf{(1)} and \textbf{(3)}. 

Suppose \textbf{(1)} holds. Let $k\in \mathbb{N}$ be fixed. Recall that, for $\boldsymbol{j}\in \text{supp}(\theta^k)$, the columns $\theta_{\boldsymbol{j}}$ and $\theta_{A^{-1}\odot \boldsymbol{j}}$ admit the expansions
\begin{align*}
    \theta^{ }_{\boldsymbol{j}}&=\theta^{ }_{\boldsymbol{j}_{0}}\circ \theta^{ }_{\boldsymbol{j}_{1}}\circ \cdots \circ \theta^{ }_{\boldsymbol{j}_{k-1}}\\
     \theta^{ }_{A^{-1}\odot\boldsymbol{j}}&=\theta^{ }_{A^{-1}\odot\boldsymbol{j}_{0}}\circ \theta^{ }_{A^{-1}\odot\boldsymbol{j}_{1}}\circ \cdots \circ \theta^{ }_{A^{-1}\odot\boldsymbol{j}_{k-1}}
\end{align*}
where $[\boldsymbol{j}_{k-1},\ldots,\boldsymbol{j}_0]$ is the $Q$-adic expansion of $\boldsymbol{j}$; see Section~\ref{sec:subs}.
We then have
\begin{align*}
    \tau\circ \theta^{ }_{\boldsymbol{j}}(M)&=\tau \circ \theta^{ }_{\boldsymbol{j}_{0}}\circ \theta^{ }_{\boldsymbol{j}_{1}}\circ \cdots \circ \theta^{ }_{\boldsymbol{j}_{k-1}}(M)\\
    &=\theta^{ }_{A^{-1}\odot\boldsymbol{j}_0} \circ \tau \circ \theta^{ }_{\boldsymbol{j}_{1}}\circ \cdots \circ \theta^{ }_{\boldsymbol{j}_{k-1}}(M) \\
    &=\theta^{ }_{A^{-1}\odot\boldsymbol{j}_{0}}\circ \theta^{ }_{A^{-1}\odot\boldsymbol{j}_{1}}\circ \cdots \circ \theta^{ }_{A^{-1}\odot\boldsymbol{j}_{k-1}}\circ \tau (M)\\
    &=\theta^{ }_{A^{-1}\odot \boldsymbol{j}}(\tau[M]).
\end{align*}
Here, the second equality holds because $\theta^{ }_{\boldsymbol{j}_{1}}\circ \cdots \circ \theta^{ }_{\boldsymbol{j}_{k-1}} (M)=M^{\prime}$ is a minimal set and \textbf{(1)} holds for $M^{\prime}$ as well. This proves the claim for \textbf{(1)}. 

Next we show that an analogous statement holds for property \textbf{(3)}. Recall from Eq.~\eqref{eq:sigma_cocycle} that we have the cocycle relation
\[
\beta_{M,g\circ f}^{-1} = \beta_{f[M],g}^{-1}\circ\beta_{M,f}^{-1}.
\]
for any $M$ and $f,g\in S_{\theta}$. 

Let $\boldsymbol{j}\in \text{supp}(\theta^k)$, with $Q$-adic expansion $\boldsymbol{j}=[\boldsymbol{j}_{k-1},\ldots,\boldsymbol{j}_0]$.
We then have
\begin{equation}\label{eq:beta-levelk}
\beta^{-1}_{M,\boldsymbol{j}}=
\beta^{-1}_{\theta_{\boldsymbol{j}_1}\circ\cdots  \circ\theta_{\boldsymbol{j}_{k-1}}[M],\boldsymbol{n_0}}\circ \cdots \circ \beta^{-1}_{\theta_{\boldsymbol{j}_{k-1}}[M],\boldsymbol{j}_{k-2}}\circ 
\beta^{-1}_{M,\boldsymbol{j}_{k-1}}
\end{equation}
Suppose \textbf{(3)} holds for all $\boldsymbol{j}_j\in \text{supp}(\theta)$. 
We then have
\begin{align*}
\tau'\circ \beta^{-1}_{M,\boldsymbol{j}}&=
\tau'\circ \beta^{-1}_{\theta_{\boldsymbol{j}_1}\circ\cdots  \circ\theta_{\boldsymbol{j}_{k-1}}[M],\boldsymbol{n_0}}\circ \cdots \circ \beta^{-1}_{\theta_{\boldsymbol{j}_{k-1}}[M],\boldsymbol{j}_{k-2}}\circ 
\beta^{-1}_{M,\boldsymbol{j}_{k-1}}\\
&=\bar{\beta}^{-1}_{\tau\circ\theta_{\boldsymbol{j}_1}
\circ \cdots  \circ\,\theta_{\boldsymbol{j}_{k-1}}[M],A^{-1}\odot\boldsymbol{n_0}} \circ \tau' \circ \cdots \circ \beta^{-1}_{\theta_{\boldsymbol{j}_{k-1}}[M],\boldsymbol{j}_{k-2}}\circ 
\beta^{-1}_{M,\boldsymbol{j}_{k-1}}\\
&=\bar{\beta}^{-1}_{\theta_{A^{-1}\odot\boldsymbol{j}_1}
\circ \cdots  \circ\,\theta_{A^{-1}\odot\boldsymbol{j}_{k-1}}(\tau[M]),A^{-1}\odot\boldsymbol{n_0}} \circ \tau' \circ \cdots \circ \beta^{-1}_{\theta_{\boldsymbol{j}_{k-1}}[M],\boldsymbol{j}_{k-2}}\circ 
\beta^{-1}_{M,\boldsymbol{j}_{k-1}}
\end{align*}

For the third equality, we have used \textbf{(1)}, i.e., $\tau(\theta_{\boldsymbol{j}_j}[M])=\theta_{A^{-1}\odot \boldsymbol{j}_j}(\tau[M])$ to simplify the corresponding minimal set. 
Iterating this process  and invoking Eq.~\eqref{eq:beta-levelk} yields
\[
\tau'\circ \beta^{-1}_{M,\boldsymbol{j}}=
\bar{\beta}^{-1}_{\tau[M],A^{-1}\odot\boldsymbol{j}}\circ \tau'. \qedhere
\]
\end{proof}

\begin{remark}
    Given a candidate $(\tau,A)$ for a supertile-shuffling extended symmetry, it is then necessary that the minimal sets of the orbit of every point in the support under $A$ are consistent with the letter exchange map; compare condition \textbf{(1)} in Theorem~\ref{thm:supertile-shuffling}. Then, simultaneous identities on minimal sets along a single orbit can be achieved by taking the least common multiple of the powers determined in Lemma~\ref{lem:power-idempotent}, for each element in the orbit. 
\end{remark}

We demonstrate the result of the previous theorem with an example.

\begin{example}
    Consider the following one-dimensional substitution on $\mathcal{A}=\{a,b,c\}$:
    \begin{equation}\label{eq:subs-rev}
    \theta\colon \begin{cases}
        a\mapsto aacbaa &\\
        b\mapsto bcaacc &\\
        c\mapsto bbaabc
    \end{cases}.
    \end{equation}

We show that $\tau=(bc)\in \Sigma_{\mathcal{A}}$ generates a supertile-shuffling reversor. This example has $c_{\theta}=2$. The collection $\mathcal{X}$ of minimal sets is given by $\mathcal{X}=\left\{M_1,M_2\right\}$, with $M_1=\{a,b\},\,M_2=\{a,c\}$. 
    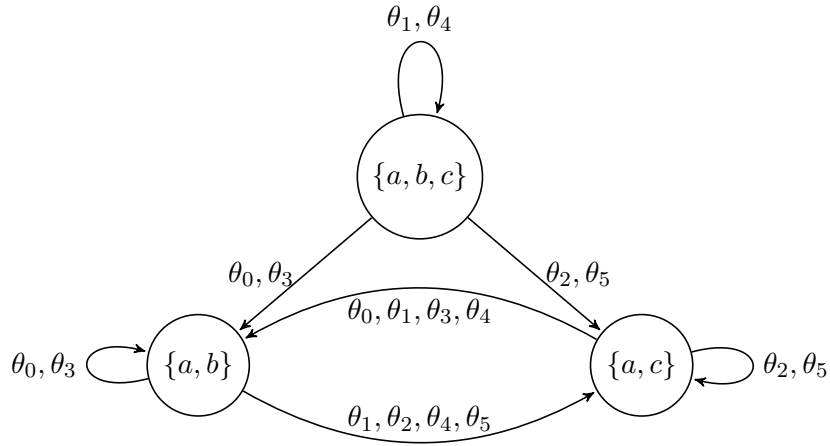
\begin{figure}[!h]
   \begin{center}
  \begin{tikzpicture}[->,>=stealth',shorten >=1pt,auto,node distance=2.5cm,
                    semithick]
  \tikzstyle{every state}=[fill=none,text=black]

    \node[state] (AB)                    {$\{a,b,c\}$};
    \node[state] (A) [below of=AB, left =1.4cm of AB]       {$\{a,b\}$};
    \node[state] (B) [below of=AB, right =1.4cm of AB]       {$\{a,c\}$};

  \path (AB) edge  [left, pos=0.5]    node {$\theta_0,\theta_3$} (A)
        (AB) edge  [right, pos=0.5]    node {$\theta_2,\theta_5$} (B)
        (AB) edge  [loop above]    node {$\theta_1,\theta_4$} (AB)
        
        (A) edge  [bend right =90, out=330, in=210]    node {$\theta_1,\theta_2,\theta_4, \theta_5$} (B)
        (A) edge  [loop left]    node {$\theta_0,\theta_3$} (A)

        (B) edge  [bend right =90, out=330, in=210]    node {$\theta_0,\theta_1,\theta_3,\theta_4$} (A)
        (B) edge  [loop right]    node {$\theta_2,\theta_5$} (B)
        ;
\end{tikzpicture}
\end{center}
        \caption{The coincidence graph for the substitution in Eq.~\eqref{eq:subs-rev}}
        \label{fig:min-sets}
    \end{figure}

Note that both $M_1$ and $M_2$ appear as idempotent columns of $\theta$ (as the first and last columns, respectively), and this pair can be used to define the respective encodings $\nu,\bar{\nu}\colon \mathcal{A}\to [c_\theta]$. 
More explicitly, these are given by
\[
\nu\colon \begin{cases}
    a\mapsto 0 &\\
    b\mapsto 1&\\
    c\mapsto 1&
\end{cases}
\quad \text{and} \quad \quad  
\bar{\nu}\colon \begin{cases}
    a\mapsto 0 &\\
    b\mapsto 1&\\
    c\mapsto 1.&
\end{cases}
\]
 In other words, these two idempotents  define the same encoding.

First, note that the set of length-2 legal words $\mathcal{L}_2$ contains all length-2 words in $\mathcal{A}^{+}$, hence the condition $(\tau\circ A)\cdot(\mathcal{L}_2)\subseteq \mathcal{L}_2$ is trivially satisfied.
Next, we check whether the conditions of Theorem~\ref{thm:supertile-shuffling} are satisfied. Since the idempotent columns we used to define the encodings $\nu$ and 
$\bar{\nu}$ already appear in $\theta$, it suffices to check the three conditions in the theorem for this power of the substitution.

For condition \textbf{(1)}, we need to check $\theta_{\boldsymbol{j}}(\tau[M])=\tau[\theta_{A\odot\boldsymbol{j}}(M)]$ for all $\boldsymbol{j}\in \left\{0,\ldots,5\right\}$. Here, we get two simplifications. Since, $\tau[M_1]=M_2$ and vice versa, it suffices to check whether $\theta_{\boldsymbol{j}}(M)=\tau[\theta_{A\odot\boldsymbol{j}}(\overline{M})]$, where $\overline{M_i}=M_{(12)i}$. 
Moreover, $A^{-1}\odot \boldsymbol{j}=L-\boldsymbol{j}-1$, since we are in dimension one, which yields $\theta_{\boldsymbol{j}}(M)=\tau[\theta_{L-\boldsymbol{j}-1}(\overline{M})]$. The table of images of minimal sets under the columns is given below. 
\end{example}

\begin{table}[!h]
    \centering
    \begin{tabular}{c|c|c|c|c|c|c}
        & 0 & 1 & 2 &3 &4 & 5\\
         \hline
         $M_1$ & $M_1$ & $M_2$ &$M_2$ & $M_1$ & $M_2$ & \textcolor{red}{$M_2$}  \\
         \hline 
        $M_2$ & \textcolor{red}{$M_1$} & $M_1$ &$M_2$ & $M_1$ & $M_1$ & $M_2$ \\
         \hline 
    \end{tabular}
    \caption{Images of minimal sets under the columns of $\theta$ in Eq.~\eqref{eq:subs-rev}}
    \label{tab:minsets}
\end{table}

One can check that condition \textbf{(1)} is satisfied for all columns and all minimal sets. As an example, if we pick $\boldsymbol{j}=0$ for $M_2$, this yields 
\[
\theta_0(M_2)=M_1=\tau[M_2]=\tau[\theta_{5}(M_1)]; 
\]
see the highlighted entries in the table above. 

Condition \textbf{(2)} is easily seen to be satisfied, with $\tau^{\prime}$ being the identity map on $[c_\theta]=\{0,1\}$. This simplifies Condition \textbf{(3)} into $
\beta^{ }_{M,\boldsymbol{j}}=\beta^{ }_{\overline{M},L-\boldsymbol{j}-1}$.
The table of the permutations $\beta_{M,\boldsymbol{j}}\in \Sigma_{[c_\theta]}$ is given below. 

\begin{table}[!h]
    \centering
    \begin{tabular}{c|c|c|c|c|c|c}
        & 0 & 1 & 2 &3 &4 & 5\\
         \hline
         $M_1$ & id & id & $(01)$ & $(01)$ & id & id  \\
         \hline 
        $M_2$  & id & id & $(01)$ & $(01)$ & id & id\\
         \hline 
    \end{tabular}
    \caption{Collection of $\beta^{ }_{M,\boldsymbol{j}}$ for the substitution $\theta$ in Eq.~\eqref{eq:subs-rev}}
    \label{tab:beta}
\end{table}

For this example, $\beta^{ }_{M,\boldsymbol{j}}$ is independent of $M$, whence it suffices to check one row to confirm the validity of Condition \textbf{(3)}, which is visually apparent from the table. Note that this is not true in general, in particular, when $\nu\neq \bar{\nu}$. It then follows from Theorem~\ref{thm:supertile-shuffling} that the mirroring map $m(x)_{n}:=x_{-1-n}$, together with the letter exchange map $\tau=(bc)$, generates a reversor for $(\X_{\theta},\mathbb{Z})$.

\section{The non-trivial height case}\label{sec:height}

\subsection{Height and substitutions}
Let $\theta\colon\mathcal{A}\to\mathcal{A}^\ell$ be a primitive substitution and $u\in \shift{\theta}$ be one of its $\theta$-fixed (or periodic) points; in its most standard form, the height $h:=h(\theta)$ of $\theta$ is defined as:
\begin{equation}\label{eq:1D_height}
    h = \gcd\{k\in\mathbb{N} \::\: u_k = u_0 \text{ and } \gcd(k,\ell) = 1\}.
\end{equation}
The chosen $u\in \shift{\theta}$, and whether it is a fixed point or not, turns out to be irrelevant because of minimality. Note that $h$ is always coprime to $\ell$, by definition. Other known characterisations of height in the aperiodic case are as follows; see \cite{Dek}.

\begin{fact}\label{fact:height-1d}
    Let $\theta$ be a primitive, aperiodic length-$\ell$ substitution, and let $h\ge 1$ be an integer coprime to $\ell$. The following are equivalent:
    \begin{enumerate}[label=\textup{(\roman*)},leftmargin=*]
        \item $h$ is the height of $\shift{\theta}$,
        \item $h$ is the largest integer coprime to $\ell$ such that the alphabet $\mathcal{A}$ partitions into $h$ non-empty sets $\{\mathcal{A}_j\}_{j\in\mathbb{Z}/h\mathbb{Z}}$ so that, if $uv\in\mathcal{L}_2(\shift{\theta})$ is legal and       $u\in\mathcal{A}_j$, then $v\in\mathcal{A}_{j+1}$,
        \item $h$ is the largest integer coprime to $\ell$ such that $\ee^{2\pi \ii/h}$ is a \emph{continuous eigenvalue} of $\shift{\theta}$, that is, there is a continuous non-zero function $f\colon\shift{\theta}\to\mathbb{C}$ such that $f\circ\sigma =\ee^{2\pi \ii/h}\cdot f $,
        \item the \textnormal{MEF} of $\shift{\theta}$ is the group rotation $(\mathbb{Z}_\ell \times \mathbb{Z}/h\mathbb{Z},+(1,1))$.\qed 
    \end{enumerate}
\end{fact}

We follow the definition  of Frank for the height of a substitutive $\mathbb{Z}^d$-shift, which is done in several steps; see \cite{Frank2, Bartlett, Cab2}.
    Let $\Gamma_1,\Gamma_2\le\mathbb{Z}^2$ be two sublattices of $\mathbb{Z}^d$. We say $\Gamma_1$ and $\Gamma_2$ are \emph{coprime} if $\Gamma_1+\Gamma_2=\{\boldsymbol{m}_1+\boldsymbol{m}_2 \::\:\boldsymbol{m}_j\in\Gamma_j\}$ equals the whole of $\mathbb{Z}^d$.
   Given a minimal $\mathbb{Z}^d$-subshift $\X$ and a fixed $x\in \X$, we define the set of \emph{return times} of $\X$ as 
  \[
    \Xi = \lbrace \boldsymbol{j} \in \mathbb{Z}^d\colon \exists\, \boldsymbol{k} \text{ such that } x_{\boldsymbol{j}+\boldsymbol{k}}=x_{\boldsymbol{k}}\rbrace.    
  \]
   The \emph{return module} of $\X$, $\mathcal{L}\coloneqq\langle\Xi\rangle$, is the lattice generated by $\Xi$, that is, the set of all $\Z$-linear combinations of elements of $\Xi$.

For a minimal shift space (e.g., for one generated by a primitive substitution), the definition of the set $\Xi$, and consequently of the return module $\mathcal{L}$, are independent of the point $x$ chosen as reference. The same applies to the height lattice, which directly generalises the notion of one-dimensional height, and is defined as follows.

\begin{definition}\label{def:height_lattice}
    Given an aperiodic, primitive  substitution $\theta\colon\mathcal{A}\to\mathcal{A}^\mathcal{D}$ in $\mathbb{Z}^d$ with  associated inflation matrix $Q$, the \emph{height lattice} of $\shift{\theta}$ is the smallest lattice $\Gamma\geq \mathbb{Z}^d$ which contains the return module $\mathcal{L}$ and is coprime with the supertile lattice $Q\Z^d$.
\end{definition}

The height lattice can be characterised via an appropriate partition of the alphabet. This is the generalisation of (ii) in Fact~\ref{fact:height-1d}.

\begin{fact}\label{fact:alph-partition}
    Let $\X\subseteq\mathcal{A}^{\Z^d}$ be a minimal subshift and $\Gamma\le\Z^d$ be the associated height lattice. If $\mathcal{D}_{\Gamma}$ is a fundamental domain for $\Gamma$, then there exists a partition $(\mathcal{A}_{\boldsymbol{k}})_{\boldsymbol{k}\in \mathcal{D}_{\Gamma}}$ of the alphabet of $\X$ into non-empty sets indexed by $\mathcal{D}_{\Gamma}$ so that, for any $x\in \X$, there exists $\boldsymbol{k}_0\in\Z^d$ such that $x_{\boldsymbol{j}}\in\mathcal{A}_{\boldsymbol{k}}$ if, and only if, $\boldsymbol{j}+\boldsymbol{k}_0\equiv\boldsymbol{k}\pmod{\Gamma}$. \qed
\end{fact}

\subsection{Height and the MEF}

As we have seen in Proposition~\ref{fact:height-1d}, the height lattice is closely related to the continuous eigenfunctions of the shift space. Recall that a \emph{continuous eigenfunction} is a non-zero continuous map $f\colon \mathbb{X}\to\mathbb{C}$ satisfying that, for every $\boldsymbol{n}\in\Z^d$, $f\circ\sigma^{ }_{\boldsymbol{n}}$ is a multiple of $f$. For a minimal system, one may always assume that $f$ takes values in the unit circle $\mathbb{S}^1=\{z\in\C\::\:\lvert z\rvert=1\}$. We write $\mathcal{E}(\mathbb{X})$ for the set of all eigenfunctions of the subshift $\mathbb{X}$; this is a group under ordinary multiplication.

For the following discussion, we shall need some notions from harmonic analysis; we omit most proofs and refer the reader to any textbook on the subject \cite{Loo,Rud}. We have the following known results on Pontryagin duals of locally compact Abelian (LCA) groups.

\begin{fact}\label{fact:Pontryagin-duality}
    Let $G,H$ be LCA groups and $\widehat{G},\widehat{H}$ their Pontryagin duals. We have:
    \begin{enumerate}[label=\textup{(\roman*)},leftmargin=*]
        \item The dual $\widehat{\widehat{G}}$ of $\widehat{G}$ is isomorphic to $G$, where the isomorphism sends every $g\in G$ to the map $e_g\colon\widehat{G}\to\mathbb{S}^1$ given by $e_g(\chi)=\chi(g)$.
        \item Any continuous group homomorphism $f\colon G\to H$ induces a dual group homomorphism $f^*\colon\widehat{H}\to\widehat{G}$ given by $f^*(\chi)=\chi\circ f$. This association is functorial, in the sense that $(f\circ g)^*=g^*\circ f^*$ and $\id_G^*=\id_{\widehat{G}}$ (in particular, for an isomorphism $f$, the equality $(f^*)^{-1}=(f^{-1})^*$ holds), and furthermore, if we identify $G$ with $\widehat{\widehat{G}}$, then $f^{**}=f$.
        \item   Let $(G_i,f_i)_{i\in I}$ be an inverse system of LCA groups and continuous group homomorphisms, such that $G=\varprojlim (G_i,f_i)_{i\in I}$ is also a LCA group. Then the dual of $G$ is the direct limit of the duals of the $\widehat{G}_i$, i.e. $\widehat{G}=\varinjlim(\widehat{G}_i,f^*_i)$.
        \item The dual of a Cartesian product of two LCA groups is the product of their duals, where the isomorphism $\widehat{G}_1\times\widehat{G}_2\cong \widehat{G_1\times G_2}$ maps a pair of characters $(\chi^{ }_1,\chi^{ }_2)$ to the character $\xi((g_1,g_2))=\chi^{ }_1(g_1)\chi^{ }_2(g_2)$ of $G_1\times G_2$.\qed
    \end{enumerate}
\end{fact}

The results above are relevant since we deal with generalised odometers, which are inverse limits of (finite) groups, as the maximal equicontinuous factors of our shift space $\X$. Their duals are generalised Prüfer groups, which dictate the eigenfunction structure of the subshift $\X$ in question as follows.

\begin{fact}\label{fact:eigenfunctions_as_characters}
    Let $\X$ be a shift space whose maximal equicontinuous factor is a surjective map $\pi^{ }_{\rm MEF}\colon \X\twoheadrightarrow\mathbb{A}$ onto a LCA group $\mathbb{A}$. For any constant $C\in\mathbb{C}\setminus\{0\}$ and any character $\chi\colon\mathbb{A}\to\mathbb{S}^1$, the function $C\cdot\chi\circ\pi^{ }_{\rm MEF}$ is an eigenfunction of $\X$. Conversely, all eigenfunctions of $\X$ are of this form. \qed
\end{fact}

\begin{fact}
    Let $\X$ be a minimal shift space and $\sim$ be the equivalence relation on $\mathcal{E}(\X)$ defined by
        $f\sim g\iff(\exists\, C\in\mathbb{C}\setminus\{0\})\colon f=C\cdot g$.
    Then $\widehat{\mathbb{A}}\cong \mathcal{E}(\X)/{\sim}$, where the isomorphism maps the character $\chi$ to the equivalence class of $\chi\circ\pi^{ }_{\rm MEF}$. \qed
\end{fact}

Hence, we shall ignore any multiplicative constants and assume that all eigenfunctions are of the form $\chi\circ\pi^{ }_{\rm MEF}$ from now on. The interested reader may consult \cite{K} for a proof in the one-dimensional case; the generalisation to $\Z^d$-actions is immediate.  

\subsection{MEF for substitutions and height eigenfunctions}

Given a substitution $\theta$ with associated (expansive) inflation matrix $Q$, the MEF of $\shift{\theta}$ is a finite index extension of the $Q$-adic odometer obtained as the inverse limit $\Z^d_Q:=\varprojlim (\Z^d/Q^k\Z^d,Q)$. More precisely, if $\Gamma$ is the height lattice of $\shift{\theta}$, we have
\begin{equation}\label{eq:MEF-genform}
    \mathbb{A}\cong(\mathbb{Z}^d_Q\times(\Z^d/\Gamma)), 
\end{equation}
with the canonical actions of $\mathbb{Z}^d$ on both components.
This is the decomposition of  $\mathbb{A}$ as a product of its torsion-free and torsion component, respectively. For $x\in \X_{\theta}$, one can then write 
\[
\pi^{ }_{\textrm{MEF}}(x)=(\pi^{ }_{\textrm{tile}}(x),\pi^{ }_{\text{height}}(x)),\]
with $\pi^{ }_{\text{tile}}(x)\in \Z_Q^d$ and 
$\pi^{ }_{\text{height}}(x)\in \Z^d/\Gamma$.
This decomposition is central in the proof of Theorem~\ref{thm:height}.

Fact~\ref{fact:Pontryagin-duality} then tells us how every eigenfunction looks like, up to multiplication by a constant: as $\widehat{\mathbb{A}}$ is isomorphic to the product of the direct limit of the groups $\Z^d/Q^k\Z^d$ with the finite group $\Z^d/\Gamma$, any character of $\mathbb{A}$ decomposes as a product of a character of $\Z_Q^d$ (which may be interpreted as a group homomorphism $\Z_Q^d\to\mathbb{S}^1$ where the value of $f(\boldsymbol{z})$ depends only on finitely many digits of $\boldsymbol{z}$) and a character associated to the height. This translates to the following.

\begin{theorem}\label{thm:eigenfunction-decomp}
    Every eigenfunction $f\colon \X\to\mathbb{S}^1$ is a product of two eigenfunctions $f_{\rm tile}$ and $h$, where the former, for some $k$, satisfies that $f_{\rm tile}\circ\sigma^{ }_{\boldsymbol{n}}=f_{\rm tile}$ for all $\boldsymbol{n}\in Q^k\Z^d$, and $h$ satisfies the analogous condition $h\circ\sigma^{ }_{\boldsymbol{n}}=h$ for all $\boldsymbol{n}\in\Gamma$.
\end{theorem}

\begin{proof}
    Let $f$ be an arbitrary eigenfunction. By Lemma~\ref{fact:eigenfunctions_as_characters}, $f(x)=C\cdot\chi(\pi^{ }_{\rm MEF}(x))$ for some character $\chi\in\widehat{\mathbb{A}}$ and some constant $C\ne 0$. By the product structure of $\mathbb{A}$, seen in Eq.~\eqref{eq:MEF-genform}, and Fact~\ref{fact:Pontryagin-duality}, there exist $\chi_{\rm tile}\in\widehat{\Z_Q^d}$ and $\chi_{\rm height}\in\widehat{\Z^d/\Gamma}$, such that:
        \[\chi\circ\pi^{ }_{\rm MEF}(x) = \chi_{\rm tile}(\pi_{\rm tile}(x))\cdot\chi_{\rm height}(\pi_{\rm height}(x)). \]
    By the structure of $\Z_Q^d$ as an inverse limit, one may interpret the Pr\"ufer group $\widehat{\Z_Q^d}$ as the nested union of finite subgroups $\Z^d/Q^k\Z^d$, seen, as subgroups of $\widehat{\Z_Q^d}$, and thus $\chi_{\rm tile}$ factors as a composition of group homomorphisms $\Z^d\to\Z^d/Q^k\Z^d\to\mathbb{S}^1$, where the first homomorphism is the canonical quotient morphism. As the latter has kernel $Q^k\Z^d$, one has $\chi_{\rm tile}(\boldsymbol{m}+\boldsymbol{n})=\chi_{\rm tile}(\boldsymbol{m})$, for any $\boldsymbol{n}\in Q^k\Z^d$, the latter identified with the corresponding subgroup of $\Z_Q^d$. A simpler version of the same argument shows that $\chi_{\rm tile}(\boldsymbol{m}+\boldsymbol{n})=\chi_{\rm tile}(\boldsymbol{m})$ for $\boldsymbol{n}\in\Gamma$.
    
    Now, since $\pi_{\rm tile}\circ\sigma^{ }_{\boldsymbol{n}}(x) = \pi_{\rm tile}(x) + \boldsymbol{n}$, and the same holds for $\pi_{\rm height}$ modulo $\Gamma$, one can choose $f_{\rm tile} = C\cdot\chi_{\rm tile}\circ\pi_{\rm tile}$ and $h=\chi_{\rm height}\circ\pi_{\rm height}$, from which the result immediately follows.
\end{proof}

In light of the preceding theorem, we define a \emph{height eigenfunction} as an eigenfunction satisfying $h\circ\sigma^{ }_{\boldsymbol{n}}=h$ for any $\boldsymbol{n}\in\Gamma$. The collection of height eigenfunctions fully characterises the height lattice, as seen from the following.

\begin{lemma}\label{lem:height_lattice_as_common_stab}
For any $\boldsymbol{m}\in\Z^d\setminus\Gamma$, there exists a height eigenfunction $h$ such that $h\circ\sigma^{ }_{\boldsymbol{m}}\ne h$. As a consequence, the following equality holds:
    \[\Gamma = \{\boldsymbol{n}\in\Z^d\::\: h\circ\sigma^{ }_{\boldsymbol{n}}=h\text{ for all height eigenfunctions }h\}.\]
\end{lemma}
\begin{proof}
    Given a fundamental domain $\mathcal{D}_{\Gamma}$ for $\Gamma$ and an associated partition $\{\mathcal{A}_{\boldsymbol{k}}\}_{\boldsymbol{k}\in \mathcal{D}_{\Gamma}}$ of the alphabet, let $\mathcal{D}_{\Gamma}'\subseteq \mathcal{D}_{\Gamma}$ be a subset satisfying the following two conditions:
\begin{itemize}
    \item for any $\boldsymbol{k},\boldsymbol{k}'\in \mathcal{D}_{\Gamma}', \boldsymbol{k}-\boldsymbol{k}'$ is not a multiple of $\boldsymbol{m}$, and
    \item every element of $\mathcal{D}_{\Gamma}$ is of the form $\boldsymbol{k}+c\boldsymbol{m}$ for some $\boldsymbol{k}\in \mathcal{D}_{\Gamma}',c\in\Z$.
\end{itemize}
It immediately follows that the set $\bigcup_{j=0}^{r-1} \mathcal{D}_{\Gamma}'+j\boldsymbol{m}$ is another fundamental domain for $\Gamma$, where $r$ is the order of $[\boldsymbol{m}]$ in $\Z^d/\Gamma$; note that $r\ge 2$, as $\boldsymbol{m}\notin\Gamma$. With appropriate choices of $\mathcal{D}_{\Gamma}$ and $\mathcal{D}_{\Gamma}'$, we may ensure that this new fundamental domain equals $\mathcal{D}_{\Gamma}$. We may then write $\mathcal{A}_j=\bigcup_{\boldsymbol{k}\in \mathcal{D}_{\Gamma}'+j\boldsymbol{m}}\mathcal{A}_{\boldsymbol{k}}$, with $0\le j<r$.  The collection $\{\mathcal{A}_0,\mathcal{A}_1,\dotsc,\mathcal{A}_{r-1}\}$ is thus a partition of $\mathcal{A}$. The map $\X\to\mathbb{S}^1$ given by $f(x)=\ee^{2\pi\ii j/r} $ if $x_0\in\mathcal{A}_j$ is easily seen to be a height eigenfunction, which satisfies $f\circ \sigma^{ }_{\boldsymbol{m}}=\ee^{2\pi\ii/r} f\ne f$, as $r\ge 2$. The second claim follows immediately.
\end{proof}

\subsection{Height and extended symmetries}

Eigenfunctions have natural interactions with extended symmetries, as seen in the following results.

\begin{proposition}
    If $f\colon \X\to\mathbb{S}^1$ is a continuous eigenfunction, and $\Phi\in \mathcal{N}(\X)$, then $f\circ \Phi$ is a continuous eigenfunction as well.
\end{proposition}

\begin{proof}
    As $\Phi\in \text{Homeo}(\X)$, $f\circ\Phi$ is evidently continuous. For any $\boldsymbol{n}\in\Z^d$, we have:
    \begin{align*}
        (f\circ\Phi) \circ \sigma^{ }_{\boldsymbol{n}} &= f\circ(\Phi \circ \sigma^{ }_{\boldsymbol{n}}) = (f\circ \sigma^{ }_{A\boldsymbol{n}}) \circ \Phi \\
        &= (\lambda_{A\boldsymbol{n}}\cdot f)\circ \Phi = \lambda_{A\boldsymbol{n}}\cdot (f\circ \Phi),
    \end{align*}
    and hence the composition of $f\circ\Phi$ with a shift map is a multiple of $f\circ\Phi$.
\end{proof}

Thus, an extended symmetry $\Phi$ naturally induces a group automorphism $\Phi^\dagger\colon\mathcal{E}(X)\to\mathcal{E}(X)$, given by $f\mapsto f\circ\Phi$. Since $\Phi^\dagger(C\cdot f)=C\cdot\Phi^\dagger(f)$, this map in turn induces a group automorphism $\Phi^*\colon\widehat{\mathbb{A}}\to\widehat{\mathbb{A}}$, and both maps are related by the equality:
    \[\Phi^\dagger(\chi\circ\pi_{\rm tile}) = \Phi^*(\chi)\circ\pi_{\rm tile}.\]
We want to show that the map $\Phi^\dagger$ sends height eigenfunctions to height eigenfunctions; this imposes a significant restriction on the admissible $A\in \text{GL}(d,\mathbb{Z})$ that could be associated to an extended symmetry $\Phi$, as we shall see below. The induced map $\Phi^*$ on $\widehat{\mathbb{A}}$ will play a key role, due to the following characterisation of height eigenfunctions in terms of the decomposition $\widehat{\mathbb{A}}\cong\widehat{\mathbb{Z}_Q^d}\times\widehat{\Z^d/\Gamma}$.

\begin{lemma}\label{lem:characters_of_height_eigenfunctions}
    The function $h=\chi\circ\pi^{ }_{\rm MEF}$ is a height eigenfunction if, and only if, $\chi=(\mathbf{1}_{\widehat{\Z_Q^d}},\chi^{ }_2)$, with $\chi^{ }_2\in\widehat{\Z^d/\Gamma}$.
\end{lemma}

\begin{proof}
    Suppose $\chi=(\mathbf{1}_{\widehat{\Z_Q^d}},\chi^{ }_2)$ is a character from $\widehat{\mathbb{A}}$ and $h=\chi\circ\pi^{ }_{\rm MEF}$ the associated eigenfunction. Then, using the equality $\pi^{ }_{\rm MEF}(\sigma^{ }_{\boldsymbol{m}}(x)) = \pi^{ }_{\rm MEF}(x) + (\boldsymbol{m},[\boldsymbol{m}]_\Gamma)$, we have
    \begin{align*}
        h \circ \sigma^{ }_{\boldsymbol{m}}(x) &= \mathbf{1}_{\Z_Q^d}(\pi_{\rm tile}(x)+\boldsymbol{m})\cdot\chi^{ }_2(\pi_{\rm height}(x)+[\boldsymbol{m}]_\Gamma) \\
        &= 1\cdot \chi^{ }_2(\pi_{\rm height}(x)+[\boldsymbol{m}]_\Gamma) 
        = 1\cdot \chi^{ }_2(\pi_{\rm height}(x))\cdot \chi^{ }_2([\boldsymbol{m}]_\Gamma)
        =h(x)\cdot\chi^{ }_2([m]_\Gamma),
    \end{align*}
    and, if $\boldsymbol{m}\in\Gamma$, we have $[m]_\Gamma=[0]_\Gamma$, so that $\chi^{ }_2([m]_\Gamma)=1$ and thus $h\circ\sigma^{ }_{\boldsymbol{m}}=h$.

    Conversely, let $h=\chi\circ\pi^{ }_{\rm MEF}$ be a height eigenfunction associated to a character $\chi$, and suppose $\chi=(\chi^{ }_1,\chi^{ }_2)$. Thus, if we write $h_{\rm tile}(x)=\chi^{ }_1\circ\pi_{\rm tile}$ and $h_2=\chi^{ }_2\circ\pi_{\rm height}$, by the above argument $h_2$ is also a height eigenfunction. Thus, for any $\boldsymbol{m}\in\Gamma$, we must have:
        \[h_{\rm tile}\circ\sigma^{ }_{\boldsymbol{m}}=\frac{h\circ\sigma^{ }_{\boldsymbol{m}}}{h_2\circ\sigma^{ }_{\boldsymbol{m}}}=\frac{h}{h_2}=h_{\rm tile},\]
    so that $h_{\rm tile}$ is also a height eigenfunction.
    
    On the other hand, as $\widehat{\Z_Q^d}$ is a direct limit of the family of finite groups $\Z^d/Q^k\Z^d$, this group may be thought of as a nested union $\bigcup_{k\ge 1} G_k$ of groups of maps $\Z_Q^d\to\mathbb{S}^1$, where the value of $\eta(\boldsymbol{z})$ depends only on the last $k$ digits of $\boldsymbol{z}$ in its $Q$-adic expansion, for $\eta\in G_k$. Thus, in particular, for any $\boldsymbol{m}\in Q^k\Z^d$, we must have $\eta(\boldsymbol{z}+\boldsymbol{m})=\eta(\boldsymbol{z})$. This applies, in particular, to $\chi^{ }_1$. Thus, we have, for some $k$ and any $\boldsymbol{m}\in Q^k\Z^d$:
        \[h_{\rm tile}\circ\sigma^{ }_{\boldsymbol{m}}(x)=\chi^{ }_1(\pi_{\rm tile}(\sigma^{ }_{\boldsymbol{m}}(x)))=\chi^{ }_1(\pi_{\rm tile}(x)+\boldsymbol{m})=\chi^{ }_1(\pi_{\rm tile}(x))=h_{\rm tile}(x), \]
    and thus $h_{\rm tile}\circ\sigma^{ }_{\boldsymbol{m}+\boldsymbol{n}}=h_{\rm tile}$ for any $\boldsymbol{m}\in\Gamma,\boldsymbol{n}\in Q^k\Z^d$.
    
    However, by definition, the height lattice $\Lambda$ is chosen to be coprime to any supertile lattice $Q^k\Z^d$, and thus any $\boldsymbol{n}\in\Z^d$ is a sum of an element of $\Gamma$ with one from $Q^k\Z^d$. Thus, $h_{\rm tile}$ is constant along the orbit of any point $x$; by minimality, this implies that $h_{\rm tile}=\chi^{ }_1\circ \pi_{\rm tile}$ is a constant function. As $\pi_{\rm tile}$ is surjective, the character $\chi^{ }_1$ must equal the value $\chi^{ }_1(\boldsymbol{0})=1$ everywhere, and thus $\chi^{ }_1=\mathbf{1}_{\widehat{\Z_Q^d}}$.
\end{proof}
From this lemma, we can show that the map $\Phi^\dagger$ preserves height eigenfunctions. We need the following duality result.
\begin{lemma}\label{lem:dual_subgroup_preserving}
    Let $G=G_1\times G_2$ be a product of two LCA groups, and $\widehat{G}=\widehat{G}_1\times\widehat{G}_2$ be its dual under the standard isomorphism. Suppose that $\varphi\colon G\to G$ is a group endomorphism such that $\varphi[G_1\times\{\mathbf{1}_{G_2}\}]\subseteq G_1\times\{\mathbf{1}_{G_2}\}$, and let $\varphi^*\colon\widehat{G}\to\widehat{G}$ be the dual group endomorphism. Then $\varphi^*[\widehat{G}_1\times\{\mathbf{1}_{\widehat{G}_2}\}]\subseteq \widehat{G}_1\times\{\mathbf{1}_{\widehat{G}_2}\}$.
\end{lemma}
\begin{proof}
    Note that, since we identify a pair $(\chi^{ }_1,\chi^{ }_2)\in\widehat{G}_1\times\widehat{G}_2$ with the character $G_1\times G_2\to\mathbb{S}^1$ given by $(g_1,g_2)\mapsto\chi^{ }_1(g_1)\chi^{ }_2(g_2)$, a character $\chi\in\widehat{G}$ belongs to $\widehat{G}_1\times\{\mathbf{1}_{\widehat{G}_2}\}$ if, and only if, $\chi(g_1,g_2)=\chi(g_1,\mathbf{1}_{G_2})$, for any choice of $g_2\in G_2$.

    Now, as $\varphi$ maps the set $G_1\times\{\mathbf{1}_{G_2}\}$ into itself, there must exist a group endomorphism $\varphi_1\colon G_1\to G_1$ and a function $\psi\colon G_1\times G_2\to G_2$ (which is not necessarily a group morphism) such that
    \begin{itemize}
        \item $\varphi(g_1,g_2)=(\varphi_1(g_1),\psi(g_1,g_2))$, and
        \item $\psi(g_1,\mathbf{1}_{G_2})=\mathbf{1}_{G_2}$, for any choice of $g_1\in G_1$.
    \end{itemize}
    
    Let $\chi=(\chi^{ }_1,\mathbf{1}_{\widehat{G}_2})\in\widehat{G}_1\times\{\mathbf{1}_{\widehat{G}_2}\}$, and take any arbitrary element $(g_1,g_2)\in G$. We have:
    \begin{align*}
        \varphi^*(\chi)(g_1,g_2) &= \chi(\varphi(g_1,g_2)) = \chi(\varphi_1(g_1),\psi(g_1,g_2)) \\
        &= \chi^{ }_1(\varphi_1(g_1))\cdot \mathbf{1}_{\widehat{G}_1}(\psi(g_1,g_2))= \chi^{ }_1(\varphi_1(g_1))\cdot \mathbf{1}_{\widehat{G}_1}(\psi(g_1,\mathbf{1}_{G_2}))\\
        &=\chi(\varphi(g_1,\mathbf{1}_{G_2})) 
        \varphi^*(\chi)(g_1,\mathbf{1}_{G_2}),
    \end{align*}
    and thus $\varphi^*(\chi)$ also satisfies the condition $\varphi^*(\chi)(g_1,g_2)=\varphi^*(\chi)(g_1,\mathbf{1}_{G_2})$, so it maps an element of $\widehat{G}_1\times\{\mathbf{1}_{\widehat{G}_2}\}$ to another element of this subgroup.
\end{proof}
\begin{corollary}\label{cor:height-to-height}
    The map $\Phi^\dagger$ maps height eigenfunctions to height eigenfunctions.
\end{corollary}

\begin{proof}
    Let $h=\chi\circ\pi^{ }_{\rm MEF}$ be a height eigenfunction. By the characterisation in Lemma~\ref{lem:characters_of_height_eigenfunctions}, one has that $\chi=(\mathbf{1}_{\widehat{\Z_Q^d}},\chi_{\rm height})\in \{\mathbf{1}_{\widehat{\Z_Q^d}}\}\times\widehat{\Z^d/\Gamma}$. By definition, $\Phi^\dagger(h)=\Phi^*(\chi)\circ\pi^{ }_{\rm MEF}$. Thus, the result immediately follows from Lemma~\ref{lem:characters_of_height_eigenfunctions} by showing that $\Phi^*(\chi)\in  \{\mathbf{1}_{\widehat{\Z_Q^d}}\}\times\widehat{\Z^d/\Gamma}$, or, equivalently, that $\Phi^*$ maps this subgroup to itself.

    Note that, as $\Z_Q^d$ is a torsion-free group, and $\Z^d/\Gamma$ is a finite Abelian group, the torsion subgroup of $\mathbb{A}=\Z_Q^d\times\Z^d/\Gamma$ must equal $\{\mathbf{1}_{\Z_Q^d}\}\times\Z^d/\Gamma$. This is a characteristic subgroup of $\mathbb{A}$, and is thus preserved by any group automorphism $\mathbb{A}\to\mathbb{A}$; this applies, in particular, to the dual map $\Phi^{**}=(\Phi^*)^*$, as the dual of a bijective homomorphism is also bijective due to functoriality. Then, by appealing to Lemma~\ref{lem:dual_subgroup_preserving}, we see that the map $\Phi^{***}=(\Phi^{**})^*$ must map  $\{\mathbf{1}_{\widehat{\Z_Q^d}}\}\times\widehat{\Z^d/\Gamma}$ to itself as well. Since by Pontryagin duality $\Phi^{***}=\Phi^*$, the result follows.
\end{proof}

Now we have everything we need to give a proof of Theorem~\ref{thm:height} here, namely, that for any extended symmetry $\Phi$ with associated matrix $A$ we must have $A\Gamma=\Gamma$. Note that we do not assume here that $\Phi$ is radius-zero.

\begin{proof}[Proof of Theorem~\ref{thm:height}]
    Let $\Phi$ be an extended symmetry with associated matrix $A$. For  $f\in\mathcal{E}(\X)$, let $\operatorname{stab}(f)$ be the set $\{\boldsymbol{n}\in\Z^d\::\:f\circ\sigma^{ }_{\boldsymbol{n}}=f\}$. A quick computation yields
        \begin{align*}
            \boldsymbol{n}\in\operatorname{stab}(f\circ\Phi) &\iff f\circ\Phi=f\circ\Phi\circ\sigma^{ }_{\boldsymbol{n}}=f\circ\sigma^{ }_{A\boldsymbol{n}}\circ\Phi
            \\
            &\iff f\circ\sigma^{ }_{A\boldsymbol{n}}=f \iff A\boldsymbol{n}\in\operatorname{stab}(f), \end{align*}
    so $\operatorname{stab}(f\circ \Phi)=A^{-1}\operatorname{stab}(f)$. 
    Let $\mathcal{H}$ be the set of all height eigenfunctions. Lemma~\ref{lem:height_lattice_as_common_stab} shows that $\Gamma=\bigcap_{h\in\mathcal{H}}\operatorname{stab}(h)$, hence
    \[
        A^{-1}\Gamma = \bigcap_{h\in\mathcal{H}} A^{-1}\operatorname{stab}(h) = \bigcap_{h\in\mathcal{H}} \operatorname{stab}(h\circ\Phi).
    \]
    By Corollary~\ref{cor:height-to-height}, the set $\{h\circ\Phi\::\:h\in\mathcal{H}\}$ equals $\mathcal{H}$, so the above intersection runs through all height eigenfunctions and consequently the right-hand side equals $\Gamma$. Thus, $A^{-1}\Gamma=\Gamma$, which is equivalent to the desired result as $A\in\operatorname{GL}(d,\Z)$.
\end{proof}

 For $\boldsymbol{m}\in \mathbb{Z}^d$, a sublattice $\Gamma\leq \mathbb{Z}^d$ and a choice of fundamental domain $\mathcal{D}_{\Gamma}$, we denote by $(\boldsymbol{m})_{\Gamma}$ the unique element in $\mathcal{D}_{\Gamma}$ that is equivalent to $\boldsymbol{m}$ mod $\Gamma$. 
We will also need the notion of border-forcing for substitutions. 
A substitution $\theta\colon\mathcal{A}\to\mathcal{A}^{\mathcal{D}}$ with shape $\mathcal{D}$ is said to \emph{force the border} up to radius $r$ if, for any $x,y\in\X_\theta$ with $x_{\boldsymbol{0}}=y_{\boldsymbol{0}}$, then $\theta(x)\rvert_{\mathcal{D}+[-r,r]^d} = \theta(y)\rvert_{\mathcal{D}+[-r,r]^d}$, that is, the symbols around a supertile are also determined by the type of supertile up to a radius $r$. It is known that $\X_{\theta}$ is topologically conjugate to shift space of a substitution that forces the border up to any required radius, by using the standard technique known as \emph{collaring} and taking powers of the substitution; see \cite{Kel,Sadun}.

The following is a generalisation of \cite[Thm.~14]{Dek}  and \cite[Props.~3.4~\&~3.5]{CQY} to block substitutions.

\begin{proposition}\label{prop:pure-base}
    Given an aperiodic, primitive block substitution $\theta$ on an alphabet $\mathcal{A}$ with height lattice $\Gamma$, then there is a shift space $\mathbb{Y}_{\boldsymbol{0}}$ with torsion-free MEF such that
    \[
        (\X_\theta, \sigma) \cong (\mathbb{Y}_{\boldsymbol{0}}\times \mathcal{D}_{\Gamma}, T), \quad \text{with}\quad  T_{\boldsymbol{m}}(x,\boldsymbol{j})= (\sigma^{ }_{\boldsymbol{q}}(x),(\boldsymbol{j} + \boldsymbol{m})_\Gamma),
    \]   
    where $\Gamma=H\mathbb{Z}^d,\,\mathcal{D}_{\Gamma}\simeq\mathbb{Z}^d/\Gamma$, and $\boldsymbol{q}= H^{-1}(\boldsymbol{j}+\boldsymbol{m}- (\boldsymbol{j}+\boldsymbol{m})_\Gamma)) $. Furthermore, if the matrix $Q'=H^{-1}QH$ has integer entries, $\mathbb{Y}$ can be made substitutive, i.e. there exists a substitution $\vartheta$ with trivial height lattice and inflation matrix $Q'$ such that $\mathbb{Y}=\X_\vartheta$. Finally, any  $\Phi\in \mathcal{C}(\X_\theta)$ is of the form
        $\Phi_{\boldsymbol{i}}(x,\boldsymbol{j})= (\sigma^{ }_{\boldsymbol{q}}\Phi'(x), (\boldsymbol{j} + \boldsymbol{i})_\Gamma)$, with $\Phi'\in \mathcal{C}(\mathbb{Y}_{\boldsymbol{0}})$.
\end{proposition}

\begin{proof}
		Let $\mathcal{D}_{\Gamma}\ni \boldsymbol{0}$ be a fundamental domain for $\Gamma$, and let $\mathcal{B}:=\mathcal{L}_{\mathcal{D}_{\Gamma}}(\X_{\theta})$, that is, the set of all legal patterns of $\X_\theta$ with support $\mathcal{D}_{\Gamma}$. Define a map $\Psi\colon\X_\theta\to\mathcal{B}^{\Z^d}$ via the following equation:
		\[\Psi(x)_{\boldsymbol{n}} \coloneqq \sigma^{ }_{H\boldsymbol{n}}(x)\rvert_{\mathcal{D}_{\Gamma}}. \]
	As $\Gamma + \mathcal{D}_{\Gamma} = H\Z^d + \mathcal{D}_{\Gamma} = \Z^d$, it is clear that this map is injective, as if $x_{\boldsymbol{m}}\ne y_{\boldsymbol{m}}$ for some $\boldsymbol{m}\in H\boldsymbol{n}+\mathcal{D}_{\Gamma}$, then $\Psi(x)_{\boldsymbol{n}}\ne\Psi(y)_{\boldsymbol{n}}$. As well, a routine check shows that $\Psi$ is continuous, and its own definition shows that the following equation holds.
		\begin{equation}
			\label{eq:shift_commuting_for_height}\sigma^{ }_{\boldsymbol{n}}\circ \Psi = \Psi\circ\sigma^{ }_{H\boldsymbol{n}},
		\end{equation}
	which in particular implies that $\mathbb{Y}=\Psi[\X_{\theta}]$ is shift-invariant and closed, hence a subshift.
	
	Now, let $\{\mathcal{A}_{\boldsymbol{k}}\}_{\boldsymbol{k}\in \mathcal{D}_{\Gamma}}$ be the alphabet partition induced by $\Gamma$ as in Fact~\ref{fact:alph-partition}. Note that $\X_{\theta}$ partitions into $\lvert \mathcal{D}_{\Gamma}\rvert$ subsets $\{X_{\boldsymbol{k}}\}_{\boldsymbol{k}\in \mathcal{D}_{\Gamma}}$ given by
		$X_{\boldsymbol{k}}\coloneqq\{x\in\X_\theta\::\: x_{\boldsymbol{0}}\in\mathcal{A}_{\boldsymbol{k}}\}$,
	and that, for each $\boldsymbol{k}\in \mathcal{D}_{\Gamma}$, the shift map $\sigma^{ }_{\boldsymbol{k}}$ maps bijectively $X_{\boldsymbol{0}}$ into $X_{\boldsymbol{k}}$, while $\sigma^{ }_{\boldsymbol{m}}[X_{\boldsymbol{k}}]=X_{\boldsymbol{k}}$ for any $\boldsymbol{m}\in \Gamma$. Thus, $\mathbb{Y}$ is the disjoint union of $\lvert \mathcal{D}_{\Gamma}\rvert$ pairwise conjugate subshifts $\mathbb{Y}_{\boldsymbol{k}}=\Psi[X_{\boldsymbol{k}}]$, where $\Psi\circ\sigma^{ }_{\boldsymbol{k}}\circ\Psi^{-1}$ is a conjugacy between $\mathbb{Y}_{\boldsymbol{0}}$ and $\mathbb{Y}_{\boldsymbol{k}}$. Eq.~\eqref{eq:shift_commuting_for_height} then implies that $\Psi$, restricted appropriately, is a conjugacy between the $\mathbb{Z}^d$-group action
		$\sigma^H\colon
		(\boldsymbol{n},x) \mapsto \sigma^{ }_{H\boldsymbol{n}}(x)$ 
	on $X_{\boldsymbol{0}}$ and the shift action $\sigma$ on $\mathbb{Y}_{\boldsymbol{0}}$. From here, it is easy to see that the corresponding action $\sigma^H$ on $X_{\boldsymbol{k}}$ is conjugate to $\mathbb{Y}_{\boldsymbol{0}}$ via the map $\Psi\circ\sigma^{ }_{-\boldsymbol{k}}$.
	
	Now, let $Y$ be the set $\mathbb{Y}_{\boldsymbol{0}}\times (\Z^d/\Gamma)$. As $\mathcal{D}_{\Gamma}$ is a fundamental domain for $\Gamma$, it may be bijectively identified with the group $\Z^d/\Gamma$, as every equivalence class has a unique representative in $\mathcal{D}_{\Gamma}$. For any $\boldsymbol{n}\in\Z^d$ the following map $T_{\boldsymbol{n}}\colon Y\to Y$ is well-defined:
		\[T_{\boldsymbol{n}}(x,(\boldsymbol{k})_{\Gamma})=(\sigma^{ }_{H^{-1}(\boldsymbol{n}+\boldsymbol{k}- (\boldsymbol{n}+\boldsymbol{k})_\Gamma )}(x),(\boldsymbol{n}+\boldsymbol{k})_{\Gamma}). \]
    It is clear that the maps $T_{\boldsymbol{n}}$ are continuous. Furthermore, the map $\psi\colon Y\to\X_\theta$ given by
		\[\psi(x,(\boldsymbol{k})_{\Gamma})=\sigma^{ }_{\boldsymbol{k}}\circ\Psi^{-1}(x),\quad\boldsymbol{k}\in D_{\Gamma},\]
	naturally sends each subset $\mathbb{Y}_{\boldsymbol{0}}\times\{(\boldsymbol{k})_{\Gamma}\}$ to $X_{\boldsymbol{k}}$ for each $\boldsymbol{k}\in D_{\Gamma}$, and is thus a (continuous) bijection. Furthermore, it is equivariant as a consequence of Eq.~\eqref{eq:shift_commuting_for_height}, as seen below:
		\begin{align*}
			\psi\circ T_{\boldsymbol{n}}(x,(\boldsymbol{k})_{\Gamma}) &=\psi(\sigma^{ }_{H^{-1}(\boldsymbol{n}+\boldsymbol{k}-(\boldsymbol{n}+\boldsymbol{k})_\Gamma)},(\boldsymbol{n}+\boldsymbol{k})_{\Gamma}) \\
			&=\sigma^{ }_{(\boldsymbol{n}+\boldsymbol{k})_\Gamma}(\Psi^{-1}\circ\sigma^{ }_{H^{-1}(\boldsymbol{n}+\boldsymbol{k}-(\boldsymbol{n}+\boldsymbol{k})_\Gamma)}(x))\\
			&=\sigma^{ }_{(\boldsymbol{n}+\boldsymbol{k})_\Gamma} (\sigma^{ }_{HH^{-1}(\boldsymbol{n}+\boldsymbol{k}-(\boldsymbol{n}+\boldsymbol{k})_\Gamma)}\circ\Psi^{-1}(x)) \\
			&=\sigma^{ }_{\boldsymbol{n}+\boldsymbol{k}}\circ\Psi^{-1}(x) = \sigma^{ }_{\boldsymbol{n}}(\sigma^{ }_{\boldsymbol{k}}\circ\Psi^{-1}(x))=\sigma^{ }_{\boldsymbol{n}}\circ\psi(x,(\boldsymbol{k})_{\Gamma}).
		\end{align*}
	
	Thus, $Y$ is a suspension of the subshift $\mathbb{Y}_{\boldsymbol{0}}$ which is conjugate to the original shift space $\X_\theta$, via the homeomorphism $\psi$. Note that a similar construction can be done with any of the $\mathbb{Y}_{\boldsymbol{k}}$, as they are all conjugate. The claim about the MEF of $\mathbb{Y}_{\boldsymbol{0}}$ follows from the fact that any eigenfunction $f\colon\mathbb{X}_\theta\to\mathbb{S}^1$ induces an eigenfunction $f'\colon\mathbb{Y}_{\boldsymbol{0}}\to\mathbb{S}^1$ defined by $f'=f\circ\Psi$, where the correspondence $f\mapsto f'$ is surjective. Any height eigenfunction is mapped to the trivial eigenfunction $\mathbf{1}$ by this map, while the same coprimality argument as in Lemma~\ref{lem:characters_of_height_eigenfunctions} shows that tiling eigenfunctions cannot be mapped to $\mathbf{1}$. Hence, the Pontryagin dual of the MEF of $\mathbb{Y}_{\boldsymbol{0}}$ is isomorphic to $\widehat{\Z^d_Q}$, wherefore the MEF of $\mathbb{Y}_{\boldsymbol{0}}$ must be isomorphic to $\Z_Q$ by Pontryagin duality, with no torsion component.
	
	It remains to show that $\mathbb{Y}_{\boldsymbol{0}}$ (or, more precisely, some $\mathbb{Y}_{\boldsymbol{k}}$) is substitutive when $Q'=H^{-1}QH$ has integral entries. Since the image of any of the sets $X_{\boldsymbol{k}}$ under $\theta$ is contained in some $X_{\boldsymbol{k}'}$, we may assume without loss of generality that $\theta[X_{\boldsymbol{0}}]\subseteq X_{\boldsymbol{0}}$, as it is not hard to check that, for some power of $\theta$, the relation $\theta^N[X_{\boldsymbol{k}}]\subseteq X_{\boldsymbol{k}}$ will hold, and our construction for $X_{\boldsymbol{0}}$ will work in the same way for any of the other $\boldsymbol{k}$. Under the assumption on $Q'$, one then may construct the substitution $\vartheta$ as follows. 
\begin{enumerate}[label=\textup{(\roman*)},leftmargin=*]
    \item[(i)] The matrix $Q'=H^{-1}QH$ has the same eigenvalues as $Q$ (and hence is also expansive), so if it has integral entries it is a suitable candidate for the inflation matrix of $\vartheta$. Define the set $\mathcal{D}_\vartheta$ as the maximal set satisfying the condition $H\mathcal{D}_\vartheta \subseteq Q\mathcal{D}_\Gamma + \mathcal{D}$; as translations of $Q\mathcal{D}_\Gamma + \mathcal{D}$ along the lattice $QH\Z^d$ tile the plane, this condition ensures that translates of $H\mathcal{D}_\vartheta$ along $HQ'\Z^d$ tile the lattice $H\Z^d$ as well, that is, that $\mathcal{D}_\vartheta$ is a fundamental domain for the lattice $Q'\Z^d.$
    \item[(ii)] Note that, for every $P$ with $\text{supp}(P)=\mathcal{D}_{\Gamma}$, $\text{supp}(\theta(P))=Q\mathcal{D}_{\Gamma}+\mathcal{D}$. There exists $r>0$ such that $Q\mathcal{D}_{\Gamma}+\mathcal{D}+[-r,r]^d\supseteq H\mathcal{D}_\vartheta+\mathcal{D}_{\Gamma}$. 
    \item[(iii)] Consider a collaring $\theta^{(c)}$ of some power of $\theta$ that is border-forcing with radius $r$; see Figure~\ref{fig:pure base} below.
    \item[(iv)]  The new alphabet for $\vartheta$ is the set $\mathcal{B}$ of all legal patches $P\in\mathcal{L}_{\mathcal{D}_\Gamma}(\mathbb{X}_{\theta^{(c)}})$ supported on $\mathcal{D}_{\Gamma}$ such that $P_{\boldsymbol{0}}\in\mathcal{A}_{\boldsymbol{0}}$. Since the support condition in (ii) is satisfied, one can construct the substitution $\vartheta$ which yields $\mathbb{Y}_0$ (with $\text{supp}(\vartheta)=\mathcal{D}_{\vartheta}$ and  $\mathcal{A}\simeq \mathcal{L}_{\mathcal{D}_{\Gamma}}$) as follows: for all $P\in\mathcal{B}$, we take any $x\in\X_\theta$ such that $x\rvert_{\mathcal{D}_\Gamma}=P$, and define $\vartheta(P)_{\boldsymbol{k}} = \theta(x)\rvert_{H\boldsymbol{k}+\mathcal{D}_\Gamma}$ (note that this corresponds to a symbol from $\mathcal{B}$ due to the condition $\theta[\X_{\boldsymbol{0}}]\subseteq\X_{\boldsymbol{0}}$) for all $\boldsymbol{k}\in\mathcal{D}_\vartheta$; compare \cite[Thm.~14]{Dek}. The border-forcing property ensures that this definition is unambiguous, as it does not depend on the chosen $x$ since the image of all possible candidates will match in a set containing $H\mathcal{D}_\vartheta+D_\Gamma$.
\end{enumerate}
\begin{figure}[!h]
    \centering

    \begin{tabular}{ll}
        \multirow{3}{*}{\smash{\includegraphics[width=0.215\linewidth]{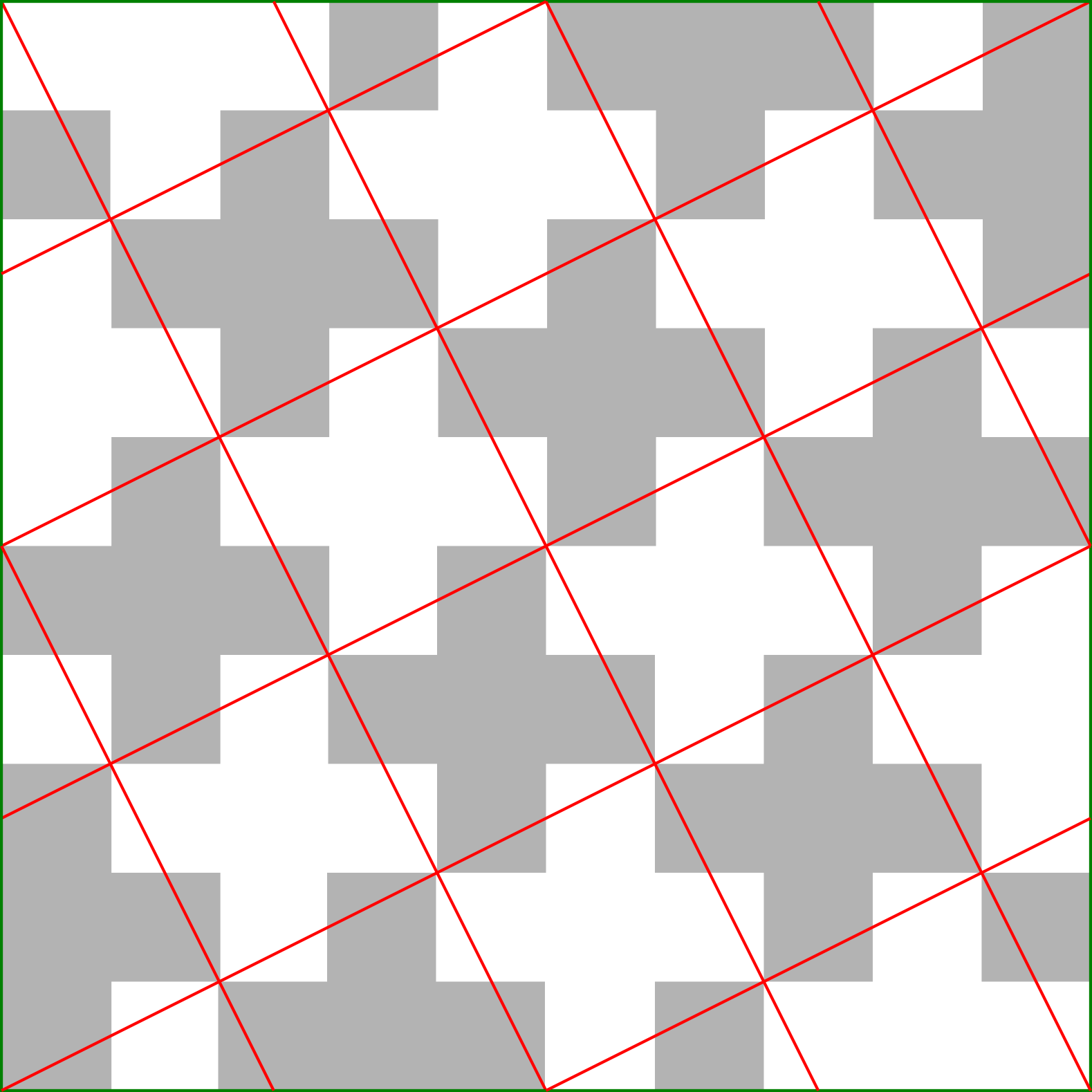}}\vspace{-0.18\linewidth}} & 
             \includegraphics[width=0.18\linewidth]{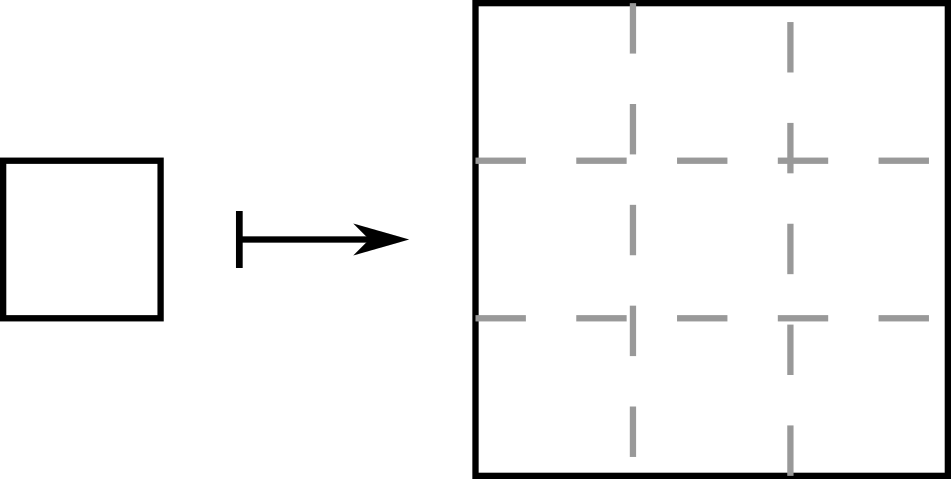} \\
             & \\
              & $ Q = \begin{bmatrix}
                3 & 0 \\ 0 & 3
            \end{bmatrix}, \qquad H=\begin{bmatrix}
                2 & -1 \\
                1 & 2
            \end{bmatrix} $
    \end{tabular} \\ \phantom{this line is intentionally left blank} \\
    \begin{tabular}{cc}
         \includegraphics[width=0.3\linewidth]{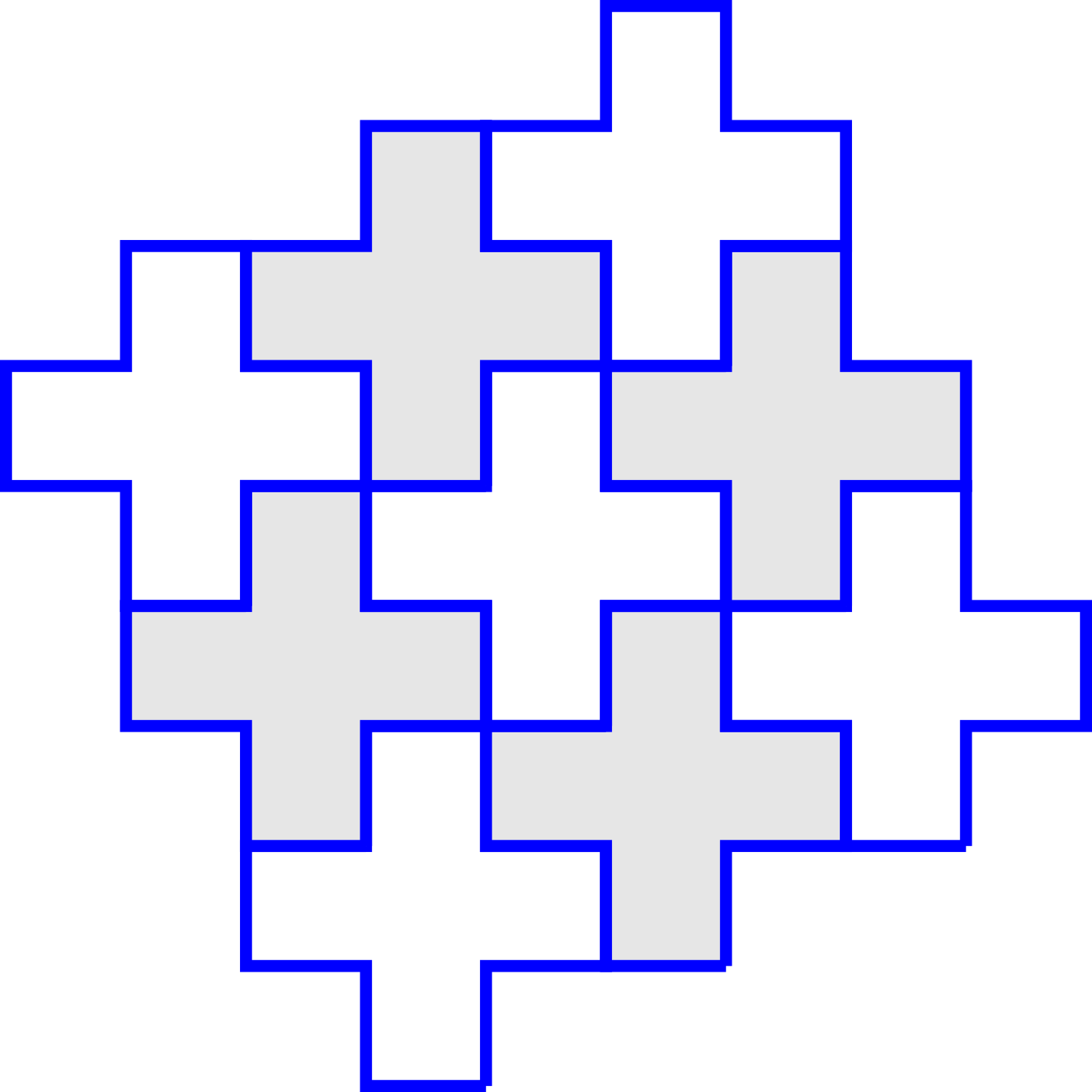}  & 
         \includegraphics[width=0.3\linewidth]{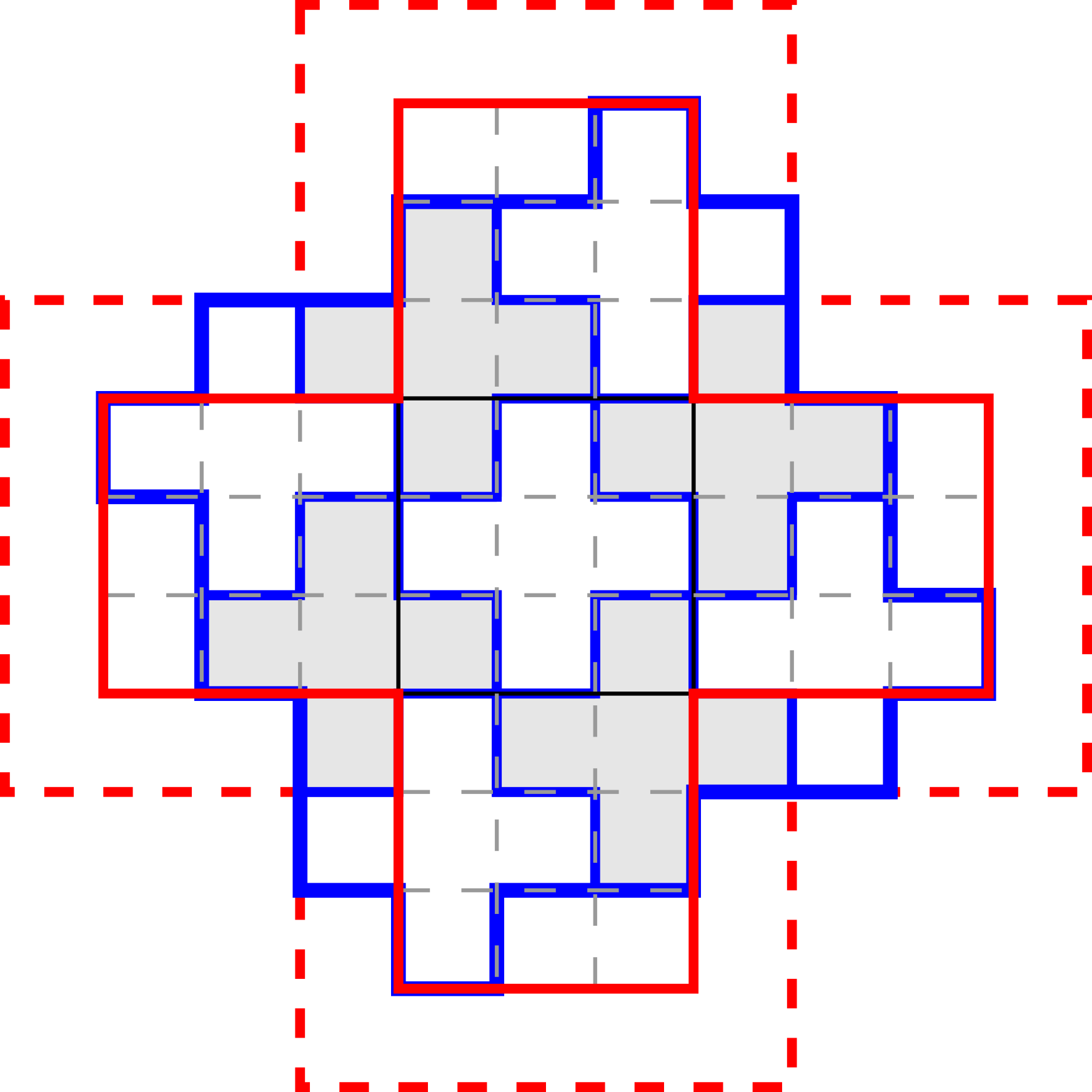}
    \end{tabular}
    \caption{A possible height lattice and fundamental domain for a $3\times 3$ rectangular substitution, with the associated inflation and height matrices. A pattern of shape $\mathcal{D}$ in the pure base corresponds to a pattern of shape $H\mathcal{D}+\mathcal{D}_\Gamma$ (bottom left). However, applying the substitution map to a pattern $P$ of shape $\mathcal{D}_\Gamma$ determines the symbols on a pattern of shape $Q\mathcal{D}_\Gamma + \mathcal{D}$. If the substitution forces the border up to radius $1$ (bottom right), the pattern $\theta(P)$ determines a pattern of shape $\mathcal{D}$ on the pure base (since $H\mathcal{D}+\mathcal{D}_\Gamma\subseteq Q\mathcal{D}_\Gamma+\mathcal{D}+[-1,1]^2$), allowing us to define a substitution (where $Q'=Q$).}
    \label{fig:pure base}
\end{figure}

The proof for the claim for the symmetries is analogous to the one-dimensional case treated in \cite{CQY}, wherefore we leave the details to the reader.
\end{proof}

\begin{remark}
    The condition on $H^{-1}QH$ being integral is naturally satisfied in several common scenarios, e.g. when the substitution is hypercubic, or both the supertile and height lattices are rectangular. Conversely, some (possibly weaker) form of this hypothesis must be required: if $\mathbb{Y}_{\boldsymbol{0}}=\X_\vartheta$ is substitutive with matrix $Q'$ and $x\in\X_\theta$ is a fixed point (and we assume, without loss of generality, that $\Psi(x)\in \mathbb{Y}$ is also fixed for $\vartheta$), a $\mathcal{D}_\Gamma$-shaped pattern in $x$ at position $H\boldsymbol{n}$ determines the symbols near position $QH\boldsymbol{n}$ as well, by virtue of $x$ being a fixed point, but also determines the symbol $\Psi(x)_{\boldsymbol{n}}$ and thus the symbols of $\Psi(x)$ near $Q'\boldsymbol{n}$ as well, since $\vartheta(\Psi(x))=\Psi(x)$. Consequently, the symbol in $x$ at position $\boldsymbol{n}$ determines the symbols of $x$ near position $HQ'\boldsymbol{n}$. Since these correlations produce corresponding eigenfunctions, we see that the scales $(Q^kH)_{k\ge 1}$ and $(H(Q')^k)_{k\ge 1}$ define the same odometer, from which it follows that some telescoping of both scales must match; compare \cite[Lem.~2]{Cor06}.
\end{remark}

With the Proposition~\ref{prop:pure-base} and Theorem~\ref{thm:height}, one can exploit the representation of $\mathbb{X}_{\theta}$ as a suspension of $\X_{\vartheta}$ to relate extended symmetries of the two shifts under some additional assumptions; compare \cite[Rem.~2]{BRY}.

\begin{proposition}
    In the setting of Proposition~\ref{prop:pure-base}, suppose further that $H^{-1}A=A H^{-1}$, for $A\in \textnormal{GL}(d,\mathbb{Z})$.  
    Then, every $\Phi\in \mathcal{N}(\X_{\vartheta})$ with matrix $A$ gives rise to  some $\Psi\in \mathcal{N}(\X_{\theta})$ in the following way 
    \begin{align*}
        \Psi:=\Phi_{\boldsymbol{i}}(\boldsymbol{x}, \boldsymbol{j})= (\sigma^{ }_{A^{-1}\boldsymbol{t}}(\Phi(x)), (\boldsymbol{i}+A \boldsymbol{j})_\Gamma),
    \end{align*}
for some $\boldsymbol{i}\in \mathcal{D}_{\Gamma}$, where $\boldsymbol{t}=(\boldsymbol{i}+A\odot \boldsymbol{j})-(\boldsymbol{i}+A\odot \boldsymbol{j})_\Gamma$.  
\end{proposition}

    \begin{proof}
We want to show that  $ \Phi_{\boldsymbol{i}}$ is an extended symmetry: 
It is clearly a bijection since $\Phi$ is  a bijection on the pure base and $\boldsymbol{i},\boldsymbol{j}$ determine every position uniquely. It then suffices to show that $T^{A\boldsymbol{n}} \circ \Phi_{\boldsymbol{i}} = \Phi_{\boldsymbol{i}} \circ T^{\boldsymbol{n}}$.
\begin{align*}
T^{A\boldsymbol{n}} \circ \Phi_{\boldsymbol{i}}(y,\boldsymbol{j})
&= \big(\sigma^{ }_{H^{-1}(A\boldsymbol{n}+(\boldsymbol{i}+A\odot\boldsymbol{j})_\Gamma -(A\boldsymbol{n}+\boldsymbol{i}+A\odot\boldsymbol{j}))_\Gamma}\sigma^{ }_{H^{-1}(\boldsymbol{i}+A\odot\boldsymbol{j}-(\boldsymbol{i}+A\odot\boldsymbol{j})_\Gamma)}\Phi(y), (\boldsymbol{i} + A \odot \boldsymbol{j} + A \boldsymbol{n})_\Gamma \big) \\ 
&= \big(\sigma^{ }_{H^{-1}[(\boldsymbol{i}+A\boldsymbol{n}+A\odot\boldsymbol{j})-(\boldsymbol{i}+A\boldsymbol{n}+A\odot\boldsymbol{j})_\Gamma]} \Phi(y), (i+A\odot\boldsymbol{j}+A\boldsymbol{n})_\Gamma \big)
\end{align*}
\begin{align*}
 \Phi_{\boldsymbol{i}}  \circ T^{\boldsymbol{n}}(y,\boldsymbol{j}) &=\Phi_{\boldsymbol{i}}\big(\sigma^{ }_{H^{-1}\big(\boldsymbol{j}+\boldsymbol{n}-(\boldsymbol{j}+\boldsymbol{n})_\Gamma\big)}(y), (\boldsymbol{j}+\boldsymbol{n})_\Gamma\big) \\
 &= \big( \sigma^{ }_{H^{-1}\big(\boldsymbol{i}+A(\boldsymbol{j+n})_\Gamma - (\boldsymbol{i} +A(\boldsymbol{j+n})_\Gamma\big)_\Gamma} \Phi (\sigma^{ }_{H^{-1}(\boldsymbol{j}+\boldsymbol{n}-(\boldsymbol{j}+\boldsymbol{n})_\Gamma}(y)), (\boldsymbol{i}+A(\boldsymbol{j}+\boldsymbol{n})_\Gamma)_\Gamma \big)
 \\
 &= \big( \sigma^{ }_{H^{-1}\big(\boldsymbol{i}+A(\boldsymbol{j+n})_\Gamma - (\boldsymbol{i} +A(\boldsymbol{j+n})\big)_\Gamma}  \sigma^{ }_{H^{-1}(A(\boldsymbol{j}+\boldsymbol{n})-A(\boldsymbol{j}+\boldsymbol{n})_\Gamma)}\Phi(y), (\boldsymbol{i}+A(\boldsymbol{j}+\boldsymbol{n}))_\Gamma \big) \\
 &= \big(\sigma^{ }_{H^{-1}[(\boldsymbol{i}+A\boldsymbol{n}+A\boldsymbol{j})-(\boldsymbol{i}+A\boldsymbol{n}+A\boldsymbol{j})_\Gamma]} \Phi(y), (i+A\boldsymbol{j}+A\boldsymbol{n})_\Gamma \big)
 \end{align*}
Note that $A(\boldsymbol{j}+\boldsymbol{n})_\Gamma=(A\boldsymbol{j}+A\boldsymbol{n})_\Gamma$ holds since $A$ leaves $\Gamma$ invariant by Theorem~\ref{thm:height}. 
    \end{proof}

\begin{example}
    Consider the following $4\times 4$ rectangular substitution $\theta$ on six letters:
    \begin{align*}
        a & \mapsto \begin{array}{|c|c|c|c|}
             \hline
             a&b&c&d \\
             \hline
             d&e&f&a \\
             \hline
             d&e&f&a \\
             \hline
             a&b&c&d \\
             \hline
        \end{array} &
        b & \mapsto \begin{array}{|c|c|c|c|}
             \hline
             e&f&d&e \\
             \hline
             b&c&a&b \\
             \hline
             b&c&a&b \\
             \hline
             e&f&d&e \\
             \hline
        \end{array} &
        c & \mapsto \begin{array}{|c|c|c|c|}
             \hline
             f&a&b&c \\
             \hline
             c&d&e&f \\
             \hline
             c&d&e&f \\
             \hline
             f&a&b&c \\
             \hline
        \end{array} \\
        d & \mapsto \begin{array}{|c|c|c|c|}
             \hline
             d&e&f&a \\
             \hline
             a&b&c&d \\
             \hline
             a&b&c&d \\
             \hline
             d&e&f&a \\
             \hline
        \end{array} &
        e & \mapsto \begin{array}{|c|c|c|c|}
             \hline
             b&c&a&b \\
             \hline
             e&f&d&e \\
             \hline
             e&f&d&e \\
             \hline
             b&c&a&b \\
             \hline
        \end{array} &
        f & \mapsto \begin{array}{|c|c|c|c|}
             \hline
             c&d&e&f \\
             \hline
             f&a&b&c \\
             \hline
             f&a&b&c \\
             \hline
             c&d&e&f \\
             \hline
        \end{array}
    \end{align*}
Any point in the shift $\mathbb{X}_\theta$ can be written as a concatenation of $3\times 1$ patterns of the form $abc$ or $def$, so the height lattice for this substitution is $\Gamma = 3\mathbb{Z}\times\mathbb{Z}$. This substitution is bijective, so the quotient $\mathcal{N}(\mathbb{X}_\theta)/\mathcal{C}(\mathbb{X}_\theta)$ is isomorphic to a subgroup of $D_4$; compare \cite{B}. 

Further inspection reveals that the pure base of this subshift is the classic two-dimensional version of the Thue--Morse substitutive shift (indeed, the two symbols $0$ and $1$ of the Thue--Morse shift correspond to $3\times 1$ blocks of the form $abc$ or $def$, respectively). By Proposition~\ref{prop:pure-base} and \cite{B}, one can check that $(a\,d)(b\,e)(c\,f)$ induces the only non-trivial symmetry (modulo a shift), corresponding to the symmetry induced by $(0\,1)$ in the Thue--Morse substitution.

It is not hard to check that $\theta^2$ satisfies the condition from \cite{BLM} of having identity columns on all four corners, so we may use the criteria from there to determine which letter swaps engender extended symmetries. As $\theta^2$ is bijective, a necessary condition for a pair $(\tau,A)$ to generate an extended symmetry in this scenario is for the columns $\theta^2_{\boldsymbol{j}}$ and $\theta^2_{A^{-1}\odot\boldsymbol{j}}$ to be conjugate, for which they ought to have the same cycle structure. 

Consider the pairs 
\[
A_1 = \begin{bmatrix}
    1 & 0 \\
    0 & -1
\end{bmatrix}\text{ with }\tau^{ }_1=\text{id}\quad \quad \quad \text{and}\quad \quad \quad 
A_2 = \begin{bmatrix}
    -1 & 0 \\
    0 & 1
\end{bmatrix}\text{ with }\tau^{ }_2=(a\,c)(d\,f).
\]
One can verify using the criteria in \cite{BLM} that both pairs generate valid extended symmetries. 

Both of these correspond to the reflections along the vertical and horizontal axis in the original two-dimensional Thue--Morse substitution. Thus, it would be natural to expect that rotations by $\frac{1}{2}\pi$, which are compatible with the Thue--Morse substitution, also correspond to an extended symmetry for $\theta$; however, this is not the case.

Indeed, if $R$ is the associated rotation matrix, the column $\theta^2_{\boldsymbol{e}_1}$ corresponds to the cyclic permutation $(a\,b\,c\,d\,e\,f)$, while $\theta^2_{R\odot\boldsymbol{e}_1}=\theta^2_{15\boldsymbol{e}_1 + 14\boldsymbol{e}_2}$ is the order $2$ permutation $(a\,d)(b\,e)(c\,f)$, which cannot be conjugate to an order $6$ cyclic permutation. As $\theta$ is bijective, there are no other compatible matrices in $D_4$ that are not in the subgroup $\langle A_1,A_2\rangle$, and thus, the quotient $\mathcal{N}(\mathbb{X}_\theta)/\mathcal{C}(\mathbb{X}_\theta)$ is isomorphic to $C_2\times C_2$, in contrast to the original Thue--Morse substitution for which this quotient is isomorphic to $D_4$. We remark that this is consistent with the height lattice condition: we have that $R\Gamma=\mathbb{Z}\times 3\mathbb{Z}\ne\Gamma$, so $R$ cannot be associated to a valid extended symmetry.
\end{example}

\section{A counterexample}\label{sec:kappa}

In this short section, we discuss an example of a strongly injective substitution that admits a radius-$0$ extended symmetry with properties that are contradictory to Properties \textbf{(3)} and \textbf{(4)} of Proposition~\ref{fact:rad-0-symm} for 
radius-$0$ symmetries (namely, the $\kappa$-value and being supertile shuffling). We stress that this is not a counterexample for Proposition~\ref{fact:rad-0-symm} itself, but an example to show that properties of the centraliser do not necessarily translate to the normaliser.

We recall that a one-dimensional substitution $\theta$ is called \emph{ strongly left- (right-) injective} if it does not admit a pair of left (right) infinite fixed points $x,y$ such that $x,y$ agree everywhere except at their $0$th entry. If $\theta$ is both strongly left and strongly right injective, then it is called \emph{strongly injective}; see \cite[Sec.~4]{MY}.  

\begin{proposition}
There is a substitution $\varrho$, which admits a radius-$0$ extended symmetry $\Phi$, such that \textnormal{(i)} $\kappa(\Phi)\neq 0$ and \textnormal{(ii)}
 $\Phi$ is a radius-$0$ extended symmetry that is not supertile-shuffling.
\end{proposition}

\begin{example}\label{ex:kappa-nonzero}
    Start with the palindromic substitution $\theta\colon a\mapsto abcba,\, b\mapsto bcacb, \,c\mapsto cabac$. 
    It is clear that this is bijective and that it has a nontrivial reversor given by $m(x)_n:=x_{-n}$, with $\tau=\text{id}$. We can collar this substitution from the left to obtain a conjugate shift space that is no longer palindromic. For example, applying a $2$-left collaring results in a substitution over $\mathcal{A}^{(3)}=\mathcal{L}_3(\mathbb{X}_\theta)$ with rules similar to the following:
    $\prescript{ca}{}{b}\mapsto\prescript{ba}{}{b}\prescript{ab}{}{c}\prescript{bc}{}{a}\prescript{ca}{}{c}\prescript{ac}{}{b}$. To see that this substitution is no longer palindromic, it suffices to note that the mirroring map $m$ no longer maps $\mathbb{X}_\theta$ to itself, as e.g., the word $\prescript{ba}{}{b}\prescript{ab}{}{c}\prescript{bc}{}{a}$ is legal but $\prescript{bc}{}{a}\prescript{ab}{}{c}\prescript{ba}{}{b}$ is not.
    
    However, as each symbol corresponds to a $3$-letter block in the original subshift with appropriate overlaps, composing $m$ with a letter swap of the form $\prescript{ab}{}{c}\mapsto abc\stackrel{m}{\mapsto} cba\mapsto\prescript{cb}{}{a}$ results in a valid extended symmetry, a natural consequence of the conjugacy between the new subshift and the original. Since the new substitution is no longer (strongly) injective, we use the algorithm described in \cite{MY}, to obtain a new, strongly injective substitution whose subshift is conjugate to $\mathbb{X}_{\theta}$, while not being palindromic:  
     \begin{align*}
          \varrho:      0 &\mapsto 30261, & 1 &\mapsto 31603, & 2 &\mapsto 52603, \\
                        3 &\mapsto 53025, & 5 &\mapsto 15261, & 6 &\mapsto 16025.
    \end{align*}      
  As this substitution has height 1, the tiling and maximal equicontinuous factors coincide. By inspecting the coincidence graph, we see that the only irregular fibre (modulo $\mathbb{Z}$) corresponds to $(\dotsc 11111)_5=-\frac{1}{4}$; this fibre has cardinality $6$, while all regular fibres have cardinality $3$.
    
    From the equations $\pi^{ }_{\text{MEF}}\circ\varrho(x)=5\cdot\pi^{ }_{\text{MEF}}(x)$ and $\pi^{ }_{\text{MEF}}\circ\sigma(x)=\pi^{ }_{\text{MEF}}+1$, which are consequences of recognizability, it immediately follows that $\sigma\circ\varrho$ maps the fibre $\pi^{-1}_{\text{MEF}}[-\frac{1}{4}]$ to itself. Moreover, since the second column of $\varrho$ is the identity, the six points of this fibre are fixed points for $\sigma\circ\varrho$: indeed, these points have a first order supertile located at $[-1,4]$, so the identity column of $\varrho$ in this supertile is located at the $0$ position. Thus, these fixed points are entirely determined by the symbol at $0$, and are as follows.
    \begin{align*}
        x^{(0)} &= \dotsc 53025\underline{3.\textcolor{red}{\bf 0}261}52603\dotsc &
        x^{(1)} &= \dotsc 53025\underline{3.\textcolor{red}{\bf 1}603}16025\dotsc\\
        x^{(2)} &= \dotsc 15261\underline{5.\textcolor{red}{\bf 2}603}16025\dotsc &
        x^{(3)} &= \dotsc 15261\underline{5.\textcolor{red}{\bf 3}025}30261\dotsc\\
        x^{(5)} &= \dotsc 31603\underline{1.\textcolor{red}{\bf 5}261}52603\dotsc &
        x^{(6)} &= \dotsc 31603\underline{1.\textcolor{red}{\bf 6}025}30261\dotsc
    \end{align*}
   The pairs $\{x^{(0)},x^{(1)}\},\{x^{(2)},x^{(3)}\},\{x^{(5)},x^{(6)}\}$ are \emph{left-asymptotic} (LA), i.e., the two points of each pair coincide at every position on a left-infinite ray $(-\infty,k_0]$. Similarly, the three pairs $\{x^{(0)},x^{(5)}\},\{x^{(1)},x^{(2)}\},\{x^{(3)},x^{(6)}\}$ are \emph{right-asymptotic} (RA), with an analogous definition. By the CHL theorem, a nontrivial reversor $\Phi\in \mathcal{N}(\mathbb{X}_{\varrho})$ must map a left-asymptotic pair to a right-asymptotic pair and vice versa.

    By Proposition~\ref{prop:ext-sym-kappa}, we may assume (by composing with a shift if necessary) that the  $\Phi$ maps the fiber $\pi^{-1}_{\text{MEF}}[-\frac{1}{4}]$ to itself, and hence permutes the above six points. By minimality, the induced permutation $\tau\colon\mathcal{A}\to\mathcal{A}$ on the alphabet given by $\Phi(x^{(a)})=x^{(\tau(a))}$ determines $\Phi$ completely. The previous observation about asymptotic pairs shows that if, for instance, $\Phi(x^{(0)})=x^{(0)}$, we would necessarily have $\tau(1)=5$, as the LA pair $\{x^{(0)},x^{(1)}\}$ must be mapped to the RA pair  $\{x^{(0)},x^{(5)}\}$. By inspecting the permutations that satisfy these constraints, we obtain that the permutation $\tau = (0\,3)(1\,6)(2\,5)$ induces a valid extended symmetry $\Phi=(\tau,m)$, which relabels every symbol via $\tau$ and then applies a mirroring. However, $\Phi$ does not follow the criteria from \cite{BRY}, as it does not send fixed points to fixed points, and it is not supertile-shuffling. This does not contradict Theorem~\ref{thm:supertile-shuffling}, since $\tau$ fails to satisfy the  condition  $\theta_{1}(\tau[M])=\tau[\theta_{3}(M)]$ for all $M\in\mathcal{X}$; this fails, for example, for $M=\{1,3,5\}$.
\end{example}

\section{The structure of irregular fibres in higher dimensions}\label{sec:irregular-fibres}

We dedicate this section to study the properties of the tiling factor $\pi_{\rm tile}\colon\X_\theta\to\Z_Q$, which codifies the positions of the supertiles into a $Q$-adic number and is a powerful tool to study the symmetries of $\X_\theta$. From Proposition~\ref{prop:pure-base}, we can assume without loss of generality that the shift spaces under study have trivial height lattice, so that the tiling factor coincides with the maximal equicontinuous factor, that is, $\mathbb{A}=\Z_{Q}$ in Eq.~\eqref{eq:MEF-genform}. One of the key properties of $\pi_{\rm tile}$ is that any automorphism or extended symmetry $\Phi\colon \X_\theta\to \X_\theta$ preserves the cardinality of the fibres of $\pi_{\rm tile}$; that is, for any $x\in \X_\theta$, we have $\lvert\pi_{\rm tile}^{-1}(\pi_{\rm tile}(x)) \rvert = \lvert\pi_{\rm tile}^{-1}(\pi_{\rm tile}(\Phi(x))) \rvert$; thus, it is natural to study how these cardinalities behave. In particular, we are interested in the irregular fibres; compare Section~\ref{sec:mef-kappa}. 

As we shall see, there exists a dense subset $\mathcal{R}\subset \Z_Q^d$ where $\lvert\pi_{\rm tile}^{-1}(z) \rvert$ always equals the column number $c_\theta$, so we refer to the sets $\pi_{\rm tile}^{-1}(\boldsymbol{z})$ with $\boldsymbol{z}\in \mathcal{R}$ (and, by abuse of terminology, the corresponding $\boldsymbol{z}$ themselves) as \emph{regular fibres}; any element of $\Z_Q^d\setminus\mathcal{R}$ will necessarily satisfy $|\pi_{\rm tile}^{-1}(\boldsymbol{z})|>c_\theta$, so we refer to it as an \emph{irregular fibre}.
The following result in one dimension relates irregular fibres to the graph $\mathcal{G}(\theta)$ defined in Section~\ref{sec:minsets}.
Here, $\theta$ is a one-dimensional length-$\ell$ substitution, which is aperiodic and primitive.

\begin{theorem}[\hspace{-0.02em}{\cite[Lem.~3.12]{CQY}}]\label{thm:CQY-irreg}
    Let $G:=\widetilde{\mathcal{G}}(\theta)^{\textnormal{op}}$ be the graph obtained from $\mathcal{G}(\theta)$ by removing every subset of $\mathcal{A}$ of cardinality $c_\theta$ and reversing all edges. This graph defines an one-sided sofic shift $\mathcal{Z}_\theta\coloneqq \X_G$ over the alphabet $\{0,1,\dotsc,\ell-1\}$; we identify the elements of $\mathcal{Z}_\theta$ with points from the odometer $\mathbb{Z}_\ell$ in a natural way. Then, any $z\in\mathbb{Z}_\ell\setminus\mathbb{Z}$ constitutes an irregular fibre if, and only if, $z+N\in\mathcal{Z}_\theta$ for some integer $N\in\mathbb{Z}$. \qed
\end{theorem}
Thus, $\mathcal{Z}_\theta$ determines all irregular fibres except perhaps for the fixed points of $\theta$ and their shifts, for which there are dedicated techniques to study them separately.
The proof can be modified in a non-trivial way to handle $d$-dimensional block substitutions, which we do in this section.

To this end, we define a more general family of graphs $\mathcal{G}(\theta,J),J\subseteq\{1,\dotsc,d\}$ that completely describe the fibre structure of a rectangular substitution by considering these edge cases. These graphs will be defined via the use of auxiliary lower-dimensional substitutions, which we call \emph{$J$-derived substitutions}.

\subsection{Derived substitutions and the graph $\widetilde{\mathcal{G}} (\theta, J)^{\mathrm{op}}$}

Before we go to the definition of derived substitutions, we introduce the following notations. For the rest of this section, we let $\theta \colon \mathcal{A} \rightarrow \mathcal{A}^R$ be a $d$-dimensional
  primitive, aperiodic, block substitution, whose support is $R = \prod_{j =
  1}^d [0, \ell_j - 1]$. For a fixed  $1 \leqslant k \leqslant  d$, define 
 $\mathcal{A}_k =\mathcal{L}_{B_k} (\theta)$ to be the set of all
  $\theta$-legal patterns of support $B_k = \{ 0 \}^{k - 1} \times \{ - 1, 0
  \} \times \{ 0 \}^{d - k}$.  Note that $\lvert B_k \rvert = 2$, and thus we identify a  pattern $P$ with support $B_k$ with the two-letter word $P_{(0, \ldots, - 1,
  \ldots 0)} P_{(0, \ldots, 0, \ldots, 0)}$. Define  $R_k := \prod_{1
  \leqslant j \leqslant d}^{j \neq k} [0 , \ell_j - 1]$.

\begin{definition}\label{def:kth-derived}
Let $\theta$ be as above.  Its $k$-th
  \emph{derived substitution}, $\partial_k \theta \colon \mathcal{A}_k
  \to \mathcal{A}_k^{R_k}$, is a $(d - 1)$-dimensional substitution
  over the alphabet $\A_k$ with support $R_k$, and is defined by
  \[ \partial_k \theta (ab) = [\theta_{(n_1, \ldots,
     n_{k - 1}, \ell_k - 1, n_{k + 1}, \ldots, n_d)} (a) \theta_{(n_1, \ldots,
     n_{k - 1}, 0, n_{k + 1}, \ldots, n_d)} (b)]_{0 \leqslant n_j < \ell_j}^{j
     \neq k} . \]
     for a legal pattern $P=ab$.
\end{definition}

\begin{example}[180°] \label{ex:180}
Consider the substitution
 \begin{align*}
\theta\,\,\colon &
\begin{squarecells}{1}
$a$  \nl
\end{squarecells}\,\,\mapsto\,\, 
\begin{squarecells}{2}
$c$ & $b$ \nl
$a$ & $c$  \nl
\end{squarecells}
\quad \quad \quad 
\begin{squarecells}{1}
$b$  \nl
\end{squarecells}\,\,\mapsto\,\, 
\begin{squarecells}{2}
$a$ & $c$ \nl
$c$ & $b$  \nl
\end{squarecells}
\quad \quad \quad 
\begin{squarecells}{1}
$c$  \nl
\end{squarecells}\,\,\mapsto\,\, 
\begin{squarecells}{2}
$a$ & $b$ \nl
$b$ & $a$  \nl
\end{squarecells}
\end{align*}
The corresponding derived substitutions $\partial_{J}\theta$ for $J=\left\{1\right\}$ and $J=\left\{2\right\}$ are
\begin{align*}
\partial^{ }_1 \theta\,\,\colon &
\begin{squarecells}{2}
    $a$  & $a$ \nl
\end{squarecells}\,\,\mapsto\,\, 
\begin{squarecells}{4}
    $c$ & $a$ & $b$ & $c$  \nl
\end{squarecells} \quad \quad
\begin{squarecells}{2}
    $a$  & $b$ \nl
\end{squarecells}\,\,\mapsto\,\, 
\begin{squarecells}{4}
    $c$ & $c$ & $b$ & $b$  \nl
\end{squarecells} \quad \quad
\begin{squarecells}{2}
    $a$  & $c$ \nl
\end{squarecells}\,\,\mapsto\,\, 
\begin{squarecells}{4}
    $c$ & $b$ & $b$ & $a$  \nl 
\end{squarecells} \\
&\begin{squarecells}{2}
    $b$  & $a$ \nl
\end{squarecells}\,\,\mapsto\,\, 
\begin{squarecells}{4}
    $a$ & $a$ & $c$ & $c$  \nl
\end{squarecells} \quad \quad
\begin{squarecells}{2}
    $b$  & $b$ \nl
\end{squarecells}\,\,\mapsto\,\, 
\begin{squarecells}{4}
    $a$ & $c$ & $c$ & $b$  \nl
\end{squarecells} \quad \quad
\begin{squarecells}{2}
    $b$  & $c$ \nl
\end{squarecells}\,\,\mapsto\,\, 
\begin{squarecells}{4}
    $a$ & $b$ & $c$ & $a$  \nl
\end{squarecells} \\
&\begin{squarecells}{2}
    $c$  & $a$ \nl
\end{squarecells}\,\,\mapsto\,\, 
\begin{squarecells}{4}
    $a$ & $a$ & $b$ & $c$  \nl
\end{squarecells} \quad \quad
\begin{squarecells}{2}
    $c$  & $b$ \nl
\end{squarecells}\,\,\mapsto\,\, 
\begin{squarecells}{4}
    $a$ & $c$ & $b$ & $b$  \nl
\end{squarecells} \quad \quad
\begin{squarecells}{2} 
    $c$  & $c$ \nl
\end{squarecells}\,\,\mapsto\,\, 
\begin{squarecells}{4}
    $a$ & $b$ & $b$ & $a$  \nl
\end{squarecells}
\end{align*} 

 \begin{align*}
\partial^{ }_2 \theta \,\,\colon 
&\begin{squarecells}{2}
    $a$  & $b$ \nl
\end{squarecells}\,\,\mapsto\,\, 
\begin{squarecells}{4}
    $c$ & $c$ & $b$ & $a$  \nl
\end{squarecells} \quad \quad
\begin{squarecells}{2}
    $a$  & $c$ \nl
\end{squarecells}\,\,\mapsto\,\, 
\begin{squarecells}{4}
    $c$ & $b$ & $b$ & $a$  \nl
\end{squarecells} \\
& \begin{squarecells}{2}
    $b$  & $a$ \nl
\end{squarecells}\,\,\mapsto\,\, 
\begin{squarecells}{4}
    $b$ & $a$ & $c$ & $c$  \nl 
\end{squarecells} \\
& \begin{squarecells}{2}
    $c$  & $b$ \nl
\end{squarecells}\,\,\mapsto\,\, 
\begin{squarecells}{4}
    $a$ & $c$ & $b$ & $a$  \nl
\end{squarecells} \quad \quad
\begin{squarecells}{2}
    $c$  & $c$ \nl
\end{squarecells}\,\,\mapsto\,\, 
\begin{squarecells}{4}
    $a$ & $b$ & $b$ & $a$  \nl
\end{squarecells}
\end{align*}
Since $\partial_1\theta$ is defined on an alphabet with $9$ letters and $\partial_2\theta$ on one with just $5$, it is immediate that this example cannot have a supertile-shuffling extended symmetry with $A=\text{Rot}_{\pi/2}$.
\end{example}

\begin{remark}
  Note that primitivity and aperiodicity are not necessarily inherited by $\partial_k \theta$. While adding this
  condition would simplify some proofs, our results hold regardless of these
  properties, so we are not going to make any assumptions about $\partial_k
  \theta$. We note, though, that $\X_{\partial_k \theta}$ is
  guaranteed to be non-empty, as we assume $\ell_k > 1$ for all $k$; any
  element $x \in \X_{\theta}$ with $\pi_{\rm tile} (x)_k = 0$
  induces an element of $\X_{\partial_k \theta}$ naturally and vice
  versa.
\end{remark}

We now generalise Definition~\ref{def:kth-derived} to arbitrary subsets of $\{1,\ldots,d\}$.

\begin{definition}
  Given $J = \{ j_1, \ldots, j_r \} \subseteq \{1 \dotsc d\}$ with $j_1 < \cdots <
  j_r$, and under the above hypotheses, define the $J$-{\emph{derived
  substitution}} $\partial_J \theta$ as the substitution $\partial_{j_1}
  (\partial_{j_2} (\cdots \partial_{j_r} \theta))$ obtained by taking derived
  substitutions inductively in decreasing order of indices. The
  $J$-\emph{derived coincidence graph} $\mathcal{G} (\theta, J) \coloneqq
  \mathcal{G} (\partial_J \theta)$ is the coincidence graph of the substitution $\partial_J \theta$.
\end{definition}

\begin{remark}
  The definition of $\mathcal{G} (\theta, J)$ makes sense even when $J =
  \varnothing$ or $J = \{ 1, \ldots, d \}$. In the first case, we just have
  $\mathcal{G} (\theta, \varnothing) =\mathcal{G} (\theta)$; in the second
  case, we can regard $\mathcal{G} (\theta, \{ 1, \ldots, d \})$ as a
  label-less directed graph that represents the natural map $\mathcal{L}_{\{ -
  1, 0 \}^d} (\X_{\theta}) \rightarrow \mathcal{L}_{\{ - 1, 0 \}^d}
  (\X_{\theta})$ induced by the corner columns of $\theta$, and every
  vertex of this graph that lies on a strongly connected component corresponds
  to a valid seed for a periodic point of $\theta$.
\end{remark}

The following lemma ensures that this definition is consistent and depends
only on the set of indices $J$, and not the order in which derived
substitutions are taken. 

\begin{lemma}
  For $i < j$, we have that $\partial_i (\partial_j \theta) = \partial_{j - 1}
  (\partial_i \theta)$, up to a relabelling.
\end{lemma}

\begin{proof}
  By definition, an element of the alphabet for $\partial_i (\partial_j
  \theta)$ corresponds to two elements of the alphabet $\mathcal{A}_j$ that,
  adjacent along the $i$-th coordinate, form a legal pattern for $\partial_j
  \theta$. Legal patterns with support $P$ in $\partial_j \theta$ are, by
  definition of the latter substitution, in a $1$-$1$ correspondence with
  legal patterns for $\theta$ whose support is of the form:
  \[ P^{(j)} = \{ (n_1, \ldots, n_d) \in \mathbb{Z}^d : n_j \in \{ - 1, 0 \},
     (n_1, \ldots, \widehat{n}_j, \ldots, n_d) \in P \}, \]
  that appear in the boundary between two supertiles. In our particular
  scenario, the alphabet for $\partial_i (\partial_j \theta)$ is mapped
  bijectively to the set of all $\theta$-legal patterns appearing in a shared
  border between two supertiles, with support $\{ 0 \}^{i - 1} \times \{ - 1,
  0 \} \times \{ 0 \}^{j - i - 1} \times \{ - 1, 0 \} \times \{ 0 \}^{d - j}$;
  by the same reasoning, a legal $\partial_i (\partial_j \theta)$ pattern with
  support $Q$ can be identified with a $\partial_j \theta$-legal pattern with
  support $Q^{(i)}$, and these patterns are in turn identified with
  $\theta$-legal patterns with support $(Q^{(j)})^{(i)} = \{ (n_1, \ldots,
  n_d) \in \mathbb{Z}^d : n_i, n_j \in \{ - 1, 0 \}, (n_1, \ldots, \widehat{n}_i,
  \ldots, \widehat{n}_j, \ldots, n_d) \in Q \}$ that appear in the appropriate
  boundary. Applying the same reasoning to $\partial_i \theta$, the result
  follows immediately from the equality $(Q^{(j)})^{(i)} = (Q^{(i)})^{(j -
  1)}$ when $i < j$, which is a simple computation.
\end{proof}

Before we go to the main result linking the graph  $\mathcal{G} (\theta, J)$ and irregular fibres, we introduce further notation for ease of presentation. We define $\mathcal{G} (\theta, J)^{\mathrm{op}}$ to be the graph obtained from $\mathcal{G} (\theta, J)$ by reversing the direction of the edges (\emph{reversed graph}).  The vertices of $\mathcal{G} (\theta, J)^{\mathrm{op}}$
  remain to be the sets of legal patterns in $\mathbb{X}_{\theta}$ with support $B_J = \prod_{j \notin J} \{ 0 \} \times \prod_{j \in J} \{ - 1, 0 \}$.
  Finally, we denote by $\widetilde{\mathcal{G}} (\theta, J)^{\mathrm{op}}$ to be the subgraph of 
$\mathcal{G} (\theta, J)^{\mathrm{op}}$ with the vertices with cardinality equal to the column number $c_{\theta}$ removed (\emph{reversed and pruned graph}). 

\begin{example}[90-degree example, non-bijective]\label{ex:90-deg-marginal}
Consider the non-bijective substitution
 \begin{align*}
\theta\,\,\colon &
\begin{squarecells}{1}
$a$  \nl
\end{squarecells}\,\,\mapsto\,\, 
\begin{squarecells}{2}
$c$ & $b$ \nl
$b$ & $c$  \nl
\end{squarecells}
\quad \quad \quad 
\begin{squarecells}{1}
$b$  \nl
\end{squarecells}\,\,\mapsto\,\, 
\begin{squarecells}{2}
$a$ & $c$ \nl
$c$ & $a$  \nl
\end{squarecells}
\quad \quad \quad 
\begin{squarecells}{1}
$c$  \nl
\end{squarecells}\,\,\mapsto\,\, 
\begin{squarecells}{2}
$a$ & $b$ \nl
$b$ & $a$  \nl
\end{squarecells}
\end{align*}
with combinatorially-identical derived substitutions
\begin{align*}
\partial_1=\partial_2 \theta\,\,\colon &
\begin{squarecells}{2}
    $a$  & $b$ \nl
\end{squarecells}\,\,\mapsto\,\, 
\begin{squarecells}{4}
    $c$ & $c$ & $b$ & $a$  \nl
\end{squarecells} \quad \quad
\begin{squarecells}{2}
    $a$  & $c$ \nl
\end{squarecells}\,\,\mapsto\,\, 
\begin{squarecells}{4}
    $c$ & $b$ & $b$ & $a$  \nl 
\end{squarecells} \\
&\begin{squarecells}{2}
    $b$  & $a$ \nl
\end{squarecells}\,\,\mapsto\,\, 
\begin{squarecells}{4}
    $a$ & $b$ & $c$ & $c$  \nl
\end{squarecells} \quad \quad
\begin{squarecells}{2}
    $b$  & $c$ \nl
\end{squarecells}\,\,\mapsto\,\, 
\begin{squarecells}{4}
    $a$ & $b$ & $c$ & $a$  \nl
\end{squarecells} \\
&\begin{squarecells}{2}
    $c$  & $a$ \nl
\end{squarecells}\,\,\mapsto\,\, 
\begin{squarecells}{4}
    $a$ & $b$ & $b$ & $c$  \nl
\end{squarecells} \quad \quad
\begin{squarecells}{2} 
    $c$  & $c$ \nl
\end{squarecells}\,\,\mapsto\,\, 
\begin{squarecells}{4}
    $a$ & $b$ & $b$ & $a$  \nl
\end{squarecells}
\end{align*}

Below we have the graph $\widetilde{\mathcal{G}}(\theta,J)^{\text{op}}$ (with vertices $D,E,F$), where the pruned arrows (corresponding to edges connected to minimal sets) are drawn with dotted lines. We omit the labels for simplicity.
\begin{center}
  \begin{tikzpicture}[->,>=stealth',shorten >=1pt,auto,node distance=1.5cm,
                    semithick]
  \tikzstyle{every state}=[fill=none,text=black]
    \node (11)                    {$\textcolor{gray}{A}$};
    \node (21) [below of=11, left =1cm of 11]       {$\textcolor{gray}{B}$};
    \node (22) [below of=11, right =1cm of 11]       {$\textcolor{gray}{C}$};
    \node[state] (31) [below of=21, left = 0.8cm of 21]                   {$D$};
    \node[state] (32) [below of=21, right =1cm of 21]                   {$E$};
    \node[state] (33) [below of=22, right = 0.8cm of 22]                   {$F$};

  \path (11) edge  [dashed, bend left=20,gray]    node {} (21)
        (11) edge  [dashed, bend right=20,gray]    node {} (22)
        
        (21) edge  [dashed, bend left=20,gray]    node {} (11)
        (21) edge  [dashed, bend left=20,gray]    node {} (22)
        (21) edge  [dashed,gray ]    node {} (31)

        (22) edge  [dashed, bend right=20,gray]    node {} (11)
        (22) edge  [dashed, bend left=20,gray]    node {} (21)
        (22) edge  [dashed,gray ]    node {} (33)

        (31) edge  []    node {} (32)

        (33) edge  []    node {} (32)
        (31) edge  [loop left]    node {} (31)
        (33) edge  [loop right]    node {} (33)
        
        ;
\end{tikzpicture}
\end{center}
Here, the minimal sets are $A=\{cc\}, B=\{ab \}, C=\{ba \}$, while $D=\{ba,ca,bc,cc\}, E =\{bc,cc,cb,ac,ab,ba,ca\}$ and $ F=\{cb,ac,ab,cc\}$.
\end{example}

\begin{remark}
    To show the contrast to the previous example we present the derived coincidence graphs for Example~\ref{ex:180}. Here, the graphs $\widetilde{\mathcal{G}}(\theta,J)^{\text{op}}$ look vastly different depending on the direction; see Figure~\ref{fig:180-graphs}. 
\begin{figure}[!h]
    \subfloat[$J=\left\{1\right\}$]{
  \begin{tikzpicture}[->,>=stealth',shorten >=1pt,auto,node distance=1cm,
                    semithick]
  \tikzstyle{every state}=[fill=none,text=black]
    \node[state] (11)                    {$A$};
    \node[state] (21) [below of=11, left =0.5cm of 11]       {$B$};
    \node[state] (22) [below of=11, right =0.5cm of 11]       {$C$};
    \node[state] (31) [below=0.7cm of 21]                   {$D$};
    \node[state] (32) [right =0.5cm of 31]                   {$E$};
    \node[state] (33) [below=0.7cm of 22]                   {$F$};

  \path (11) edge  [bend left=20]    node {} (21)
        (11) edge  [bend right=20]    node {} (22)
        
        (21) edge  [bend left=20]    node {} (11)
        (21) edge  [bend left=-20]    node {} (22)
        (21) edge  []    node {} (31)

        (22) edge  [bend right=20]    node {} (11)
        (22) edge  [bend left=60]    node {} (21)
        (22) edge  []    node {} (33)

        (31) edge  []    node {} (32)

        (33) edge  []    node {} (32)        
            ;
\end{tikzpicture}
}\quad \quad 
\subfloat[$J=\left\{2\right\}$]{
  \begin{tikzpicture}[->,>=stealth',shorten >=1pt,auto,node distance=2cm,
                    semithick]
  \tikzstyle{every state}=[fill=none,text=black]

    \node[state] (11)                    {$A'$};
    \node[state] (12) [right of =11]       {$B'$};
    \node[state] (13) [right of =12]       {$C'$};

  \path (11) edge  []    node {} (12)

        (12) edge  [bend right=20]    node {} (11)
        (12) edge  [bend left=20]    node {} (11)

        (12) edge  []    node {} (13)
        
        ;
\end{tikzpicture}
}\caption{$\widetilde{\mathcal{G}}(\theta,J)^{\text{op}}$ for Example~\ref{ex:180}}\label{fig:180-graphs}
\end{figure}
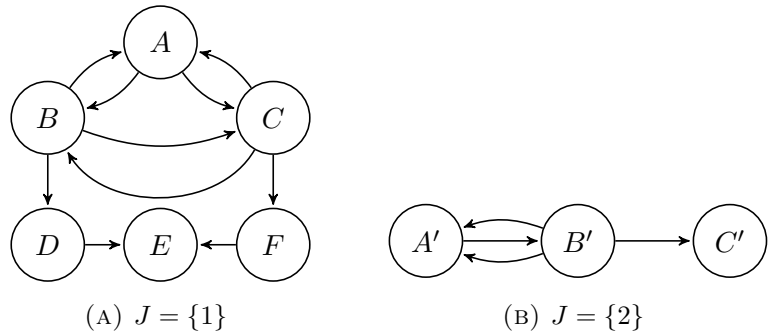
\end{remark}

\subsection{Derived substitutions and irregular fibres}

 Let $\boldsymbol{z} = (z^{(1)},
  \ldots, z^{(d)}) \in \mathbb{Z}_{\boldsymbol{\ell}}$, where each $z^{(j)}$ has
  a base $\ell_j$ expansion given by $\ldots z^{(j)}_2 z^{(j)}_1 z^{(j)}_0$. For each $k\in\N$, define $\boldsymbol{n}_k =
  (z_k^{(1)}, \ldots, z_k^{(j)}, \ldots, z_k^{(d)})$. Further, for $J\subseteq \left\{1,\ldots, d\right\}$, let $\widehat{\boldsymbol{n}}_k :=\widehat{\boldsymbol{n}}^{(J)}_k=
  (z^{(1)}_k, \ldots, \widehat{z}^{(j)}_k, \ldots, z_k^{(d)})$, where the coordinates corresponding to $j\in J$ are deleted.
The next result provides a characterisation of fibres with certain cardinality bounds. 

\begin{theorem}
  Let $\theta : \mathcal{A} \rightarrow \mathcal{A}^R$ be a injective, primitive and
  aperiodic block substitution in $\Z^d$, with expansive matrix $Q=\textnormal{diag}(\ell_1,\ldots,\ell_d)$.
  Take $J = \{ 1 \leqslant j \leqslant d : z^{(j)} \in \mathbb{Z} \}$ and
  suppose that $J = \{ j_1, \ldots, j_r \} \neq [1 ,\ldots, d]$ (that is, at least
  one coordinate is not an integer).  Then the following are equivalent
  \begin{enumerate}
      \item[\textnormal{(i)}] $\pi_{\mathrm{tile}}^{- 1} [\{
  \boldsymbol{z} \}] \geqslant C$ 
  \item[\textnormal{(ii)}] there exists an infinite path
  $\ldots e_2 e_1 e_0 $ in $\mathcal{G} (\theta, J)^{\mathrm{op}}$ that passes
   through vertices of cardinality at least $C$ infinitely many times, and
  where the label of edge $e_k$ corresponds to the tuple $(z_k^{(1)}, \ldots,
  \widehat{z}_k^{(j_1)}, \ldots, \widehat{z}_k^{(j_m)}, \ldots, \widehat{z}_k^{(j_d)},
  \ldots, z_k^{(d)}) \in \prod_{j \in J^c} [0 ,\ldots, \ell_j - 1]$.
  \end{enumerate}
\end{theorem}

\begin{proof}
  For simplicity we will assume that $z^{(j)} = 0$ for every $j
  \in J$. Equivalently, this means that in every deleted coordinate in $\widehat{\boldsymbol{n}}_k$, the label of  edge $e_k$, is $0$ for all $k$. 

  We first show that (ii) implies (i).   Let $e_0 e_1 e_2 \ldots$ be an infinite path in this graph and $v_k
  =\mathfrak{i} (e_k)$ be the corresponding initial vertex of each edge. We can
  identify each $v_k$ with some subset of $\mathcal{L}_{B_J}
  (\mathbb{X}_{\theta})$, and, by the definition of $\mathcal{G}
  (\theta,J)$, the existence of an edge from $v_k$ to
  $v_{k + 1}$ with label $\widehat{\boldsymbol{n}}_k$ implies that, for every
  pattern $P \in v_k$ there exists some pattern $P' \in v_{k + 1}$ such that
  $  \theta (P') |_{B_J + \boldsymbol{n}_k} = P$. Thus, given a $x =
  x^{(0)} \in \pi_{\mathrm{tile}}^{- 1} [\{ \boldsymbol{z} \}]$, as by
  desubstitution there is a uniquely defined infinite sequence of points $\{
  x^{(k)} \}_{k \in \mathbb{N}} \subseteq \mathbb{X}_{\theta}$ such that
  $x^{(k)} = \sigma^{ }_{{\boldsymbol{n}_k} } \circ \theta (x^{(k + 1)})$, we have
  that if $  x^{(k + 1)} |_{B_J} = P'$, then $  x^{(k)}
  |_{B_J} = P$; see Figure~\ref{fig:recog-P}.

 \begin{figure}
 \begin{center}
      \subfloat[Appearance of the patches $P$ and $P'$ in $x^{(k)}$ and $x^{(k+1)}$, respectively. Here $J=\left\{2\right\}$ and $B_J=\{(0,0),(0,-1)\}$.]{\includegraphics[scale=0.75]{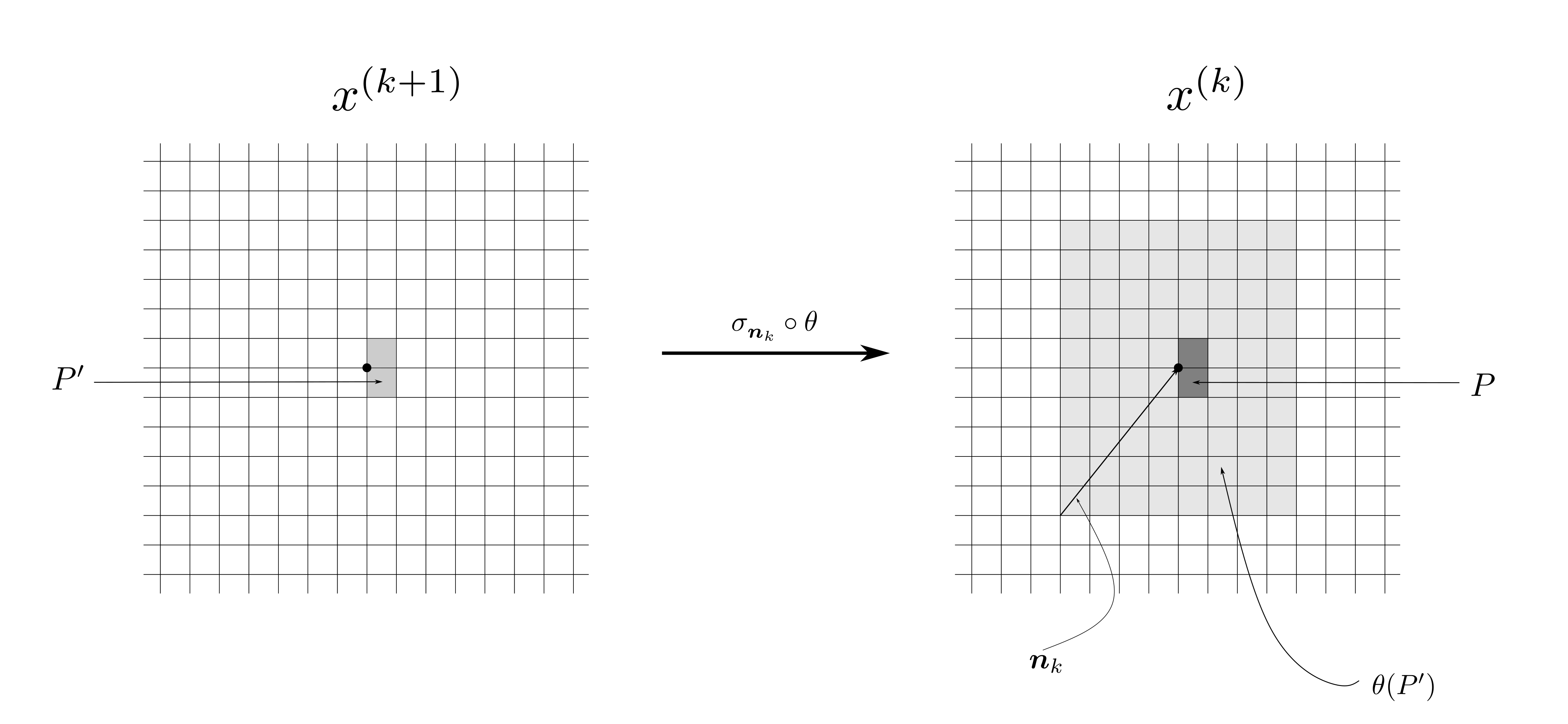}}
      
     \subfloat[A section of $\mathcal{G}(\theta,J)$ with vertices $v_k,v_{k+1}\subset \mathcal{L}_{B_J}(\mathbb{X}_{\theta})$ that encodes the relation in (A)]{\includegraphics[scale=0.75]{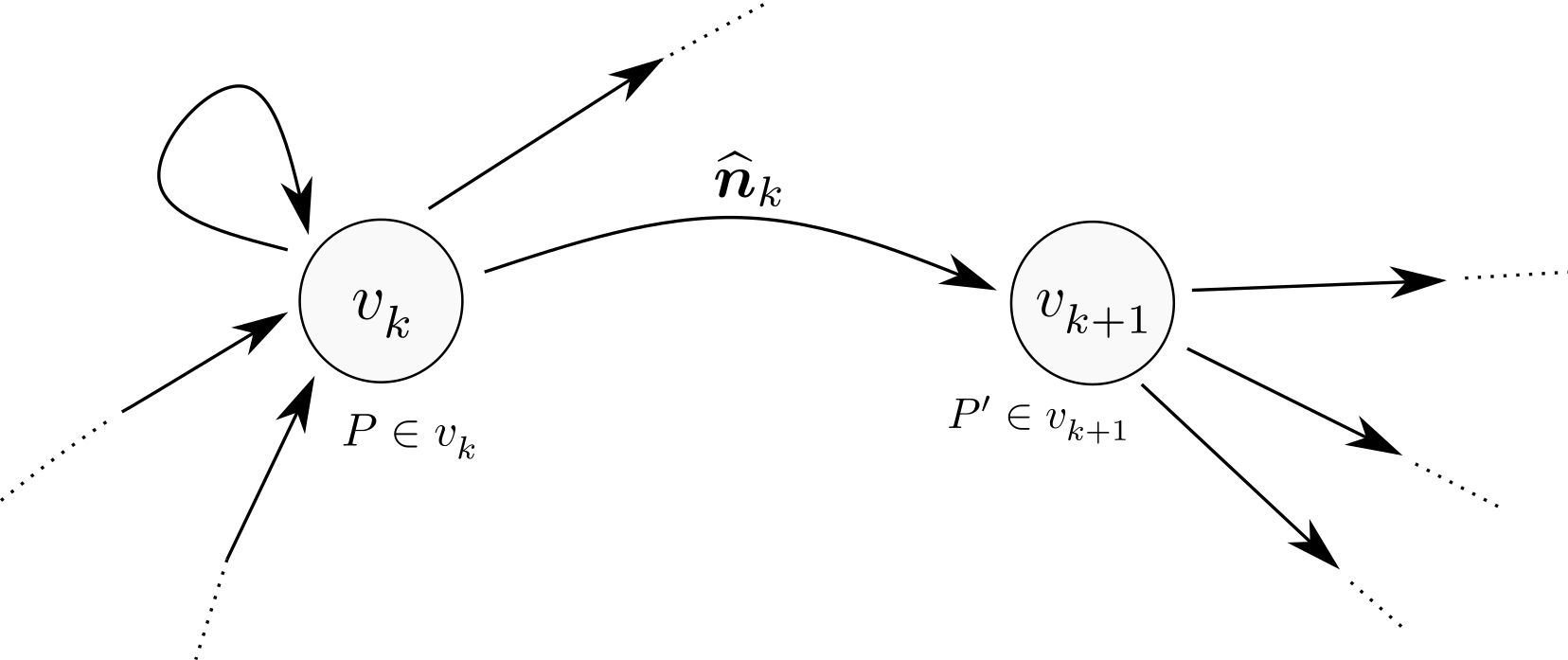}}
 \end{center}
 \caption{Illustration of the properties of the points $x^{(k)}$ and $x^{(k+1)}$ and its corresponding encoding in a section of the graph $\mathcal{G}(\theta,J)$}
     \label{fig:recog-P}
 \end{figure}
  
  Given that $\mathcal{G} (\theta, J)$ is a finite graph, there is some vertex
  $v^{{*}}$ such that $v^{*} = v_k$ for infinitely many
  indices $k_0 < k_1 < k_2 < \cdots \in \mathbb{N}$. By desubstituting $x$
  appropriately if required, we may assume without loss of generality that
  $v_0 = v^{*}$, as both $\theta$ and $\sigma^{ }_{\boldsymbol{n}_k}$ are
  injective functions; equivalently, we assume that $k_0 = 0$. Write
  $t^{(j)}_m$ for the integer with $\ell_j$-adic expansion $z^{(j)}_{m - 1}
  z^{(j)}_{m - 2} \ldots z^{(j)}_0$, and $\boldsymbol{t}_m = (t_m^{(1)}, \ldots,
  t_m^{(d)})$, for any $m\in\mathbb{N}$. By construction, $\boldsymbol{t}_m\to \boldsymbol{z}$ with respect to the topology of  $\mathbb{Z}_{\boldsymbol{\ell}}$.

  Given that every element of $v^{*}$ is a legal pattern $P$, we may
  take some point $u^{(P)} \in \mathbb{X}_{\theta}$ such that $  u^{(P)}
  |_{B_J} = P$. It is easy to see that, for every $m$, $\left.
  \sigma^{ }_{\boldsymbol{t}_{k_m}} \circ \theta^{k_m} (u^{(P)}) \right|_{B_J} \in
  v^{*}$ from our discussion above.  This implies that the
  map $P \mapsto \left. \sigma^{ }_{\boldsymbol{t}_{k_m}} \circ \theta^{k_m} (u^{(P)})
  \right|_{B_J}$ is a surjection from $v^{*}$ onto itself, and hence a
  bijection as $v^{*}$ is finite.
  
  From this, it is not hard to see
  that, for any fixed $m \in \mathbb{N}$ and $P \in v^{*}$, we can
  choose $u^{(P)}_m$ such that \textbf{(1)} $ u^{(P)}_m |_{B_J} = P$,  and \textbf{(2)} there is
  some $u'\in \mathbb{X}_{\theta} $ with $  u' |_{B_J} \in v^{*}$ such that
  $\sigma^{ }_{\boldsymbol{t}_{k_{m + 1}}} \circ \theta^{k_{m + 1}} (u') = u^{(P)}_m$, hence
  \[
  \pi_{\mathrm{tile}} (u_m^{(P)}) \equiv \boldsymbol{t}_{k_m}~\pmod {\ell^{k_m}}.\]

  By compactness, there is some subsequence of $(u_m^{(P)})_{m \in \mathbb{N}}$
  converging to some $\widetilde{u}^{(P)} \in \mathbb{X}_{\theta}$. From \textbf{(1)} we
  immediately see that $ \widetilde{u}^{(P)} |_{B_J} = P$, and \textbf{(2)} and
  continuity imply that $\pi_{\mathrm{tile}} (\widetilde{u}^{(P)}) =
  \boldsymbol{z}$. Hence, there is at least one element $x = \widetilde{u}^{(P)}
  \in \pi_{\mathrm{tile}}^{- 1} [\{ \boldsymbol{z} \}]$ satisfying $  x
  |_{B_J} = P$, for every $P \in v^{*}$, from which $| \pi_{\mathrm{tile}}^{-1} [\{ \boldsymbol{z} \}] | \geqslant | v^{*} |$ is immediate.  This proves the
  first half of the theorem.

  Now, suppose that $| \pi_{\mathrm{tile}}^{- 1} [\{ \boldsymbol{z} \}] | = C$.
  There exists some $N > 0$ such that, for every $x \neq y \in
  \pi_{\mathrm{tile}}^{- 1} [\{ \boldsymbol{z} \}]$, we have $  x
  |_{[- N ,\ldots, N - 1]^d} \neq   y |_{[- N ,\ldots, N - 1]^d}$. Let
  $x = \sigma^{ }_{\boldsymbol{t}_m} \circ \theta^m (x^{(m)})$, which exists by recognisability.
  Recall that, for a rectangular substitution, one has $\text{supp}(\theta^m)=\prod_{j\in \left\{1,\ldots,d\right\}}[0,\ell^{m}_{j}-1]$.

  Fix a subset $J^{\prime}\subseteq J$ and define 
  $\boldsymbol{\varepsilon}_{J'} = - \sum_{j \in J'} \boldsymbol{e }_j$ (where $\boldsymbol{e}_j$ is the $j$-th  element of the canonical basis). 
  Define 
\[ W_{J', m} = \left[\text{supp}(\theta^m)+Q^m\boldsymbol{\varepsilon}_{J}\right]-\boldsymbol{t}_m\subset \mathbb{Z}^d.
\]
This set satisfies the following conditions.
\begin{enumerate}
    \item $W_{J',m}$ is the support of the supertile that contains $\boldsymbol{\varepsilon}_{J'}-\boldsymbol{t}_m$ in $x$
    \item The support $B_J$ of a pattern $P$ contained in some vertex of
  $\mathcal{G} (\theta,J)$ intersects each $W_{J', m}$ in exactly one
  point, for each $J' \subseteq J$.
\end{enumerate}
  Hence, $  x^{(m)} |_{B_J}$
  determines $x^{(0)}$ in the whole of $W_m = \bigcup_{J' \subseteq J} W_{J',
  m}$. An illustration for $W_m$ is provided in Figure~\ref{fig:Tm}.

  \begin{figure}[!h]
      \centering
      \includegraphics[width=0.5\linewidth]{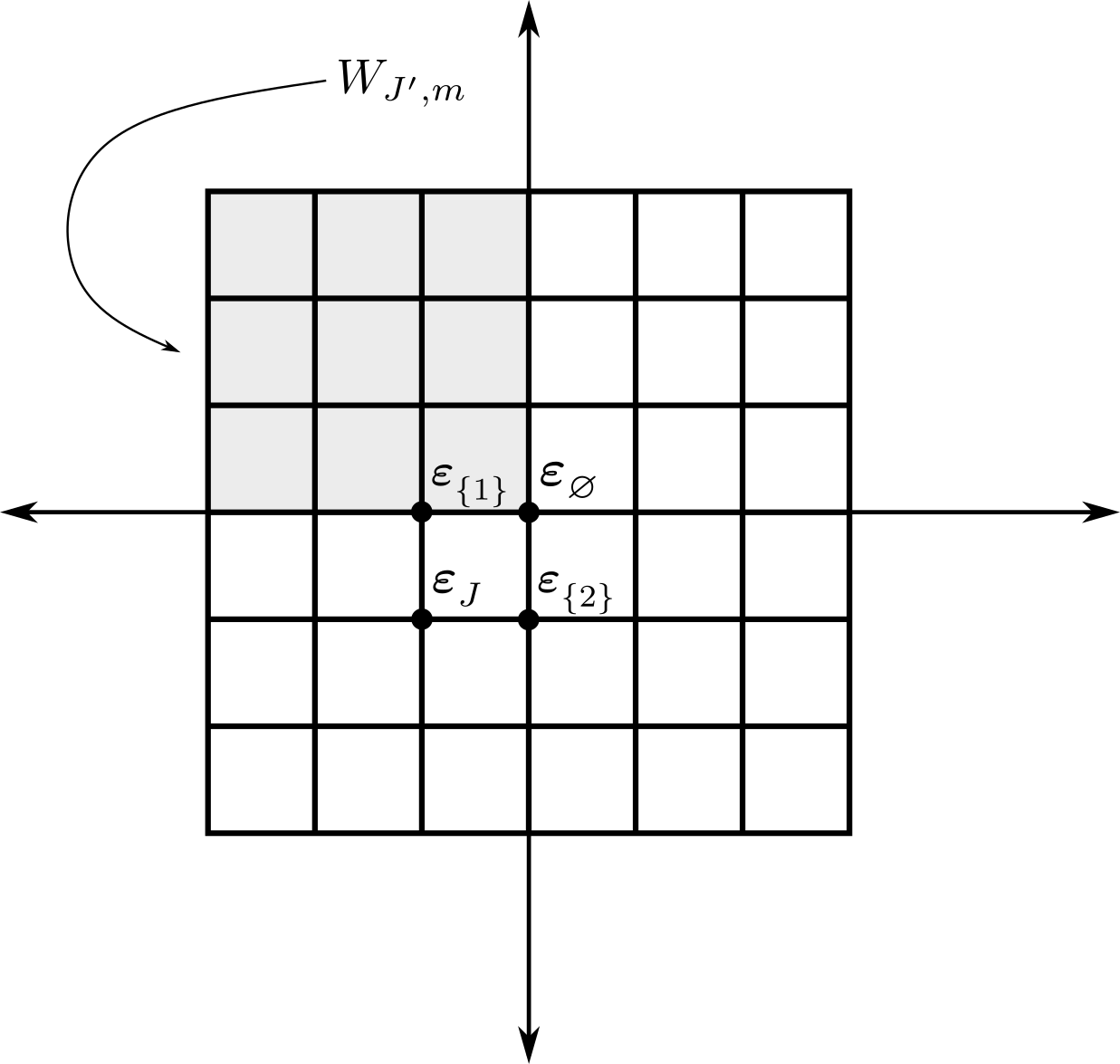}
      \caption{An illustration for $W_m$. Here, $m=1, Q=\text{diag}(3,3), J=\left\{1,2\right\}$ and $\boldsymbol{t}_m=0$. The locations of $\boldsymbol{\varepsilon}_{J'}$, for each $J'\subseteq J$, are also indicated. The gray region corresponds to $W_{J',m}$, where $J'=\left\{1\right\}$.}
      \label{fig:Tm}
  \end{figure}

 Direct computation shows that $W_m$ can be written as:
  \begin{equation}\label{eq:W_m}
  W_m = \left( \prod_{j \in J}  [- \ell_j^m ,\ldots, \ell_j^m - 1] \right)
     \times \left( \prod_{j \notin J}  [0 ,\ldots, \ell_j^m - 1] -
     \widehat{\boldsymbol{t}}_m \right), \end{equation}
  where, as before, $\widehat{\boldsymbol{t}}_m$ is the vector obtained from
  $\boldsymbol{t}_m$ by removing all coordinates with indices in $J$ (without
  loss of generality assumed to be $0$). Since $z^{(j)} \notin \mathbb{Z}$ for
  $j \notin J$, it is not hard to see that the second component in Eq.~\eqref{eq:W_m} grows to cover all of $\mathbb{Z}^{d - | J |}$ as $m$ increases. The
  same evidently holds for the first component in this product, which covers all of
  $\mathbb{Z}^{| J |}$. Thus, $[- N , N - 1]^d \subseteq W_m$ for a
  sufficiently large value of $m$, and all subsequent values of $m$ thereof,
  implying that:
  \[ x \neq y \in \pi_{\mathrm{tile}}^{- 1} [\{ \boldsymbol{z} \}] \Longrightarrow
       x |_{W_m} \neq   y |_{W_m} \Longrightarrow  
     x^{(m')} |_{B_J} \neq   y^{(m')} |_{B_J}, \text{ for all } m'
     \geqslant m. \]
  If we write \[v_k = \{ x^{(k)} |_{B_J} : \pi_{\mathrm{tile}}  (x) =
  \boldsymbol{z} \},\] 
  it is clear then that for any $k \geqslant m$ we must have
  $| v_k | \geqslant C = | \pi_{\mathrm{tile}}^{- 1} [\{ \boldsymbol{z} \}] |$. Recall that $\widehat{\boldsymbol{n}}_k =
  (z^{(1)}_k, \ldots, \widehat{z}^{(j_i)}_k, \ldots, z_k^{(d)})$, where the coordinates corresponding to $j_i\in J$ are deleted.
  Given that $v_k$ corresponds to symbols that appear in a specific
  column of $(\partial_J \theta)^k$, there is by definition of $\partial_J \theta$ at least one
  vertex $\bar{v}_k$ in $\mathcal{G} (\theta,J)$ such that $v_k
  \subseteq \bar{v}_k$, for any $k$, and that any such choice of $\bar{v}_k$,
  for $k \geqslant 1$, necessarily has an inbound edge with label
  $\widehat{\boldsymbol{n}}_k$, such that the initial vertex of such an edge has
  to be a set that contains $v_{k - 1}$. 
  This is because the image of $\bar{v}_k$ under the column  $((\partial_J \theta)^k)^{ }_{\widehat{\boldsymbol{n}}_k}$ has to contain the image of $v_k$ under the same column,
  which in turn contains $v_{k - 1}$.

  As $\mathcal{G} (\theta,J)$ is finite, we have finitely many
  choices for each possible combination of $\bar{v}_k$ and the needed inbound
  edge, so this implies that we can choose the $\bar{v}_k$ in such a way that
  there is an infinite path $e_0 e_1 e_2 \ldots$ with labels
  $\widehat{\boldsymbol{n}}_0, \widehat{\boldsymbol{n}}_1,
  \widehat{\boldsymbol{n}}_2, \ldots$ passing through them. For every $k
  \geqslant m$, we have $| \bar{v}_k | \geqslant | v_k | \geqslant C$, which
  proves the converse implication.
\end{proof}

We now have the following higher-dimensional version of Theorem~\ref{thm:CQY-irreg}. In what follows, we define $X_J$ to be the one-sided edge shift generated by the graph $\widetilde{\mathcal{G}}(\theta,J)^{\text{op}}$. Here, 
we take the reversed graph to be consistent with \cite{CQY}, which considers elements of a sofic shift as \emph{left}-infinite sequences. 
Define 
\[
\mathcal{Z}_{\theta}=\bigcup_{J\subset [1,\ldots,d]} X_J\subseteq R^{\mathbb{N}},
\]
which is sofic. We identify elements of $\mathcal{Z}_\theta$ with points from the $\mathbb{Z}^d_Q$ in a natural way.

\begin{theorem}\label{thm:irreg}
   Let $\theta$ and $J$ be as in the previous theorem.   Any $z\in\mathbb{Z}^d_Q \setminus\mathbb{Z}^d$ constitutes an irregular fibre (that is, $|\pi_{\textnormal{MEF}}^{-1}(\boldsymbol{z})|>c_{\theta}$) if, and only if, $\boldsymbol{z}+N\in\mathcal{Z}_\theta$ for some $N\in\mathbb{Z}^d$.
\end{theorem}

Theorem~\ref{thm:irreg-fibres} then follows as an immediate corollary.

\subsection{Examples}

\begin{example}\label{ex:coinc-C4}
From a substitution $\theta$ with support $R$, one can construct the \emph{column cardinality mosaic} $\mathfrak{C}_{\theta}\subset \mathbb{N}^{R}$ via  via $(\mathfrak{C}_{\theta})^{ }_{\boldsymbol{n}}=\text{card}(\theta_{\boldsymbol{n}}(\mathcal{A}))$.
Here, we have a substitution that admits a coincidence, and admits a supertile-shuffling extended symmetry, for every $A\in C_4$. 

\begin{align*}
\theta\,\,\colon &
\begin{squarecells}{1}
$a$  \nl
\end{squarecells}\,\,\mapsto\,\, 
\begin{squarecells}{2}
$b$ & $e$ \nl
$e$ & $e$  \nl
\end{squarecells}
\quad \quad \quad 
\begin{squarecells}{1}
$b$  \nl
\end{squarecells}\,\,\mapsto\,\, 
\begin{squarecells}{2}
$e$ & $c$ \nl
$e$ & $e$  \nl
\end{squarecells}
\quad \quad \quad 
\begin{squarecells}{1}
$c$  \nl
\end{squarecells}\,\,\mapsto\,\, 
\begin{squarecells}{2}
$e$ & $e$ \nl
$e$ & $d$  \nl
\end{squarecells}\\
&
\begin{squarecells}{1}
$d$  \nl
\end{squarecells}\,\,\mapsto\,\, 
\begin{squarecells}{2}
$e$ & $e$ \nl
$a$ & $e$  \nl
\end{squarecells}
\quad \quad \quad 
\begin{squarecells}{1}
$e$  \nl
\end{squarecells}\,\,\mapsto\,\, 
\begin{squarecells}{2}
$a$ & $b$ \nl
$d$ & $c$  \nl
\end{squarecells}
\end{align*}

One can compute that there exists a power of $\theta$ that admits a coincidence. 
It is visually clear that the rotation $A=\text{Rot}_{-\frac{\pi}{2}}$, together with the permutation $\tau$ satisfies the condition $\theta^n(\tau(a))_{\boldsymbol{n}}=\tau(\theta^n(a))_{A\odot \boldsymbol{n}}$,  for all $\boldsymbol{n}\in \text{supp}(\theta)$. It is clear that  $\mathfrak{C}_{\theta}$ has to stay invariant under $A\in \text{GL}(d,\mathbb{Z})$ if $A$ induces a supertile-shuffling extended symmetry, as one can deduce from Theorem ~\ref{thm:supertile-shuffling}. Clearly, $\mathfrak{C}_{\theta}$ in Figure~\ref{fig: ex-C4} does not permit reflections along diagonals.

\begin{center}
\begin{figure}[!h]
\begin{subfigure}[b]{0.3\textwidth}
    \centering
    \includegraphics[scale=0.5]{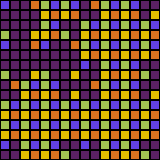}
    \caption{$\theta(a)$}
\end{subfigure}\begin{subfigure}[b]{0.3\textwidth}
    \includegraphics[scale=0.5]{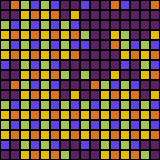}
    \centering
    \caption{$\theta(b)$}
\end{subfigure}\begin{subfigure}[b]{0.3\textwidth}
    \centering
    \includegraphics[scale=0.5]{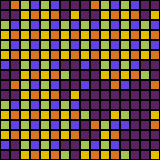}
    \caption{$\theta(c)$}
    \end{subfigure}
    \vspace{2mm}
    
\begin{subfigure}[b]{0.3\textwidth}
    \includegraphics[scale=0.5]{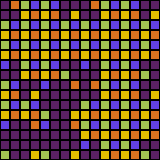} 
        \centering
    \caption{$\theta(d)$}
    \end{subfigure}\begin{subfigure}[b]{0.3\textwidth}
    \centering
    \includegraphics[scale=0.5]{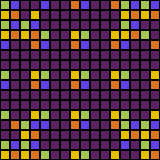}
    \caption{$\theta(e)$}
    \end{subfigure}\begin{subfigure}[b]{0.3\textwidth}
    \includegraphics[scale=0.5]{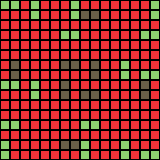}   
        \centering
    \caption{$|\theta_{\boldsymbol{j}}(\mathcal{A})|\in 
            \{\textcolor{lime}{1}, \textcolor{red}{2}, \textcolor{olive}{3}\}$}
    \end{subfigure}
    \caption{The level-4 supertiles of the substitution in Example~\ref{ex:coinc-C4} }
    \label{fig: ex-C4}
\end{figure}
\end{center}
\end{example}

In the following example, we show that the invariance of $\mathfrak{C}_{\theta}$ under $A$ does not guarantee the existence of an extended symmetry with such a geometric map.

\begin{example}\label{ex:manta-ray}
      Consider $\theta\colon \mathcal{A}\to \mathcal{A}^{R}$ defined in Figure~\ref{fig:manta-ray-def}, where one has $\mathcal{A}=\big\{j^{ }_i\colon 0\leqslant i,j\leqslant 2\big\}$ and $R=[0,4]^2$.
      This substitution is primitive, aperiodic, has trivial height lattice, and has column number  $c_{\theta}=6$. Clearly, $\mathfrak{C}_{\theta}$  has full $D_4$-symmetry, but the level-2 supertiles reveal that there cannot exist any supertile-shuffling element of $\mathcal{N}(\mathbb{X}_{\theta})$ apart from one where $A$ is reflection along the  main diagonal. 

\begin{figure}[!h]
\begin{center}
 $0^{ }_i \mapsto $
 \begin{squarecells2}{5}        
            $0^{ }_i$&$0^{ }_i$&$0^{ }_i$&$0^{ }_i$& $0^{ }_i$ \nl
            $0^{ }_i$&$0^{ }_i$&$0^{ }_i$&$0^{ }_i$& $0^{ }_i$ \nl
            $0^{ }_i$&$0^{ }_i$&$2_i$&$0^{ }_i$& $0^{ }_i$ \nl
            $0^{ }_i$&$0^{ }_i$&$0^{ }_i$&$0^{ }_i$& $0^{ }_i$ \nl
            $0^{ }_i$&$0^{ }_i$&$0^{ }_i$&$0^{ }_i$& $0^{ }_{i+1}$ \nl
\end{squarecells2} \;
 $1^{ }_i \mapsto $
 \begin{squarecells2}{5}        
            $1^{ }_i$&$1^{ }_i$&$1^{ }_i$&$1^{ }_i$& $1^{ }_i$ \nl
            $1^{ }_i$&$1^{ }_i$&$2^{ }_i $&$1^{ }_i$& $1^{ }_i$ \nl
            $1^{ }_i$&$2^{ }_i $&$0^{ }_i$&$2^{ }_i $& $1^{ }_i$ \nl
            $1^{ }_i$&$1^{ }_i$&$2^{ }_i $&$1^{ }_i$& $1^{ }_i$ \nl
            $1^{ }_i$&$1^{ }_i$&$1^{ }_i$&$1^{ }_i$& $1^{ }_{i+1}$ \nl
\end{squarecells2} \;
 $2^{ }_i   \mapsto $
 \begin{squarecells2}{5}        
            $2^{ }_i $&$2^{ }_i $&$2^{ }_i $&$2^{ }_i $& $2^{ }_i $ \nl
            $2^{ }_i $&$1^{ }_i$&$1^{ }_i$&$1^{ }_i$& $2^{ }_i $ \nl
            $2^{ }_i $&$1^{ }_i$&$0^{ }_i$&$1^{ }_i$& $2^{ }_i $ \nl
            $2^{ }_i $&$1^{ }_i$&$1^{ }_i$&$1^{ }_i$& $2^{ }_i $ \nl
            $2^{ }_i $&$2^{ }_i $&$2^{ }_i $&$2^{ }_i $& $2^{ }_{i+1}$ \nl
\end{squarecells2}
\end{center}
\caption{The substitution in Example~\ref{ex:manta-ray}}\label{fig:manta-ray-def}
\end{figure}

It is easy to see from the definition of the substitution that the ``manta ray pattern" present on the lower right side of each supertile exists for other supertiles with $i\in \left\{1,2\right\}$ as well.
For this example, one can compute the two (non-trivial) derived substitutions $\partial_{1}\theta$ and $\partial_{2}\theta$, which correspond to the two cardinal directions in $\mathbb{Z}^2$. Both of them are bijective and non-primitive. From the latter, it follows that the corresponding graphs $\mathcal{G}(\theta,J)$ are disconnected. 
\begin{figure}[h]
\begin{center}
\begin{subfigure}[b]{0.4\textwidth}
    \centering
    \includegraphics[scale=0.5]{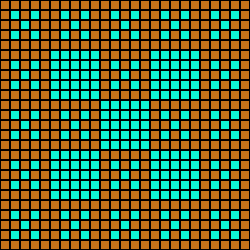}
    \caption{$\mathfrak{C}_\theta$}
\end{subfigure}
\begin{subfigure}[b]{0.4\textwidth}
    \centering
     \includegraphics[scale=0.5]{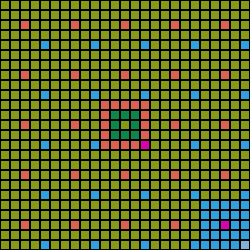}
    \caption{$\theta(0_0)$}
    \end{subfigure} \\
\begin{subfigure}[b]{0.4\textwidth}
    \includegraphics[scale=0.5]{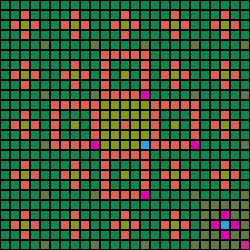}
        \centering
    \caption{$\theta(\mathbf{1}_0)$}
    \end{subfigure}
\begin{subfigure}[b]{0.4\textwidth}
    \centering
    \includegraphics[scale=0.5]{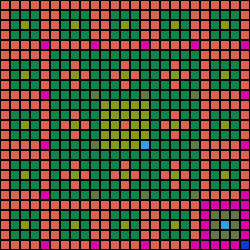}
    \caption{$\theta(2_0)$}
    \end{subfigure}
    \caption{Some supertiles for $\theta$ and the column cardinality mosaic $\mathfrak
    {C}_{\theta}$}\label{fig: manta ray}
    \end{center}
\end{figure}

Since the alphabets for the derived substitutions are quite large, we omit them in the paper and just discuss the relevant quantities for the irregular fibres. This substitution admits $657$ fixed points, which are responsible for an orbit of irregular fibres (\textit{fixed point irregular}). Here, one has $|\pi^{-1}_{\text{MEF}}(\boldsymbol{z})|=657$, whenever $\boldsymbol{z}\in \mathbb{Z}^2\subset \mathbb{A}$.
Another class of irregular fibres (\textit{boundary irregular}) 
come from $\widetilde{\mathcal{G}}(\theta,J)^{\text{op}}$ obtained from $\partial_{1}\theta$ and $\partial_{2}\theta$, which both account for fibres of cardinality 81, respectively. The \textit{interior irregular} fibres come from $\widetilde{\mathcal{G}}(\theta)^{\text{op}}$ and each of them has fibre cardinality $9$. Finally, the regular fibres have cardinality $c_{\theta}=6$.

Because of the cardinality restrictions, boundary irregular fibres should be mapped to boundary irregular fibres by the $\kappa$-cocycle; see Sec.~\ref{sec:ext-sym}. One then gets the following conclusions regarding an extended symmetry $\Phi\in \mathcal{N}(\mathbb{X}_\theta)$.

\begin{itemize}
    \item Any matrix $A$ that generates a $\Phi\in \mathcal{N}(\mathbb{X}_\theta)$ must be in $D_4$ (axes are mapped to axes)
    \item For any $\Phi\in \mathcal{N}(\mathbb{X}_\theta)$, $\kappa(\Phi)=0$, since it must map fixed points to fixed points.
\end{itemize}

From the ``manta ray" argument above, one can also exclude rotations in $D_4$. On the other hand, reflection along the diagonal together with the letter-exchange map $\tau=\text{id}$ generates a valid element of $\mathcal{N}(\mathbb{X}_\theta)$.

Combining together these observations, one gets $\mathcal{N}(X)\supset(\mathbb{Z}^2\ltimes C_2)\times C_3^2$. Note that the $C_3^2$ part comes from radius-zero elements of $\mathcal{C}(\mathbb{X}_\theta)$, which are generated by letter-swaps of the form 
\[
\tau^{ }_1\colon j_i\to ((j+1)\mod 3)_i\quad \text{and}\quad \tau^{ }_2\colon j_i\to j_{(i+1)\mod 3}
\]
One can easily check that both induce supertile-shuffling symmetries. The $C_2$ component is precisely the one generated by the reflection. We suspect that the inclusion above is indeed an equality, but the $\kappa$-value and irregular fibre analysis does not suffice to conclude this. The missing ingredient is a version of strong injectivity (see \cite[Thm.~33]{MY} and Proposition~\ref{fact:rad-0-symm}) in the higher-dimensional case to conclude that all elements of $\mathcal{C}(\mathbb{X}_\theta)$ (modulo shifts) are radius-zero, and hence are all letter swaps.
\exend
\end{example}

\begin{example}[Modified half-hex substitution]\label{ex:half-hex-mod}
    Consider the substitution
    \begin{center}
        $\varrho \colon\quad$  $a \mapsto$ \begin{squarecells}{2}
            $a$ & $b$ \nl
            $a$ & $c$\nl
        \end{squarecells} \quad \quad      
         $b \mapsto$ \begin{squarecells}{2}     
            $a$ & $b$ \nl
            $b$ & $c$ \nl
      \end{squarecells} \quad \quad  $c \mapsto$ \begin{squarecells}{2}
            $a$ & $b$ \nl
            $c$ & $c$\nl
        \end{squarecells}\quad .
\end{center}

    This substitution $\varrho$ is a rectangular encoding of the pattern of hexagons observed in the half-hex tiling; see \cite{BG}, and is the first example of a zero entropy $\mathbb{Z}^d$ shift space, where the normaliser is an infinite-index extension of the centraliser. In fact, every matrix $M\in \operatorname{GL}(2,\mathbb{Z})$ induces an extended symmetry, that is,  $\mathcal{N}(\mathbb{X}_\varrho) \cong \mathbb{Z}^2 \rtimes \operatorname{GL}(2,\mathbb{Z})$; see \cite{CabP} for details. Here, one can show that $\mathcal{C}(\mathbb{X}_\varrho)=\langle\sigma^{ }_{(1,0)},\sigma^{ }_{(0,1)}\rangle\cong\mathbb{Z}^2$; see \cite[Thm.~3.1]{Olli}. It is worth noting here that all elements of $\mathcal{N}(\mathbb{X}_\varrho)$ are, up to a shift, of radius zero.

    \begin{figure}
        \centering
        \raisebox{3em}{\includegraphics[width=0.3\linewidth]{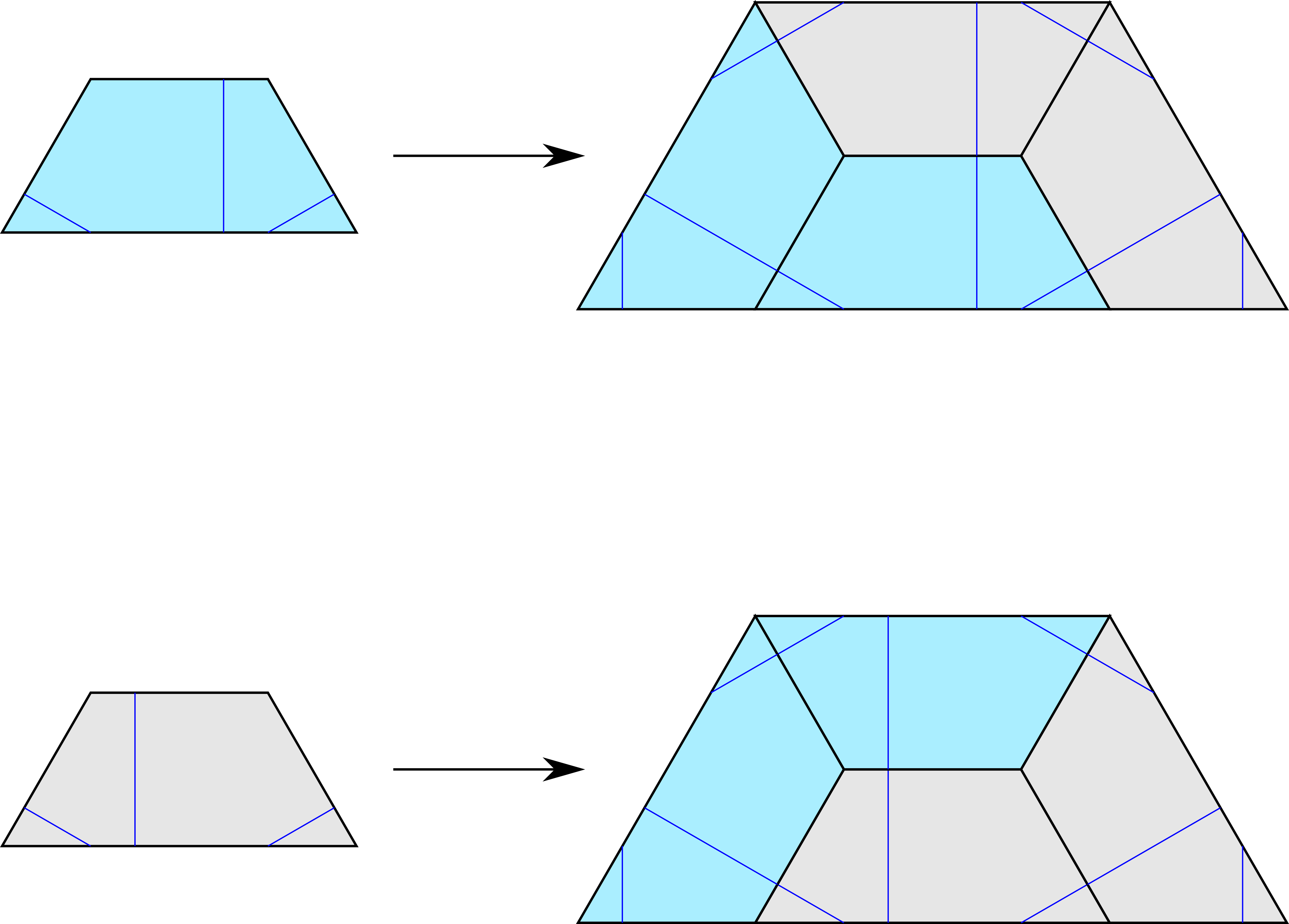}} %
        \hspace{1cm}\includegraphics[width=0.4\linewidth]{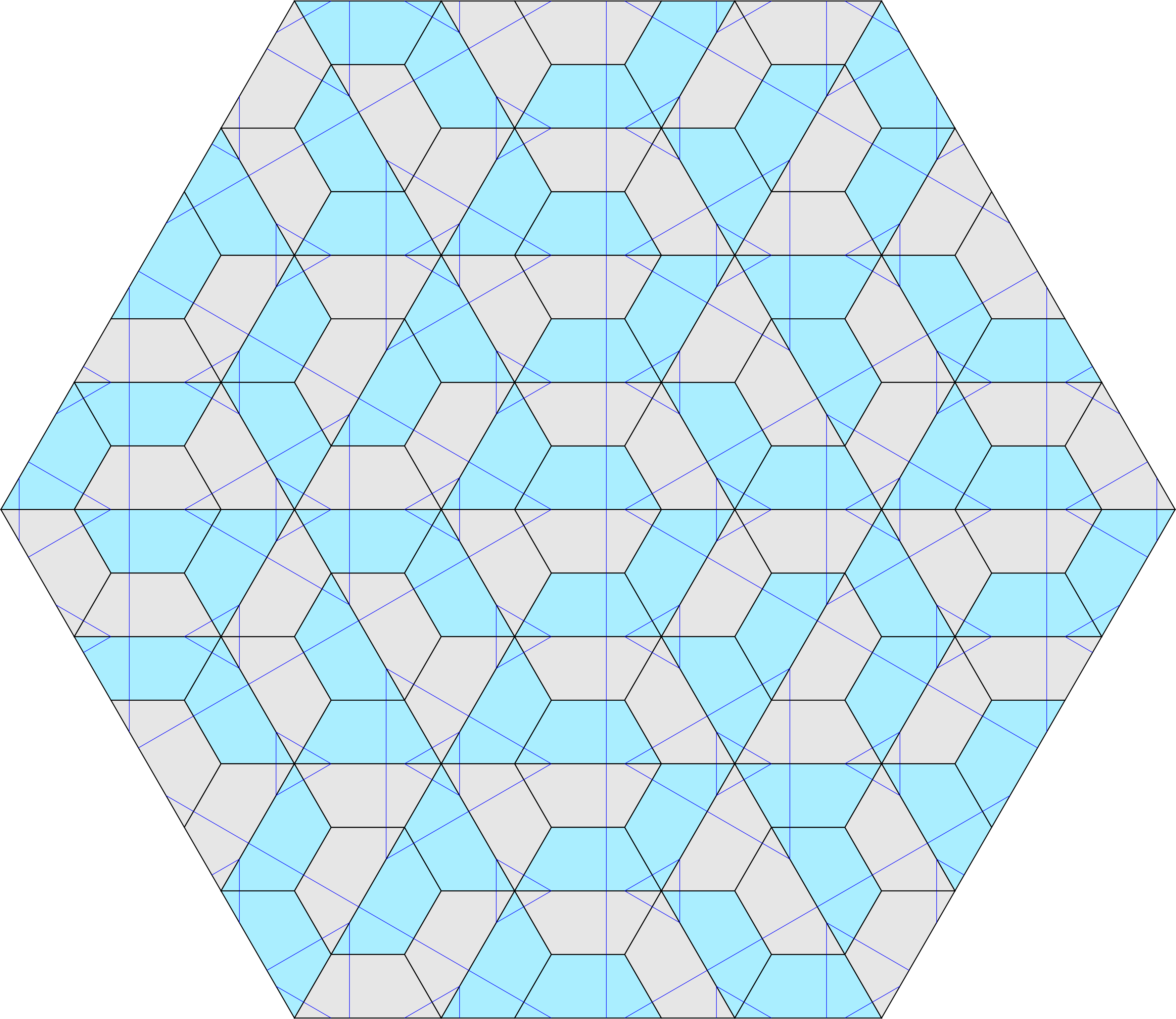}
        \caption{The inflation rule for the half-hex tiling, with decorations, and a patch for this tiling. The decorations induce the appearance of two distinct chiralities, distinguished by color.}
        \label{fig:decorated_half_hex}
    \end{figure}
    Consider now the oriented version $\widetilde{\varrho}$ of $\varrho$ which corresponds to the addition of decorations similar to those found on the Taylor tiling; see \cite[Ex.~6.6]{BG} and Figure~\ref{fig:decorated_half_hex}.

\begin{figure}[!h]    
\begin{center}
         $a^\pm\mapsto$ \begin{squarecells3}{2}
            $a^\pm$ & $b^-$ \nl
            $a^\pm$ & $c^-$\nl
        \end{squarecells3} \quad $b^\pm \mapsto$ \begin{squarecells3}{2}
            $a^+$ & $b^\pm$ \nl
            $b^\pm$ & $c^\mp$\nl
        \end{squarecells3}  \quad
        $c^\pm \mapsto$ \begin{squarecells3}{2}
            $a^-$ & $b^+$ \nl
            $c^\pm$ & $c^\pm$\nl
        \end{squarecells3}
\end{center}\caption{The substitution $\widetilde{\varrho}$ on the six-letter alphabet $\left\{a^\pm,b^\pm,c^\pm\right\}$}

\end{figure}    
    This new substitution once again has a coincidence, so it is an almost-1-1 extension of the underlying odometer $\mathbb{Z}_2\times \mathbb{Z}_2$, which is its MEF. This implies, as well, that the natural factor map $\pi\colon\mathbb{X}_{\tilde{\varrho}}\twoheadrightarrow\mathbb{X}_\varrho$ obtained by removing the orientations (that is, both $a^+$ and $a^-$ are mapped to $a$, and so on) is almost-1-1 as well, as the maximal equicontinuous factor map $\varphi_{\tilde{\varrho}}$ of the new shift space factors naturally through $\pi$ as $\pi^{ }_{\text{MEF}}:=\varphi_{\tilde{\varrho}} = \varphi_{\varrho}\circ\pi$.
    
    The column $\tilde{\varrho}_{(0,0)}$ has cardinality $6$, while every other column of the form $(\tilde{\varrho}^n)_{(j,0)},(\tilde{\varrho}^n)_{(0,j)}$ or $(\tilde{\varrho}^n)_{(j,j)}$ has cardinality $2$, and all other columns in any power of $\tilde{\varrho}$ have cardinality $1$.
   Direct computation of the derived substitutions or their associated graphs show that the only irregular fibres not of cardinality $2$ are those associated with the fixed points, for which there exist exactly $6$ possible seeds, as seen in Figure~\ref{fig:coincidence_graph_dechalfhex}. One may alternatively argue that the oriented version of the half-hex substitution has a $2$-to-$1$ factor map onto the standard half-hex substitution by removing the decorations, and the latter has a single irregular fibre of cardinality $3$, with all others being of cardinality $1$. Either argument shows that every fibre of $\pi^{ }_{\text{MEF}}$ satisfies
    \[
    |\pi^{-1}_{\text{MEF}}[\boldsymbol{z}]|\in \left\{1,2,6\right\},
    \]
    with the $6$-fibres corresponding (modulo a shift) to the fixed points of the substitution and the $2$-fibres corresponding to pairs of points $\{x^+,x^-\}$ which match everywhere in $\mathbb{Z}^2$ except on a horizontal, vertical or diagonal (north-east) ``fault line'', where they differ on the sign (orientation) of the symbol repeated along that line (e.g. we may have $x^\pm_{(j,j)}=b^\pm$ for every $j\in\mathbb{Z}$, while $x^+_{(j,k)} = x^-_{(j,k)}$ whenever $j\ne k$).

    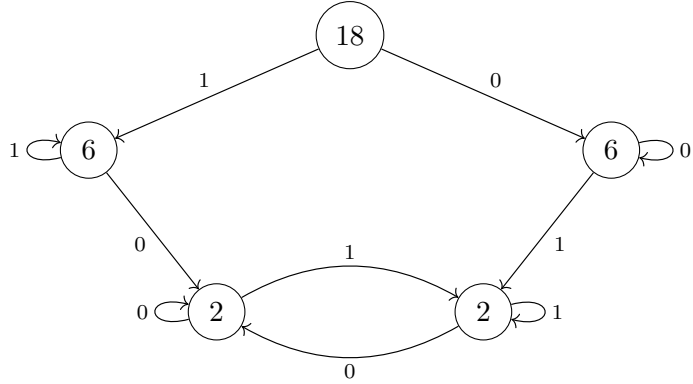
\begin{figure}
        \centering
        \begin{tikzcd}
            & & |[draw,circle]| 18 \arrow[dll,"1"'] \arrow[drr,"0"] \\
            |[draw,circle]| 6\arrow[ddr,"0"'] \arrow[loop left]{r}{1}  & & & & |[draw,circle]|\arrow[ddl,"1"] 6 \arrow[loop right]{r}{0}  \\
            \\
            & |[draw,circle]| 2 \arrow[rr,bend left,"1"] \arrow[loop left]{r}{0}  & & |[draw,circle]| 2 \arrow[ll,bend left,"0"] \arrow[loop right]{r}{1} 
        \end{tikzcd}
        \caption{The coincidence graph associated to the horizontal derived substitution obtained from the decorated half-hex substitution. The numbers indicate the cardinality of the set associated to each vertex. The graph associated to the vertical derived substitution has the same structure.}
        \label{fig:coincidence_graph_dechalfhex}
    \end{figure}
    
    An application of the CHL theorem is thus sufficient to ensure that any extended symmetry must map any such fibre $\{x^+,x^-\}$ to another pair of the same kind $\{y^+,y^-\}$, as extended symmetries preserve proximal and distal pairs (similar to Example~\ref{ex:kappa-nonzero} above). This, in turn, shows that the associated matrix $M$ of any such extended symmetry must map the set $\{(1,0),(0,1),(1,1),(-1,0),(0,-1),(-1,-1)\}$ to itself, as this set corresponds to the possible directions of the ``fault lines''. Thus, the only possible matrices are those coming from the natural embedding of the group $D_6$ of symmetries of the hexagon into $\operatorname{GL}(2,\mathbb{Z})$. It is relatively easy to find extended symmetries representative of two generators of this group, and a slightly more involved argument shows that there are no nontrivial standard symmetries besides the shift maps, proving that in this scenario $\mathcal{N}(\mathbb{X}_{\tilde{\varrho}})\cong\mathbb{Z}^2\rtimes D_6$. 
    
    We note, finally, that $\mathbb{X}_{\tilde{\varrho}}$ is the unique minimal subset of an aperiodic zero-entropy shift of finite type $\mathbb{X}_{\rm ST}$ given by local rules described by the possible decorations, and thus the ``fracture method'' described in \cite{B} for the Robinson shift can also be employed here to show that $\mathcal{N}(\mathbb{X}_{\rm ST})\cong\mathbb{Z}^2\rtimes D_6$ as well.
\end{example}

\section*{Acknowledgments}

The authors would like to thank Michael Baake, Sebasti\'{a}n Donoso and Reem Yassawi for fruitful discussions. This work was supported by the German Research Foundation (DFG) under SFB-TRR 358/1 (2023)--91392403; AB-G also received funding from ANID/FONDECYT Postdoctorado 3230159 (year 2023).  NM would like to thank the Center for Mathematical Modelling (CMM) at Santiago, Chile where part of this work is completed (under a Postdoctoral Visiting Fellowship).

\bibliographystyle{alpha}

\end{document}